\documentclass[leqno, a4paper]{amsart}
\usepackage{amsmath, amssymb, amsthm, mathrsfs, a4wide}
\usepackage[usenames]{color}
\usepackage{eepic,epic}

\newtheorem{theorem}{Theorem}[section]

\newtheorem{corollary}[theorem]{Corollary}
\newtheorem{lemma}[theorem]{Lemma}
\newtheorem{proposition}[theorem]{Proposition}

\theoremstyle{definition}
\newtheorem{definition}[theorem]{Definition}

\theoremstyle{remark}
\newtheorem{remark}[theorem]{Remark}
\newtheorem{example}[theorem]{Example}

\newtheorem*{claim}{Claim}

\numberwithin{equation}{section}

\newcommand{\op}{\operatorname}

\newcommand{\brak}[1]{\langle #1 \rangle}

\newcommand{\C}{\mathbb{C}}
\newcommand{\bH}{\mathbb{H}}
\newcommand{\N}{\mathbb{N}}

\newcommand{\R}{\mathbb{R}}

\newcommand{\Z}{\mathbb{Z}}

\newcommand{\cL}{\mathcal{L}}
\newcommand{\cU}{\mathcal{U}}
\newcommand{\cV}{\mathcal{V}}

\newcommand{\cX}{\mathcal{X}}
\newcommand{\sN}{\mathscr{N}}

\newcommand{\fg}{\mathfrak{g}}

\newcommand{\rk}{\op{rk}}

\newcommand{\ad}{\op{ad}}

\newcommand{\Ow}{\op{O}_{w}}
\newcommand{\DO}{\op{DO}}

\newcommand{\lie}{\chi}

\begin{document}

\title{Privileged Coordinates and Nilpotent Approximation of Carnot Manifolds, I. General Results}

\subjclass[2000]{Primary}

 \author{Woocheol Choi}
\address{Department of Mathematics Education, Incheon National University, Incheon, South Korea}
\email{choiwc@inu.ac.kr}

 \author{Rapha\"el Ponge}
\address{Department of Mathematical Sciences, Seoul National University, Seoul, South Korea}
 \email{ponge.snu@gmail.com}

 \thanks{WC was partially supported by POSCO TJ Park Foundation. RP\ was partially supported by Research Resettlement Fund and Foreign Faculty Research Fund of Seoul National University, and  Basic Research grants 2013R1A1A2008802 and 2016R1D1A1B01015971 of National Research Foundation of Korea.}

\begin{abstract}
In this paper we attempt to give a systematic account on privileged coordinates and nilpotent approximation of Carnot manifolds. By a Carnot 
manifold it is meant a manifold with a distinguished filtration of subbundles of the tangent bundle which is 
compatible with the Lie bracket of vector fields.  This paper lies down the background for its sequel~\cite{CP:AMP17} by clarifying a few points on privileged coordinates and the nilpotent approximation of Carnot manifolds. In particular, we give a description of all the systems of privileged coordinates at a given point. We also give an algebraic characterization of all nilpotent groups that appear as the nilpotent approximation at a given point. In fact, given a nilpotent group $G$ satisfying this algebraic characterization, we exhibit all the changes of variables that transform a given system of privileged coordinates into another system of privileged coordinates in which the nilpotent approximation is given by $G$. 
\end{abstract}

\maketitle

 \section{Introduction}
This paper is part of a series of two papers on privileged coordinates and nilpotent approximation of Carnot manifolds. By a Carnot manifold we mean a manifold $M$ together with a filtration of subbundles,
\begin{equation}
   H_{1}\subset H_{2}\subset \cdots \subset H_{r}=TM,
   \label{eq:Intro.Carnot-filtration}
\end{equation}which is compatible with the Lie bracket of vector fields. 
We refer to Section~\ref{sec:Carnot-Manifolds}, and the references therein, for various examples of Carnot manifolds. Many of those examples are equiregular Carnot-Carath\'eodory manifolds, in which case the Carnot filtration~(\ref{eq:Intro.Carnot-filtration}) arises from the iterated  Lie bracket of sections of $H_1$. However,  even for studying equiregular (and even non-regular) Carnot-Carath\'eodory structures we may be naturally led to consider non bracket-generated Carnot filtrations (see Section~\ref{sec:Carnot-Manifolds} on this point). 

It is a general understanding that (graded) nilpotent Lie group are models for Carnot manifolds. 
From an algebraic perspective, any filtration~(\ref{eq:Intro.Carnot-filtration}) gives rise to a graded vector bundle $\fg M:=\fg_1M\oplus \cdots \oplus \fg_r M$, where $\fg_wM=H_{w}/H_{w-1}$. As a vector bundle $\fg M$ is (locally) isomorphic to the tangent bundle $TM$. Moreover, as observed by Tanaka~\cite{Ta:JMKU70}, the Lie bracket of vector fields induces on each fiber $\fg M(a)$, $a\in M$, a Lie algebra bracket. This turns $\fg M(a)$ into a graded nilpotent Lie algebra. Equipping it with its Dynkin product we obtain a graded nilpotent group $GM(a)$, which is called the tangent group at $a$ (see Section~\ref{sec:Carnot-Manifolds} for a review of this construction). 

There is an alternative construction of nilpotent graded groups associated with $(M,H)$.  This construction originated from the work of Folland-Sein~\cite{FS:CMPAM74}, Rothschild-Stein~\cite{RS:ActaMath76}, and others on hypoelliptic PDEs. In this context, it is natural to weight the differentiation by a vector field according to which sub-bundle $H_j$ of the filtration~(\ref{eq:Intro.Carnot-filtration}) that vector field lies. For instance, as $[H_1,H_1]\subset H_2$ we would like to regard directions in $H_2\setminus H_1$ as having order~$2$. More generally, directions in $H_w\setminus H_{w-1}$ have weight~$w$. Note that this notion of weight is consistent with the grading of $\fg M$ described above. 

In local coordinates centered at a given point $a\in M$ this gives rise to a one-parameter of anisotropic dilations $\delta_t$, $t\in \R$. Rescaling vector fields by means of these dilations and letting $t\rightarrow 0$ we obtain asymptotic expansions whose leading terms form a graded nilpotent Lie algebra of vector fields $\fg^{(a)}$. As it turns out, the algebraic structure of $\fg^{(a)}$ heavily depends on the choice of the coordinates. However, there is a special class of coordinates, called privileged coordinates, where the grading of $\fg^{(a)}$ is compatible with the weight. In addition, in these coordinates the graded nilpotent Lie algebra $\fg^{(a)}$ is isomorphic to $\fg M(a)$. The Lie algebra $\fg^{(a)}$ is actually the Lie algebra of left-invariant vector fields on a graded nilpotent Lie group $G^{(a)}$. This group gives rise to the so-called nilpotent approximation of $(M,H)$ at $a$. We refer to Section~\ref{sec:Privileged} and Section~\ref{sec:Nilpotent-approx}, and the references therein, for more details on privileged coordinates and nilpotent approximation. 

At the conceptual level and for the sake of applications, it is desirable to understand better the relationship between the tangent group and the nilpotent approximation. The ultimate aim of this paper and its sequel~\cite{CP:AMP17} is to single out a special class of privileged coordinates, called Carnot coordinates, for which the nilpotent approximation is naturally given by the tangent group. In particular, these coordinates are an important ingredient in the approach of~\cite{CP:Groupoid} on the generalization of Pansu derivative to maps between general Carnot manifolds and the construction of an analogue for Carnot manifolds of Connes' tangent groupoid. The existence of such a groupoid was conjectured by Bella\"iche~\cite{Be:Tangent}. This provides us with definitive evidence that the tangent groups as described above are the relevant osculating objects of Carnot manifolds. In particular, this gives a conceptual explanation for the occurrence of the group structure.

The Carnot coordinates will be introduced in~\cite{CP:AMP17}. In this paper, we attempt to give a systematic account on privileged coordinates and nilpotent approximation of Carnot manifolds.  In particular, we clarify a few important points on privileged coordinates and nilpotent approximation. In addition,  this lies down the background for~\cite{CP:AMP17}.

There are various constructions of privileged coordinates~\cite{AS:DANSSSR87, BS:SIAMJCO90, Be:Tangent, Go:LNM76, He:SIAMR, RS:ActaMath76, St:1988}. In particular, Bella\"iche~\cite{Be:Tangent} produced a simple and effective construction of privileged coordinates by means of a suitable polynomial change of coordinates. Bella\"iche's construction was carried out in the setting Carnot-Carath\'eodory manifolds.  
There is no major difficulty to extend Bella\"iche's construction to arbitrary Carnot manifolds (see Proposition~\ref{prop:privileged} and Proposition~\ref{prop:privileged.psi-privileged}). We shall refer to these coordinates as the $\psi$-privileged coordinates. We also give a characterization of these coordinates among all privileged coordinates. More precisely, we show that the polynomial change of variables that is used to obtain the  $\psi$-privileged coordinates is the only one of its form that produces privileged coordinates (Proposition~\ref{prop-uni}). 

As alluded to above, the nilpotent approximation arises from anisotropic asymptotic expansions of vector fields in local coordinates. This type of asymptotic expansions has been considered by a number of authors~\cite{ABB:SRG, BG:CHM, Be:Tangent, Go:LNM76, Gr:CC, He:SIAMR, Je:Brief14, MM:JAM00, Me:CPDE76, Mi:JDG85, Mo:AMS02, Ro:INA, RS:ActaMath76} in various levels of generality. In Section~\ref{sec:anisotropic}, we attempt to give a systematic account on anisotropic asymptotic expansions of functions, multi-valued maps, vector fields and differential operators. In particular, we show that the various asymptotics we consider actually hold with respect to the corresponding $C^\infty$-topologies. 

As it turns out, there are two different definitions of privileged coordinates. In this paper, we use the definition of~\cite{BS:SIAMJCO90, Be:Tangent} in terms of the orders of the coordinate functions. The earlier definition of~\cite{Go:LNM76} involves the weights of vector fields. It is well known that privileged coordinates in the 
sense of~\cite{BS:SIAMJCO90, Be:Tangent} are privileged coordinates in the sense of~\cite{Go:LNM76} (see~\cite{BS:SIAMJCO90, Be:Tangent, Je:Brief14, Mo:AMS02}).  We establish the converse result, and so this shows that the two  notions of privileged coordinates are equivalent (Theorem~\ref{thm-pri-equiv}). We use this result to give a simple characterization of the change of coordinates that transform privileged coordinates into privileged coordinates (Proposition~\ref{prop:char-priv.phi-hatphi-Ow}). 

We give a few applications of this characterization result. First, by combining it with the construction of the $\psi$-privileged coordinates, we obtain an explicit description of all the systems of privileged coordinates (Corollary~\ref{cor:privileged.all-privileged-coord}). Another application is a new proof that the canonical coordinates of the 1st kind of Rothschild-Stein~\cite{RS:ActaMath76} and Goodman~\cite{Go:LNM76} are privileged coordinates (Proposition~\ref{prop:can.1stkind-privileged}). By using similar arguments we also give a new proof that the canonical coordinates of the 2nd kind of Bianchini-Stefani~\cite{BS:SIAMJCO90} and Hermes~\cite{He:SIAMR} are privileged coordinates (Proposition~\ref{prop:can.2ndkind-privileged}). The approach presupposes the existence of privileged coordinates (e.g., $\psi$-privileged coordinates), but it bypasses the manipulations on flows of vector fields of previous proofs (compare~\cite{Je:Brief14, Mo:AMS02}). 

As mentioned above, the nilpotent approximation $G^{(a)}$ arises from anisotropic asymptotic expansions of vector fields in privileged coordinates. The graded nilpotent Lie group $G^{(a)}$ has $\R^n$ as underlying manifold, but its group law depends on the choice of the privileged coordinates. Thus, it is natural to ask what nilpotent groups arise as nilpotent approximations at a given point. We give an algebraic characterization of these groups (see Corollary~\ref{cor:nilpotent-approx.characterization}).  In fact, given any nilpotent group $G$ in this class we actually exhibit all the change of coordinates that convert any given system of privileged coordinates into a system of privileged coordinates in which the nilpotent approximation is given by $G$ (Theorem~\ref{thm:Nilpotent.coord-G(a)G}). 

The class of groups that satisfy this algebraic characterization is as large as it can be. In particular, we can exhibit all the nilpotent approximations of Heisenberg manifolds, including contact manifolds and CR manifolds of hypersurface type (see Proposition~\ref{prop:nilpotent-approx.Heisenberg-approx}). We similarly can exhibit all the nilpotent approximation of step~2 Carnot manifolds. For general Carnot manifolds, the construction leads us to a large class of nilpotent approximations at a given point. 

 This paper is organized as follows. In Section~\ref{sec:Carnot-Manifolds}, we review the main definitions and examples regarding Carnot manifolds and their tangent group bundles.  In Section~\ref{sec:Privileged}, we explain how to extend to Carnot manifold Bella\"iche's construction of privileged coordinates. In Section~\ref{sec:anisotropic}, we give a detailed account on the anisotropic asymptotic expansions of functions, maps, vector fields and differential operators. In Section~\ref{sec:charact-privileged}, we give a characterization of privileged coordinates in terms of weight of vector fields. This shows that the two main notions of privileged coordinates are equivalent. In Section~\ref{sec:Nilpotent-approx}, we give a systematic account on the nilpotent approximation of Carnot manifolds. Finally, in Section~\ref{sec:Canonical-coord} we give new proofs that the canonical coordinates of the 1st kind of~\cite{Go:LNM76, RS:ActaMath76} and canonical coordinates of the 2nd kind of~\cite{BS:SIAMJCO90, He:SIAMR} are privileged coordinates.  
  
\subsection*{Acknowledgements}
The authors wish to thank Andrei Agrachev, Davide Barilari, Enrico Le Donne, and Fr\'ed\'eric Jean for useful discussions related to
the subject matter of this paper. They also thank an anonymous referee whose insightful comments help improving the presentation of the paper. In addition, they would like to thank Henri Poincar\'e Institute (Paris, France), McGill University 
(Montr\'eal, Canada) and University of California at Berkeley (Berkeley, USA) for their hospitality during the 
preparation of this paper.  

\section{Carnot Manifolds. Definitions and Examples}\label{sec:Carnot-Manifolds}
In this section, we give the main definitions and examples related to graded nilpotent Lie groups and Carnot manifolds. 

\subsection{Carnot Groups and Graded nilpotent Lie groups} Prior to getting to Carnot manifolds we review some basic facts on Carnot groups and nilpotent graded Lie groups and their Lie algebras. 

\begin{definition} \label{def:Carnot.Carnot-algebra}
A \emph{step $r$ nilpotent graded Lie algebra} is the data of a real Lie algebra $(\fg,[\cdot, \cdot])$ and a grading $  \fg=\fg_{1}\oplus \fg_{2}\oplus \cdots \oplus \fg_{r}$, which is compatible with the Lie bracket, i.e., 
\begin{equation}
        [\fg_{w},\fg_{w'}]\subset \fg_{w+w'} \ 
        \text{for $w+w'\leq r$} \quad \text{and} \quad [\fg_{w},\fg_{w'}]=\{0\}\ 
        \text{for $w+w'> r$}.
         \label{eq:Carnot.grading-bracket}
\end{equation}
We further say that $\fg$ is  \emph{Carnot algebra} when $\fg_{w+1}=[\fg_{1},\fg_{w}]$ for $w=1,\ldots, r-1$. 
\end{definition}

\begin{remark}
 The conditions~(\ref{eq:Carnot.grading-bracket}) automatically imply that $\fg$ is nilpotent of step~$r$. 
\end{remark}

\begin{remark}
    Carnot algebras are also called stratified nilpotent Lie algebras  (see~\cite{Fo:SM79}). 
\end{remark}

\begin{remark}
    Any commutative real Lie algebra $\fg$ (i.e., any real vector space) is a step 1 Carnot algebra with $\fg_{1}=\fg$. 
\end{remark}

\begin{remark}
   A classification of $n$-dimensional Carnot algebras of step $n-1$ was obtained by Vergne~\cite{Ve:BSMF70}. A 
   classification of  rigid Carnot algebras  was obtained by Agrachev-Marigo~\cite{AM:JDCS05}. 
\end{remark}

In what follows, by the Lie algebra of a Lie group we shall mean the tangent space at the unit element equipped with the induced Lie bracket. 

\begin{definition}\label{def:Carnot.Carnot-group}
 A \emph{graded nilpotent Lie group} (resp., \emph{Carnot group}) is a connected simply connected nilpotent real Lie group whose Lie algebra is a  graded nilpotent  Lie algebra (resp., 
 Carnot algebra). 
\end{definition}

\begin{remark}
 We refer to the monographs~\cite{BLU:Springer07, CG:Cambridge90, FR:Birkhauser16} detailed accounts on nilpotent Lie groups and Carnot groups. 
\end{remark}

Let $G$ be a step $r$ graded nilpotent Lie group with unit $e$. Then its Lie algebra $\fg=TG(e)$ is canonically identified with the Lie algebra of left-invariant vector fields on $G$. More precisely, with any $\xi \in \fg$ is associated the unique left-invariant vector field $X_\xi$ on $G$ such that $X_\xi(e)=\xi$. In addition, as $G$ is a connected simply connected nilpotent Lie group, its exponential map is a global diffeomorphism $\exp: \fg\rightarrow G$ (see, e.g., \cite{CG:Cambridge90, FS:Hardy82}). For any $\xi \in G$, the flow $\exp(tX_\xi)$ exists for all times $t\in \R$. We then have
\begin{equation*}
\exp(\xi)=\exp(X_\xi), \qquad \text{where $\exp(X_\xi):=\exp(tX_\xi)(e)_{|t=1}$}. 
\end{equation*}
In addition, the flow $\R\ni t\rightarrow \exp(tX_\xi)$ is a one-parameter subgroup of $G$. Conversely, any one-parameter subgroup of $G$ is generated by a (unique) left-invariant vector field.  

By assumption $\fg$ comes equipped with a grading $  \fg=\fg_{1}\oplus \fg_{2}\oplus \cdots \oplus \fg_{r}$ which is compatible with its Lie algebra bracket. This grading then gives rise to a family of anisotropic dilations $\xi \rightarrow t\cdot \xi$, $t \in \R$, which are linear maps given by
\begin{equation}
 t\cdot (\xi_1+\xi_2 +\cdots + \xi_r) = t\xi_1+ t^2\xi_2 +\cdots + t^r \xi_r, \qquad \xi_j \in \fg_j.  
 \label{eq:Carnot.dilations}
\end{equation}
The compatibility of the grading with the Lie bracket implies that these dilations are Lie algebra automorphisms. Therefore, they give rise to dilations $\delta_t:G\rightarrow G$, $t\in \R$, which are group isomorphisms such that
\begin{equation*}
 \delta_t\left(\exp(\xi)\right)= \exp(t\cdot \xi) \qquad \text{for all $\xi \in \fg$}. 
\end{equation*}

For $\xi \in \fg$, let $\ad_\xi: \fg\rightarrow \fg$ be the adjoint endomorphism associated with $\fg$, i.e., $\ad_\xi \eta =[\xi,\eta]$ for all $\eta \in \fg$. This is a nilpotent endomorphism. In fact, if $\xi \in \fg_w$, then~(\ref{eq:Carnot.grading-bracket}) implies that $\ad_\xi(\fg_{w'})\subset \fg_{w+w'}$ if $w+w'\leq r$, and $\ad_\xi(\fg_{w'})=\{0\}$ otherwise. Moreover, by  the Baker-Campbell-Hausdorff formula we have 
\begin{equation}
\exp(\xi) \exp(\eta)= \exp(\xi \cdot \eta) \qquad \text{for all $\xi,\eta \in \fg$},
\label{eq:Carnot.BCH-Formula}
\end{equation}where $\xi \cdot \eta$ is given by the Dynkin formula, 
\begin{align}
 \xi \cdot \eta & = \sum_{n\geq 1} \frac{(-1)^{n+1}}{n} \sum_{\substack{\alpha, \beta \in \N_0^n\\ \alpha_j+\beta_j\geq 1}} 
 \frac{(|\alpha|+|\beta|)^{-1}}{\alpha!\beta!} (\ad_\xi)^{\alpha_1} (\ad_\eta)^{\beta_1} \cdots (\ad_\xi)^{\alpha_n} (\ad_\eta)^{\beta_n-1}\eta \nonumber \\
 & = \xi +\eta + \frac{1}{2}[\xi,\eta] + \frac{1}{12} \left( [\xi,[\xi,\eta]]+  [\eta,[\eta,\xi]]\right) - \frac{1}{24}[\eta,[\xi,[\xi,\eta]]] + \cdots .
 \label{eq:Carnot.Dynkin-product}
\end{align}
The above summations are finite, since all the iterated brackets of length~$\geq r+1$ are zero. In addition, when $\beta_n=0$ we make the convention that $(\ad_\xi)^{\alpha_n} (\ad_\eta)^{\beta_n-1}\eta=(\ad_\xi)^{\alpha_n-1}\xi$. Thus, with this convention 
$(\ad_\xi)^{\alpha_n} (\ad_\eta)^{\beta_n-1}\eta=0$ when either $\beta_n\geq 2$, or $\beta_n=0$ and $\alpha_n\geq 2$. 

Conversely, if $\fg$ is a graded nilpotent Lie algebra, then~(\ref{eq:Carnot.Dynkin-product}) defines a product on $\fg$. This turns $\fg$ into a Lie group with unit $0$. Under the identification $\fg\simeq T\fg(0)$, the corresponding Lie algebra is naturally identified with $\fg$, and so we obtain a graded nilpotent Lie group. Moreover, under this identification the exponential map becomes the identity  map, and so any Lie algebra automorphism of $\fg$ is a group automorphism as well. In particular, the dilations~(\ref{eq:Carnot.dilations}) are group automorphisms with respect to the group law~(\ref{eq:Carnot.Dynkin-product}). In addition, the group law~(\ref{eq:Carnot.Dynkin-product}) implies that 
\begin{equation*}
\xi^{-1}=-\xi \qquad \text{for all $\xi \in \fg$}. 
\end{equation*}

\subsection{Carnot manifolds} 
In what follows, given a manifold $M$ and distributions $H_{j}\subset TM$, $j=1,2$, we shall denote by $[H_{1},H_{2}]$ 
the distribution generated by the Lie brackets of their sections, i.e., 
\begin{equation*}
  [H_{1},H_{2}]=\bigsqcup_{x\in M}\biggl\{ [X_{1},X_{2}](x); \ X_{j}\in C^{\infty}(M,H_{j}), j=1,2\biggr\}.
\end{equation*}

Throughout this paper we shall use the following definition of a Carnot manifold. 

\begin{definition}
    A \emph{Carnot manifold} is a pair $(M,H)$, where $M$ is a manifold and $H=(H_{1},\ldots,H_{r})$ is a finite filtration of 
    subbundles, 
    \begin{equation}
       H_{1}\subset H_{2}\subset \cdots \subset H_{r}=TM,  
        \label{eq:Carnot.Carnot-filtration}
    \end{equation}which is compatible with the Lie bracket of vector fields, i.e., 
    \begin{equation}
        [H_{w},H_{w'}]\subset H_{w+w'}\qquad 
        \text{for $w+w'\leq r$}.
        \label{eq:Carnot-mflds.bracket-condition}
    \end{equation}
    The number $r$ is called the \emph{step} of the Carnot manifold $(M,H)$. The sequence $(\rk H_{1},\ldots, \rk H_{r})$ is called its \emph{type}.  
\end{definition}

\begin{remark}
    Carnot manifolds are also called filtered manifolds in~\cite{CS:AMS09}. 
\end{remark}

\begin{definition}\label{def:Carnot.Carnot-mfld-map}
    Let $(M,H)$ and $(M',H')$ be Carnot manifolds of step $r$, where $H=(H_{1},\ldots,H_{r})$ and 
    $H'=(H_{1}',\ldots,H_{r}')$. Then
    \begin{enumerate}
        \item  A \emph{Carnot manifold map} $\phi:M\rightarrow M'$ is smooth map such that, for $j=1,\ldots,r$, 
        \begin{equation}
            \phi'(x)X\in H_{j}'(\phi(x)) \qquad \text{for all $(x,X)\in H_{j}$}.
            \label{eq:Carnot.Carnot-map}
        \end{equation}
    
        \item  A \emph{Carnot diffeomorphism} $\phi:M\rightarrow M'$ is a diffeomorphism which is a Carnot manifold map. 
    \end{enumerate}
\end{definition}

\begin{remark}
   If $\phi:M\rightarrow M'$ is a Carnot diffeomorphism, then the condition~(\ref{eq:Carnot.Carnot-map}) exactly means 
  that $\phi_{*}H_{j}=H_{j}'$ for $j=1,\ldots,r$. Therefore, in this case, the Carnot manifold $(M,H)$ and $(M',H')$ must 
   have same type and the inverse map $\phi^{-1}$ is a Carnot diffeomorphism as well.  
\end{remark}

Let $(M^{n},H)$ be an $n$-dimensional Carnot manifold of step $r$, so that $H=(H_{1},\ldots,H_{r})$, 
where the subbundles $H_{j}$ satisfy~(\ref{eq:Carnot.Carnot-filtration}).  

\begin{definition}\label{def:Carnot.weight-sequence}
    The \emph{weight sequence} of a Carnot manifold $(M,H)$ is the sequence $w=(w_{1},\ldots,w_{n})$ defined by
    \begin{equation}
    w_{j}=\min\{w\in \{1,\ldots,r\}; j\leq \rk H_{w}\}.
    \label{eq:Carnot.weight}
\end{equation}
\end{definition}

\begin{remark}
    Two Carnot manifolds have same type if and only if they have same weight sequence. 
\end{remark}

Throughout this paper we will make use of the following type of tangent frames.  

\begin{definition}  
An \emph{$H$-frame} over an open $U\subset M$ is a tangent frame $(X_{1},\ldots,X_{n})$ over $U$ which is compatible with the 
filtration $(H_{1},\ldots.,H_{r})$ in the sense that, for $w=1,\ldots,r$, the vector fields $X_{j}$, $w_{j}= w$, are sections of $H_{w}$.
\end{definition}

\begin{remark}\label{rmk:Carnot-mfld.local-frames}
    If $(X_{1},\ldots,X_{n})$ is an $H$-frame near a point $a\in M$, then, for $w=1,\ldots,r$, the family $\{X_{j}; \ w_{j}\leq w\}$ is a local frame 
of $H_{w}$ near the point $a$. Note this implies that $X_{j}$ is a section of $H_{w_{j}}\setminus H_{w_{j}-1}$ when $w_{j}\geq 2$. 
\end{remark}

\begin{remark}\label{rmk:Carnot-mfld.brackets-H-frame}
    Suppose that $(X_{1},\ldots,X_{n})$ is an $H$-frame near a point $a\in M$. Then the condition~(\ref{eq:Carnot-mflds.bracket-condition}) implies that 
    the vector field $[X_{i},X_{j}]$ is a section of $H_{w_{i}+w_{j}}$ when $w_{i}+w_{j}\leq r$. As $\{X_{k}; w_{k}\leq 
    w_{i}+w_{j}\}$ is a local frame of $H_{w_{i}+w_{j}}$, we deduce that there are unique smooth functions 
    $L_{ij}^{k}(x)$ near $x=0$ such that
    \begin{equation}
        [X_{i},X_{j}]=\sum_{w_{k}\leq w_{i}+w_{j}}L_{ij}^{k}(x)X_{k} \qquad \text{near $x=a$}. 
        \label{eq:Carnot-mfld.brackets-H-frame}
    \end{equation}   
\end{remark}

\subsection{Examples of Carnot Manifolds}\label{sec:Examples}
\renewcommand{\thesubsubsection}{\Alph{subsubsection}}
We shall now describe various examples of Carnot manifolds. The following list is by no means exhaustive, but it 
should give the reader a good glimpse at the vast diversity of examples of Carnot manifolds. 

\subsubsection{Graded Nilpotent Lie Groups}
Let $G$ be a step $r$ graded nilpotent Lie group. Then its Lie algebra $\fg=TG(e)$ has a grading $ \fg=\fg_{1}\oplus \fg_{2}\oplus \cdots \oplus \fg_{r}$
satisfying~(\ref{eq:Carnot.grading-bracket}). For $w=1,\ldots, r$, let $E_w$ be the $G$-subbundle of $TG$ obtained by left-translation of $\fg_w$ over $G$. We then obtain a $G$-vector bundle grading $TG=E_1\oplus \cdots \oplus E_r$. This grading gives rise to the filtration $H_1\subset \cdots \subset H_r=TG$, where $H_w:=H_1 \oplus \cdots H_w$. It can be shown that $[H_w,H_{w'}]\subset H_{w+w'}$ whenever $w+w'\leq r$ (see, e.g., \cite{CP:AMP17}). Therefore, this defines a left-invariant Carnot manifold structure on $G$ which is uniquely determined by the grading of $\fg$.

\subsubsection{Heisenberg Manifolds} In the terminology of~\cite{BG:CHM}, a Heisenberg manifold is 
 a manifold together with a distinguished hyperplane bundle $H\subset TM$. In particular, $(H,TM)$ is a Carnot 
filtration. Following are important examples of Heisenberg manifolds:
\begin{itemize}
    \item  The Heisenberg group $\bH^{2n+1}$ and its products with Abelian groups.

    \item  Cauchy-Riemann (CR) manifolds of hypersurface type, e.g., real hypersurfaces in complex manifolds.

    \item  Contact manifolds and even contact manifolds.

    \item  Confoliations of Eliashberg-Thurston~\cite{ET:C}. 
\end{itemize}

Given a Heisenberg manifold $(M,H)$, taking Lie bracket of horizontal vector fields modulo $H$ defines an antisymmetric 
bilinear vector bundle map $\cL:H\times H\rightarrow TM/H$, which is called the Levi form of $(M,H)$ (see Lemma~\ref{lem-dep}). 
If $M$ has odd dimension, then $H$ is a contact distribution if and only if the Levi form is non-degenerate at every 
point. If $\dim M$ is even, we say that $(M,H)$ is an even contact manifold when the Levi form has maximal rank $\dim 
M-2$ at every point (\emph{cf.}~\cite{Mo:AMS02}). 

\subsubsection{Foliations}
More generally, a Carnot manifold structure of step 2 on a given manifold $M$ reduces to the datum of a subbundle $H\subset TM$, so that $(H, TM)$ is a Carnot filtration. In the same way  as with Heisenberg manifolds, we have a Levi form $\cL:H\times H\rightarrow TM/H$. 

We have a foliation when $H$ is integrable in Fr\"obenius' sense, i.e., $[H,H]\subset H$, or equivalently, the Levi form $\cL$ vanishes at every point. There are numerous examples of foliations such as foliations arising from submersions (including Reeb foliations), those associated with suspensions (including Kr\"onecker foliation), or foliations arising from locally free actions of Lie groups on manifolds (see, e.g., \cite{MM:Cambridge03}).

\subsubsection{Polycontact manifolds} They are step 2 Carnot manifolds that are the total opposite of Foliations. A polycontact manifold is a manifold $M$ equipped with a subbundle $H\subset TM$  which is totally non-integrable. This means that, for all $x\in M$ and $\theta\in (TM/H)^{*}(x)\setminus 0$, the bilinear form $\theta \circ \cL_{x}:H(x)\times H(x)\rightarrow \R$ is non-degenerate. Following are some noteworthy examples of polycontact manifolds: 
    \begin{itemize}
        \item  M\'etivier groups~\cite{Me:Duke80}, including the $H$-type groups of Kaplan~\cite{Ka:TAMS80}.
    
        \item  Principal bundles equipped with the horizontal distribution defined by a fat connection 
        (\emph{cf.}~Weinstein~\cite{We:Adv80}). 
    
        \item   The quaternionic contact manifolds of Biquard~\cite{Bi:Roma99, Bi:Asterisque}. 
    
        \item  The unit sphere $\mathbb{S}^{4n-1}$ in quaternionic space (see~\cite{vE:Polycontact}). 
    \end{itemize}

\begin{remark}
     Polycontact manifolds are called Heisenberg manifolds in~\cite{CC:SRG}. Polycontact distributions are called fat 
    in~\cite{Mo:AMS02}. The terminology polycontact was coined by van Erp~\cite{vE:Polycontact} (following a suggestion 
    of Alan Weinstein).
\end{remark}
    
 \subsubsection{Pluri-contact manifolds} These manifolds are introduced in the recent preprint~\cite{ACGL:Contact17}{\footnote{The terminology ``pluri-contact manifold'' is not used in~\cite{ACGL:Contact17}. We use it for the sake of exposition's clarity.}}. 
 They generalize polycontact manifolds. A pluri-contact manifold is a manifold $M$ equipped with a subbundle $H\subset TM$ such that, for every $x\in M$,  there is at least one $\theta\in (TM/H)^{*}(x)\setminus 0$ such that the bilinear form $\theta \circ \cL_{x}:H(x)\times H(x)\rightarrow \R$ is non-degenerate. 
Beside polycontact manifolds are examples of pluri-contact manifolds, examples of pluri-contact manifolds  that are not polycontact include
 \begin{itemize}
\item Products of contact manifolds~\cite{ACGL:Contact17}.

\item  Nondegenerate CR manifolds of non-hypersurface-type~\cite{ACGL:CR17}. 
\end{itemize}
We refer to~\cite{ACGL:Contact17, ACGL:CR17} (and the references therein) for further examples of pluri-contact manifolds.

\subsubsection{Carnot-Carath\'eodory Manifolds}
Given any subbundle, or even distribution, $H\subset TM$, we recursively define distributions $H^{[j]}$, $j\geq 1$, by
\begin{equation*}
    H^{[1]}=H \qquad \text{and} \qquad H^{[j+1]}=H^{[j]}+\left[H,H^{[j]}\right], \ j \geq 1. 
\end{equation*}

\begin{definition} A \emph{Carnot-Carath\'eodory} manifold  (or CC manifold) is a pair $(M,H)$, where $M$ is a manifold and $H$ is a subbundle of 
    $TM$ such that, for all $x\in M$, we can find $r\in \N$ so that $H^{[r]}(a)=TM(x)$. 
\end{definition}

\begin{definition}[\cite{Gr:CC}]
    An \emph{equiregular Carnot-Carath\'eodory manifold} (or ECC manifold) is a Carnot-Carath\'eodory
    manifold $(M,H)$ for which there is $r\in \N$ such that $H^{[r]}=TM$ and 
    each distribution $H^{[j]}$, $j=2,\ldots,r-1$, has constant rank.
\end{definition}

Any equiregular Carnot-Carath\'eodory manifold $(M,H)$ is a Carnot manifold with Carnot filtration 
$\left(H^{[1]},\ldots H^{[r]}\right)$, since each distribution $H^{[j]}$ is a subbundle. Moreover, any non-equiregular Carnot-Carath\'eodory structure can be 
"desingularized" into an equiregular Carnot-Carath\'eodory structure (see, e.g., \cite{Je:Brief14}).

If $G$ is a Carnot group, then the Carnot structure described above is actually an EEC structure. 
More generally, Carnot-Carath\'eodory structures naturally appear in sub-Riemannian geometry and control theory associated with 
non-holonomic systems of vector fields. In particular, they occur in real-life situations such as skating motion~\cite{Be:Tangent, 
Bl:Control}, rolling penny~\cite{CC:SRG}, 
car-like robots~\cite{CC:SRG}, or car with $N$-trailers~\cite{Je:ESAIM96}. They also appear in numerous applied mathematics settings 
such as the Asian option model in finance, computer vision, image 
processing, dispersive groundwater and pollution, statistical properties of laser light, extinction in systems of 
interacting biological populations, dynamics of polymers and astronomy (distribution of clusters in space) 
(see~\cite{Br:Brief14} and the references therein). 

We refer to~\cite{ABB:SRG, CC:SRG, Gr:CC, Je:Brief14, Mo:AMS02, Ri:Brief14}, and the references therein, for detailed accounts on 
Carnot-Carath\'eodory manifolds and their various applications. 

Contact and polycontact manifolds are examples of ECC manifolds with $r=2$. An important 
example of ECC manifolds with $r= 3$ is provided by Engel manifolds, which are  4-dimensional manifold equipped with a plane-bundle 
$H\subset TM$ such that $H^{[2]}$ has constant rank 3 and $H^{[3]}=TM$ (see~\cite{Mo:AMS02}). More generally, parabolic geometry provides us with a wealth of examples of ECC manifolds. 
We refer to~\cite{CS:AMS09} for a thorough account on parabolic geometric structures. In addition to the examples above, further examples include the following:
\begin{itemize}
    \item  Nondegenerate (partially) integrable CR manifolds of hypersurface type (see~\cite{CS:AMS09}).  
    
    \item Contact path geometric structures, including contact projective structures (\emph{cf}.~\cite{Fo:IUMJ05, 
    Fo:arXiv05}). 

    \item  Contact quaternionic structures of Biquard~\cite{Bi:Roma99, Bi:Asterisque}. 

    \item  Generic $\left(m,\frac{1}{2}m(m+1)\right)$-distributions, including generic $(3,6)$-distributions studied by 
    Bryant~\cite{Br:RIMS06} (see~\cite{CS:AMS09}). 
    
    \item Generic $(2,3,5)$-distributions introduced by Cartan~\cite{Ca:AENS10}. 
\end{itemize}

The next two examples provide are examples of Carnot manifolds that are not ECC manifolds. 

\subsubsection{The heat equation on an ECC manifold.} Let $(M,H)$ be an ECC manifold, so that the filtration 
$H=(H_{1},\ldots,H_{r})$ is generated by the iterative Lie brackets of sections of $H_{1}$. Let $X_{1},\ldots,X_{m}$ be 
a spanning frame of $H_{1}$. Then the sub-Laplacian $\Delta=-(X_{1}^{2}+\cdots X_{m}^{2})$ and the associated heat 
operator $\Delta+\partial_{t}$ are hypoelliptic (see~\cite{Ho:ActaMath67}). Following Rothschild-Stein~\cite{RS:ActaMath76} (and other authors) in order to study the heat operator $\Delta+\partial_{t}$ we seek for a setup where 
differentiation with respect to time has degree $2$ on $M\times \R$. This amount to use the filtration 
$\tilde{H}=(\tilde{H}_{1},\ldots,\tilde{H}_{r})$,  where
\begin{equation*}
  \tilde{H}_{1}=\pi_{1}^{*}H_1, \qquad \tilde{H}_{j}=\pi_{1}^{*}H_{j}+\pi_{2}^{*}T\R, \quad  j\geq 2. 
\end{equation*}
Here $\pi_{1}$ (resp., $\pi_{2}$) is the projection of $M\times \R$ onto $M$ (resp., $\R$). This gives rise to a Carnot manifold structure. However, this is not
an  ECC structure since $\tilde{H}_{1}$ is not bracket-generating (its Lie brackets only span $\pi_1^*TM$). 

\subsubsection{Group actions on contact manifolds} 
Let $(M,H)$ be an orientable contact manifold and $G$ its group of contactomorphisms, i.e., diffeomorphisms preserving 
the contact distribution $H$. The group $G$ is essential (see, e.g., \cite{Ba:JGeom00}). In the framework of 
noncommutative geometry~\cite{Co:NCG}, such a situation is typically handled by constructing a 
spectral triple (see~\cite{CM:GAFA95}). In the case of an arbitrary group of diffeomorphisms this is done by passing to 
the total space of the metric bundle, i.e., the $\R_{+}^{*}$-subbundle of $T^{*}M\otimes T^{*}M$ whose fibers consist of positive-definite 
symmetric $(0,2)$-tensors $g_{ij}dx^{i}\otimes dx^{j}$ (see~\cite{CM:GAFA95}). 
In the setting of contact geometry the relevant metric bundle is the bundle of contact metrics, i.e., metrics of the form,
\begin{equation*}
    g=g_{H}\oplus \theta\otimes \theta,
\end{equation*}where $g_{H}$ is a positive-definite metric on $H$ and $\theta$ is a contact form (i.e., a non-zero 
section of $(TM/H)^{*}$). Fixing a contact form $\theta$ on $M$, this bundle can be realized as
\begin{equation*}
    P=P(H)\oplus \R_{+}^{*}\theta \stackrel{\pi}{\longrightarrow} M,
\end{equation*}where $P(H)$ is the metric bundle of $H$.  

It is expected that the differential operators appearing in the construction of the spectral triple associated with the 
action of contactomorphisms would have an anisotropic notion of degree such that
\begin{itemize}
    \item[-] Differentiation along the vertical bundle $V=\ker d\pi\subset TP$ has degree~1. 

    \item[-] Differentiation along the lift to $P$ of the contact distribution has degree~2. 

    \item[-] If we denote by $T$ the Reeb vector field of $\theta$, then its lift to $P$ has degree $4$. 
\end{itemize}
This leads us to consider the filtration $\tilde{H}=(\tilde{H}_{1},\tilde{H}_{2},\tilde{H}_{3},\tilde{H}_{4})$  of $TP$ given by
\begin{equation*}
    \tilde{H}_{1}=V, \qquad \tilde{H}_{2}=\tilde{H}_{3}=V+\pi^{*}H, \qquad \tilde{H}_{4}=TP. 
\end{equation*}
This is definitely not  a Carnot filtration associated with an ECC structure, since the vertical bundle $V=\tilde{H}_{1}$ is 
integrable. However, we do get a Carnot filtration. 

This example can be generalized to more general ECC manifolds. In any case, we obtain a Carnot filtration whose first term 
is the vertical bundle of some fibration, and, hence, is not integrable. Therefore, although we get a Carnot 
structure, we never get an ECC structure. 

\begin{remark}
 The last two examples above show that, even if our main focus is on ECC manifolds, we may be naturally led to 
    consider non-ECC Carnot structures. This is the main motivation for considering Carnot manifolds instead of
    sticking to the setup of ECC manifolds. 
\end{remark}

\subsection{The Tangent Group Bundle of a Carnot Manifold}\label{sec:tangent-group}
The constructions of the tangent Lie algebra bundle and tangent group bundle of a Carnot manifold go back to Tanaka~\cite{Ta:JMKU70} (see also~\cite{AM:ERA03, AM:JDCS05, CS:AMS09, GV:JGP88, MM:JAM00, Me:Preprint82, Mo:AMS02, VG:JSM92}).  We refer to~\cite{ABB:SRG, FJ:JAM03, MM:JAM00, Ro:INA} for alternative intrinsic constructions of the tangent Lie algebra bundle and tangent group bundle. 

The Carnot filtration $H=(H_{1},\ldots,H_{r})$ has a natural grading defined as follows. For $w=1,\ldots,r$, set 
$\fg_{w}M=H_{w}/H_{w-1}$ (with the convention that $H_{0}=\{0\}$), and define
\begin{equation}\label{eq-grading}
    \fg M := \fg_{1}M\oplus \cdots \oplus \fg_{r}M.
\end{equation}

\begin{lemma}[\cite{Ta:JMKU70}]\label{lem-dep}
 The Lie bracket of vector fields induces smooth bilinear bundle maps, 
\begin{equation*}
    \cL_{w_1,w_2}:\fg_{w_1}M\times \fg_{w_2}M\longrightarrow \fg_{w_1+w_2}M, \qquad w_1+w_2\leq r. 
\end{equation*}More precisely, given any $a\in M$ and 
 sections $X_j$, $j=1,2$, of $H_{w_j}$ near $a$, if we let $\xi_j(a)$  be the class of $X_j(a)$ in $\fg_{w_j}M(a)$, then we have
\begin{equation*}
 \cL_{w_1,w_2}\left(\xi_1(a),\xi_2(a)\right)= \textup{class of  $[X_1,X_2](a)$ in $\fg_{w_1+w_2}M(a)$}. 
  \label{eq:Tangent.Lie-backet}
\end{equation*}
 \end{lemma}

\begin{definition}
The bilinear bundle map $[\cdot,\cdot]:\fg M\times \fg M\rightarrow \fg M$ is defined as follows. For $a\in M$ and
$\xi_{j}\in \fg_{w_{j}}M(a)$, $j=1,2$, we set
    \begin{equation}\label{eq-brak-a}
        [\xi_{1},\xi_{2}](a)=\left\{
        \begin{array}{ll}
            \cL_{w_{1},w_{2}}(a)(\xi_{1},\xi_{2})& \text{if $w_{1}+w_{2}\leq r$},  \smallskip \\
           0 &  \text{if $w_{1}+w_{2}> r.$ }
        \end{array}\right.
    \end{equation}
\end{definition}

It follows from Lemma~\ref{lem-dep} that $[\cdot, \cdot]$ is a smooth bilinear bundle map. On every fiber  $\fg M(a)$, $a\in M$, it defines 
a Lie algebra bracket such that
       \begin{align}
                   \left[ \fg_{w}M, \fg_{w'}M \right]  \subset \fg_{w+w'}M   & \quad \text{if $w+w'\leq r$}, \\
                \left[\fg_{w}M,\fg_{w'}M\right] =\{0\}  &  \quad \text{if $w+w'>r $}.
          \label{eq:Carnot.Carnot-grading}
           \end{align}
This shows that the Lie bracket is compatible with the grading~(\ref{eq-grading}). It also follows from this that $\fg M(a)$ is a nilpotent of step $r$.  
Therefore, we  arrive at the following result.

\begin{proposition}[\cite{Ta:JMKU70}]
   $(\fg M, [\cdot, \cdot])$ is a smooth bundle of step $r$ graded nilpotent Lie algebras.
\end{proposition}

\begin{definition}
The Lie algebra bundle $(\fg M, [\cdot, \cdot])$  is called the \emph{tangent Lie algebra bundle} of $(M,H)$.
\end{definition}

\begin{remark}
    In~\cite{GV:JGP88, Mo:AMS02, VG:JSM92} the tangent Lie algebra bundle $\fg M$ is called the nilpotenization of 
    $(M,H)$. 
\end{remark}

\begin{remark}\label{rem-xixj}
    Let $(X_{1},\ldots,X_{n})$ be an $H$-frame near a point $a\in M$. As mentioned above, this gives rise to a 
    frame $(\xi_{1},\ldots,\xi_{n})$ of $\fg M$ near $x=a$, where $\xi_{j}$ is the class of $X_{j}$ in 
    $\fg_{w_{j}}M$. Moroever, it follows from~(\ref{eq:Carnot-mfld.brackets-H-frame}) and~(\ref{eq-brak-a}) that, near $x=a$, we have
    \begin{equation}\label{eq-jk}
        [\xi_{i},\xi_{j}]=\left\{
        \begin{array}{cl}
          {\displaystyle \sum_{w_{i}+w_{j}=w_k}L_{ij}^{k}(x)\xi_{k}}    & \text{if $w_{i}+w_{j}\leq r$},  \\
            0  & \text{if $w_{i}+w_{j}>r$},
        \end{array}\right.
    \end{equation}where the functions $L_{ij}^{k}(x)$ are given by~(\ref{eq:Carnot-mfld.brackets-H-frame}). Specializing this to $x=a$ provides us with the structure constants 
    of $\fg M(a)$ with respect to the basis $(\xi_{1}(a),\ldots,\xi_{n}(a))$.  
\end{remark}

\begin{remark}\label{rem:GM.ECC-Carnot-algebra}
 When $(M,H)$ is an ECC manifold, each fiber $\fg M(a)$, $a\in M$, is  a Carnot algebra in the sense of Definition~\ref{def:Carnot.Carnot-algebra} (see, e.g.,~\cite{CP:AMP17}). 
\end{remark}
  
The Lie algebra bundle $\fg M$ gives rise to a Lie group bundle $GM$ as follows. As a manifold we take $GM$ to be $\fg M$ and we equip the fibers $GM(a)=\fg M(a)$ with the Dynkin product~(\ref{eq:Carnot.Dynkin-product}).This turns $GM(a)$ into a step~$r$ graded nilpotent Lie group with unit $0$ whose Lie algebra is naturally isomorphic to $\fg M(a)$. The fiberwise product on $GM$ is smooth, and so we arrive at the following statement. 

\begin{proposition}
    $GM$ is a smooth bundle of step $r$ graded nilpotent Lie groups. 
\end{proposition}

\begin{definition}\label{def:Tangent.Tangent-group}
    $GM$ is called the \emph{tangent group bundle} of $(M,H)$.
\end{definition}

\begin{example}
 Suppose that $r=1$ so that $H_1=H_r=TM$. In this case, as a Lie algebra bundle, $\fg M=TM$. Therefore, as a Lie group bundle, $GM=TM$.  
\end{example}

\begin{example}
Let $G$ be a graded nilpotent Lie group with Lie algebra $\fg=\fg_1\oplus \cdots \oplus \fg_r$. We equip $G$ with the left-invariant Carnot manifold structure defined by the grading of $\fg$. Then the left-regular actions of $G$ on itself and on $\fg$ give rise to canonical identifications $\fg G\simeq G\times \fg$ and $GG\simeq G\times G$ (see~\cite{CP:AMP17}). 
\end{example}

\begin{remark}
 We refer to~\cite{CP:AMP17, Po:PJM06} for an explicit description of the tangent group bundle of a Heisenberg manifold. 
\end{remark}

\begin{remark}
When $(M,H)$ is an ECC manifold, it follows from Remark~\ref{rem:GM.ECC-Carnot-algebra} that every tangent group $GM(a)$ is a 
Carnot group in the sense of Definition~\ref{def:Carnot.Carnot-group}.
\end{remark}

\begin{remark}\label{rmk:tangent.dilations-group-automorphisms}
 In the same way as above, the grading~(\ref{eq-grading}) gives rise to a smooth family of dilations $\delta_t$, $t\in \R$. 
 This gives rise to Lie algebra automorphisms on the fibers of $\fg M$ and group automorphisms on the fibers of $GM$. In addition, the inversion on the fibers of $GM$ is just  the symmetry $\xi \rightarrow -\xi$. 
 \end{remark}

\section{Privileged Coordinates for Carnot Manifolds}\label{sec:Privileged}
In this section, we explain how to adapt to the setting of Carnot manifolds the construction of polynomial privileged coordinates
of Agrachev-Sarychev~\cite{AS:DANSSSR87} and  Bella\"\i che~\cite{Be:Tangent} for Carnot-Carath\'eodory manifolds. Our 
approach follows the approach of~\cite{Be:Tangent}, but there are two main differences. The first 
difference concerns the uniqueness of the construction which we will need later and is not stated in~\cite{Be:Tangent}. 
The second main difference concerns the definition of the order of a function (see Definition~\ref{orderD}), since we cannot use the same definition 
as in~\cite{Be:Tangent} in the setting of general Carnot manifolds (\emph{cf}.~Remark~\ref{rmk:Privileged.order-difference}).  

We refer to~\cite{ABB:SRG, Go:LNM76, Gr:CC, He:SIAMR, MM:JAM00, Me:CPDE76, Mo:AMS02, RS:ActaMath76} for alternative constructions of privileged coordinates. 
 
Throughout the remainder of the paper, we let $(M,H)$ be an $n$-dimensional Carnot manifold of step $r$, so that 
$H=(H_{1},\ldots,H_{r})$ is a filtration of subbundles satisfying~(\ref{eq:Carnot.Carnot-filtration}). In addition, we let 
$w=(w_{1},\ldots,w_{n})$ be the weight sequence of $(M,H)$ (\emph{cf.}\ Definition~\ref{def:Carnot.weight-sequence}). 

In this section, we shall work in local coordinates near a point $a\in M$ around which there is an $H$-frame 
$X_{1},\ldots,X_{n}$ of $TM$. In addition, given any
finite sequence $I=(i_{1},\ldots,i_{k})$ with values in $\{1,\ldots,n\}$, we define
\begin{equation*}
    X_{I}=X_{i_{1}}\cdots X_{i_{k}}.
\end{equation*}
For such a sequence we also set $|I|=k$ and $\brak I=w_{i_{1}}+\cdots + w_{i_{k}}$.

\begin{definition}\label{orderD} Let $f(x)$ be a smooth function defined near $x=a$ and $N$ a non-negative integer.
 \begin{enumerate}
     \item  We say that $f(x)$ has order $\geq N$ at $a$ when $X_{I}f(a)=0$ whenever $\brak I<N$.

     \item  We say that $f(x)$ has order $N$ at $a$ when it has order~$\geq N$ and there is a sequence
     $I=(i_{1},\ldots,i_{k})$ with values in $\{1,\ldots,n\}$ with $\brak I=N$ such that $X_{I}f(a)\neq 0$.
 \end{enumerate}
\end{definition}

\begin{remark}\label{rmk:Privileged.order-difference}
    The above definition of the order of a function differs from that of Bella\"{\i}che~\cite[Definition~4.12]{Be:Tangent}, 
    since Bella\"{\i}che only considers monomials in vector fields $X_{i}$ with $w_{i}=1$. This definition was used in 
    the case of CC manifolds. We cannot use the same definition for general Carnot manifolds since we would be missing 
    directions that are not obtained as commutators of the weight 1 vector fields $X_{i}$, $w_{i}=1$. Nevertheless, in 
    the ECC case the two definitions are equivalent. This is seen by using Lemma~\ref{lem:privileged.properties-order} below and a $H$-frame 
    $(X_{1},\ldots,X_{n})$ built out of iterated commutators of the weight 1 vector fields $X_{i}$, $w_{i}=1$. (This is 
    precisely the type of $H$-frames considered in~\cite{Be:Tangent}). 
\end{remark}

\begin{lemma}\label{lem:privileged.properties-order} Let $f(x)$ be a smooth function near $x=a$. Then its order is
    independent of the choice of the $H$-frame $(X_{1},\ldots,X_{n})$ near $a$.
\end{lemma}
\begin{proof} Let $(Y_{1},\ldots,Y_{n})$ be another $H$-frame near $a$.  We note that each vector field $Y_{i}$ is a
    section of $H_{i}$. Therefore, near $x=a$, we have
    \begin{equation*}
        Y_{i}=\sum_{w_{j}\leq w_{i}}c_{ij}(x)X_{j},
    \end{equation*}where the coefficients $c_{ij}(x)$ are smooth and there is an integer $j$ with $w_{j}=w_{i}$ in such that $c_{ij}(a)\neq 0$.  
    More generally, given any finite sequence $I=(i_{1},\ldots,i_{k})$ with values in $\{1,\ldots,n\}$, near $x=a$
    we have
    \begin{equation}\label{eq-YX}
        Y_{I}=Y_{i_{1}}\cdots Y_{i_{k}} = \biggl( \sum_{w_{j_{1}}\leq w_{i_{1}}}c_{i_{1}j_{1}}(x)X_{j_1}\biggr) \cdots \biggl(
        \sum_{w_{j_{k}}\leq w_{i_{k}}}c_{i_kj_k}(x)X_{j_k}\biggr)
        = \sum_{\brak{J}\leq \brak{I}}c_{IJ}(x)X_{J},
    \end{equation}where the coefficients $c_{IJ}(x)$ are smooth.

    Let $N$ be the order of $f$ with respect to the $H$-frame $(X_{1}, \ldots,X_{n})$. If $\brak I<N$, then~(\ref{eq-YX})
    shows that $Y_{I}f(a)$ is  a linear combination of terms $X_{J}f(a)$ with $\brak J\leq \brak I<N$, which are zero.
    Thus $Y_{I}f(a)=0$ whenever $\brak I<N$. Suppose now that $I$ is such that $\brak I=N$ and
    $X_{I}f(a)\neq 0$. In the same way as in~(\ref{eq-YX}), near $x=a$, we have
    \begin{equation*}
        X_{I}=\sum_{\brak J \leq \brak I}d_{IJ}(x) Y_{J},
    \end{equation*}where the coefficients $d_{IJ}(x)$ are smooth. Then, we have
   \begin{equation*}
        0\neq X_{I}f(a)= \sum_{\brak J \leq \brak I}d_{IJ}(a) Y_{J}f(a)=  \sum_{\brak J =N}d_{IJ}(a) Y_{J}f(a).
    \end{equation*}Therefore, at least one of the numbers $Y_{J}f(a)$, $\brak J=N$, must be non-zero. We then deduce that $f$
    also has order $N$ at $a$ with respect to the $H$-frame $(Y_{1},\ldots,Y_{n})$. This shows that the order of $f$ at $a$
    is independent of the choice of the $H$-frame. The lemma is thus proved.
\end{proof}

\begin{lemma}\label{lem:privileged.properties-order-product} Let $f(x)$ and $g(x)$ be smooth functions near $x=a$ of respective
    orders $N$ and $N'$ at $a$. Then $f(x)g(x)$ has order~$\geq N+N'$ at $a$.
\end{lemma}
\begin{proof}
We know that $X_{i}(fg)=(X_{i}f)g+fX_{i}g$. More generally, given any sequence $I=(i_{1},\ldots,i_{k})$, we may write
 \begin{equation}\label{eq-I}
     X_{I}(fg)=X_{i_{1}}\cdots X_{i_{k}}(fg)= \sum_{\brak{I'}+\brak{I''}=\brak{I}}c_{I'J''}(X_{I'}f)(X_{I''}g),
 \end{equation}for some constants $c_{IJ}$ independent of $f$ and $g$. If $\brak{I'}+\brak{I''}<N+N'$, then at least one of the
 inequalities $\brak{I'}<N$ or $\brak{I''}<N'$ holds. In any case the product $(X_{I'}f)(a)(X_{I''}g)(a)$ is zero.
 Combining this with~(\ref{eq-I}) we then see that $X_{I}(fg)(a)=0$ whenever $\brak{I}<N+N'$. That is, $f(x)g(x)$ has
 order~$\geq N+N'$ at $a$. The proof is complete.
\end{proof}

Given any multi-order $\alpha\in \N_{0}^{n}$, we set
\begin{equation*}
|\alpha|=\alpha_1 +\cdots + \alpha_n \quad \textrm{and}\quad \langle \alpha \rangle = w_1 \alpha_1 + \cdots + w_n \alpha_n.
\end{equation*}
In addition, we define
\[ 
X^{\alpha}:=X_{1}^{\alpha_{1}}\cdots X_{n}^{\alpha_{n}}.
\]
We note that $X^{\alpha}=X_{I}$, where $I=(i_{1},\ldots,i_{k})$ is the unique non-decreasing sequence of length $k=|\alpha|$
where each index $i$  appears with multiplicity $\alpha_{i}$. Conversely, if $I=(i_{1},\ldots,i_{k})$ is a
non-decreasing sequence, then $X_{I}=X^{\alpha}$ for some multi-order $\alpha$ with $|\alpha|=|I|$ and $\brak \alpha=\brak I$.

It is convenient to  reformulate the definition of the order of a function in terms of the sole monomials
$X^{\alpha}$. To reach this end we need the following lemma.

\begin{lemma}[{\cite[Lemma 2.1]{NSW:ActaM85}}; see also {\cite[Lemma~4.12(i)]{Be:Tangent}}]\label{lem-X-I}
   Let $I=(i_{1},\ldots,i_{k})$ be a finite sequence with values in $\{1,\ldots,n\}$ and set $w=\brak I$. Then, near
   $x=a$, we have
   \begin{equation}\label{eq-X-I}
       X_{I}=\sum_{\substack{\brak\alpha\leq w\\ |\alpha|\leq k}} c_{I\alpha}(x)X^{\alpha},
   \end{equation}where the  $c_{I\alpha}(x)$ are smooth functions near $x=a$.
\end{lemma}

Granted this lemma we shall obtain the following characterization of the order of a function. 
\begin{proposition}\label{TFAE}
  Let $f(x)$ be a smooth function defined near $x=a$.   Then $f(x)$ has order $N$ at $x=a$ if and only if the following
  two conditions are satisfied:
    \begin{enumerate}
        \item[(i)] $(X^{\alpha}f)(a)=0$ for all multi-orders $\alpha$ such that $\brak\alpha<N$.

        \item[(ii)] $(X^{\alpha}f)(a)\neq 0$ for at least one multi-order $\alpha$ with $\brak\alpha=N$.
    \end{enumerate}
\end{proposition}
\begin{proof}
 Suppose that (i) and (ii) are satisfied. Then (ii) implies that $f(x)$ has order~$\leq  N$ at $x=a$. Moreover, using (i) and
   Lemma~\ref{lem-X-I} shows that $f(x)$ has order~$\geq N$ at $x=a$. Thus $f(x)$ has order $N$ at $x=a$. 
   
   Conversely, assume that $f(x)$ has order $N$ at $x=a$. It is immediate that (i) holds. Let $I=(i_{1},\ldots,i_{k})$ be a sequence
 with values in $\{1,\ldots,n\}$ with $\brak I=N$ and $X_{I}f(a)\neq 0$. By Lemma~\ref{lem-X-I}, near $x=a$, we have
 \begin{equation*}
     X_{I}=\sum_{\brak \alpha\leq \brak I}c_{I\alpha}(x)X^{\alpha}=\sum_{\brak \alpha\leq N}c_{I\alpha}(x)X^{\alpha}
 \end{equation*}for some smooth coefficients $c_{I\alpha}(x)$. Thus,
 \begin{equation*}
     0\neq X_{I}f(a)= \sum_{\brak \alpha\leq N}c_{I\alpha}(a)X^{\alpha}f(a)=  \sum_{\brak
     \alpha=N}c_{I\alpha}(a)X^{\alpha}f(a).
 \end{equation*}This implies that at least one of the numbers $X^{\alpha}f(a)$, $\brak \alpha=N$,  is non-zero, i.e.,
 (ii) is satisfied. The proof is complete.
\end{proof}

\begin{remark}
    In the ECC case and for the type of $H$-frame mentioned in Remark~\ref{rmk:Privileged.order-difference}, we recover the characterization of the order of a function provided 
    by Bella\"iche~\cite[Lemma~4.12(ii)]{Be:Tangent} (see also~\cite[Lemma~B.4]{Je:Brief14}). 
\end{remark}

\begin{definition}\label{def-lin-adp}
   We say that local coordinates $\{x_{1},\ldots,x_{n}\}$ centered at a point $a\in M$ are \emph{linearly adapted} at $a$ to
   the $H$-frame $X_{1},\ldots,X_{n}$ when $X_{j}({0})=\partial_{j}$ for $j=1,\ldots,n$.
\end{definition}

\begin{lemma}\label{lem-affine}
    Given local coordinates $x=(x_{1},\cdots,x_{n})$, there is a unique affine change of coordinates $x\rightarrow T_{a}(x)$
   which provides us with local coordinates centered at $a$ that are linearly adapted to the $H$-frame $(X_{1},\cdots X_{n})$.
\end{lemma}
\begin{proof}
In the local coordinates $(x_{1},\ldots,x_{n})$ we have
\begin{equation*}
    X_{j}=\sum_{1\leq k \leq n}b_{jk}(x)\partial_{x_{k}}, \qquad j=1,\ldots,n,
\end{equation*}where the coefficients $b_{jk}(x)$ are smooth. Set $B(x)=(b_{jk})_{1\leq j,k\leq n}\in \op{GL}_{n}(\R)$.
In what follows we shall use the same notation for the point $a$ and its coordinate vector $a=(a_{1},\ldots,a_{n})$
with respect to the local coordinates $(x_{1},\ldots,x_{n})$.

Let $T(x)=A(x-a)$ be an affine transformation with $T(a)=0$ and $A=(a_{kl})\in \op{GL}_{n}(\R)$. Set
$y=(y_{1},\ldots,y_{n})=T(x)$, i.e., $y_{k}=\sum_{l}a_{kl}(x_{l}-a_{l})$, $k=1,\ldots,n$. Then $(y_{1},\ldots,y_{n})$
are local coordinates centered at $a$. In those coordinates, for $j=1,\ldots, n$, we have
\begin{equation}\label{eq-Xi-form}
    X_{j}=\sum_{1\leq k,l \leq n}b_{jl}(x)\frac{\partial y_{k}}{\partial x_{l}}\partial_{y_{k}}
    =\sum_{1\leq k \leq n}\biggl( \sum_{1\leq l \leq n} b_{jl}\circ T^{-1}(y) a_{kl}\biggr)\partial_{y_{k}}.
\end{equation}
Thus $X_{j}=\frac{\partial }{\partial y_{j}}$ at $y=0$ if and only if  $\sum_{1\leq l \leq n} b_{jl}(a)
a_{kl}=\delta_{jk}$. We then see that the local coordinates $(y_{1},\ldots,y_{n})$ are linearly adapted at $a$ if and
only if $B(a)A^{t}=1$, i.e., $A=\left( B(a)^{t}\right)^{-1}$. This shows that $T_{a}(x)= \left( B(a)^{t}\right)^{-1}(x-a)$
is the unique affine isomorphism that produces linearly adapted coordinates centered at $a$. The proof is
complete.
\end{proof}

\begin{definition}\label{def:Privileged-coordinates}
    We say that local coordinates $x=(x_{1},\ldots,x_{n})$ centered at $a$ are \emph{privileged coordinates} at $a$ adapted to
    the $H$-frame ($X_{1},\ldots, X_{n})$ when the following two conditions are satisfied:
    \begin{enumerate}
        \item[(i)] These coordinates are linearly adapted at $a$ to the $H$-frame $(X_{1},\ldots ,X_{n})$. 

        \item[(ii)] For all $k=1,\ldots,n$, the coordinate function $x_{k}$ has order $w_{k}$ at $a$.
    \end{enumerate}
\end{definition}

\begin{remark}
 Privileged coordinates are called adapted coordinates in~\cite{BS:SIAMJCO90}. 
\end{remark}

\begin{remark}
  As mentioned in~Remark~\ref{rmk:Privileged.order-difference},  in the ECC case our notion of order of a function agrees with that of Bella\"iche~\cite{Be:Tangent}. 
  Therefore, we see that in the ECC case Definition~\ref{def:Privileged-coordinates} agrees with the definition of privileged coordinates in~\cite{Be:Tangent}. 
\end{remark}

\begin{remark}
Our definition of privileged coordinates is different from the definition used in~\cite{ABB:SRG, Go:LNM76, Gr:CC, He:SIAMR, MM:JAM00, Mo:AMS02}.
We will see later that the two definitions are equivalent (see Corollary~\ref{cor:nilp-approx.priv-weight}). 
\end{remark}

\begin{remark}\label{rmk:Privileged.order-xj}
 If the condition (i) holds, then $X_{j}(x_{k})(a)=\partial_{x_j}(x_{k})=\delta_{jk}$. Therefore, we see that in this case $x_{k}$ has 
 order $w_{k}$ if and only if $X^{\alpha}(x_{k})=0$ for all multi-orders $\alpha$ such that $\brak\alpha<w_{k}$ and 
 $|\alpha|\geq 2$. 
\end{remark}

In what follows using local coordinates centered at $a$ we may regard the vector fields $X_{1},\ldots,X_{n}$ as vector
fields defined on a neighborhood of the origin $0\in\R^{n}$.

\begin{lemma}[{\cite[Lemma 4.13]{Be:Tangent}}]\label{lem:privileged.polynomial}
Let $h(x)$ be a homogeneous polynomial of \mbox{degree $m$}. Then
\begin{equation*}
    (X^{\alpha}h)(0)=\left\{
    \begin{array}
    {cc}
        \partial_{x}^{\alpha}h(0)& \text{if $|\alpha|=m$},  \\

        0 & \text{if $|\alpha|<m$}.
    \end{array}\right.
\end{equation*}
\end{lemma}
 \begin{remark} In the proof of the above result in \cite[page~40]{Be:Tangent}, the summation in Eq.~(34) is over all
     multi-orders $\beta=(\beta_{1},\ldots,\beta_{n})$ such that $\beta\neq \alpha$ and $\beta_{i}\leq \alpha_{i}$ for
     $i=1,\ldots,n$. This should be replaced by the summation over all multi-orders $\beta$ such that $|\beta|\leq
     |\alpha|$.
 \end{remark}

\begin{proposition}[Compare {\cite[Theorem~4.15]{Be:Tangent}}]\label{prop:privileged}
Let $(x_{1},\ldots,x_{n})$ be local coordinates centered at $a$ that are linearly adapted to the $H$-frame
$(X_{1},\ldots,X_{n})$. Then there is a unique change of coordinates $x\rightarrow \hat{\psi}(x)$ such that
\begin{enumerate}
    \item  It provides us with privileged coordinates at $a$.

    \item  For $k=1,\ldots,n$, the $k$-th component $\hat{\psi}_{k}(x)$ is of the form,
    \begin{equation}
        \hat{\psi}_{k}(x)=x_{k}+\sum_{\substack{\brak\alpha <w_{k}\\ |\alpha|\geq 2}} a_{k\alpha}x^{\alpha}, \qquad
        a_{k\alpha}\in \R.
        \label{eq-form-h}
    \end{equation}
\end{enumerate}
\end{proposition}
\begin{proof}
Let $x\rightarrow \hat{\psi}(x)$ be a change of coordinates of the form~(\ref{eq-form-h}). Set $y=\hat{\psi}(x)$ and let $k\in \{1,\ldots,n\}$. 
As pointed out in Remark~\ref{rmk:Privileged.order-xj}, the coordinate $y_{k}=\hat{\psi}_k(x)$ has order $w_{k}$ at $a$ if and only if 
$X^{\alpha}(y_{k})(a)=0$ for all multi-orders $\alpha$ with $\brak \alpha<w_{k}$ and $|\alpha|\geq 2$. Let $\alpha \in \N_0$ be such that 
$\brak \alpha<w_{k}$ and $|\alpha|\geq 2$. By Lemma~\ref{lem:privileged.polynomial} we have
\begin{align*}
     \left.X^{\alpha}(y_{k})\right|_{x=0} & =   \left.X^{\alpha}(x_{k})\right|_{x=0} + \sum_{\substack{\brak\beta <w_{k}\\
2\leq |\beta|}} a_{k\beta} \left.X^{\alpha}(x^{\beta})\right|_{x=0}\\
&=\left.X^{\alpha}(x_{k})\right|_{x=0} + \sum_{\substack{\brak\beta <w_{k}\\
 2\leq |\beta|<|\alpha|}} a_{k\beta} \left.X^{\alpha}(x^{\beta})\right|_{x=0} +\alpha!a_{k\alpha}.
\end{align*}Thus,
\begin{equation}\label{eq-pri-det}
    \left.X^{\alpha}(y_{k})\right|_{x=0}=0 \ \Longleftrightarrow \ \alpha!a_{k\alpha}=- \left.X^{\alpha}(x_{k})\right|_{x=0}
    - \sum_{\substack{\brak\beta <w_{k}\\
 2\leq |\beta|<|\alpha|}} a_{k\beta} \left.X^{\alpha}(x^{\beta})\right|_{x=0}.
\end{equation}As the right-hand side uniquely determines the coefficients $a_{k\alpha}$, we deduce that there is a 
unique map $\hat{\psi}(x)$ of the form~(\ref{eq-form-h}) such that the change of variable $x\rightarrow \hat{\psi}(x)$ provides us with privileged coordinates at $a$ that are linearly adapted to 
$H$-frame $(X_{1},\ldots,X_{n})$. The lemma is thus proved.
\end{proof}

\begin{remark}
    The proof above merely reproduces the arguments in~\cite{Be:Tangent}.  In particular, in the ECC case and for the type of $H$-frames mentioned in Remark~\ref{rmk:Privileged.order-difference}, we recover the privileged coordinates of~\cite{Be:Tangent}.  We wrote down the details for reader's convenience 
    and for recording how the coefficients $a_{k\alpha}$ are obtained (\emph{cf}.~Remark~\ref{rem-ab} below). 
    Note also that the uniqueness content is not mentioned in~\cite{Be:Tangent}, but this is an immediate consequence 
    of~(\ref{eq-pri-det}). This uniqueness result will play an important role at several places in the rest of the paper.
\end{remark}

\begin{remark}
 When $r=2$ the map $\hat{\psi}$ is the identity map, and so the privileged coordinates that we obtain are simply linearly 
 adapted coordinates. In the special case of Heisenberg manifolds, these coordinates are called $y$-coordinates in~\cite{BG:CHM}. 
\end{remark}

\begin{remark}
 When $r=3$ we recover the $p$-coordinates of Cummins~\cite{Cu:CPDE89}.   
\end{remark}

 \begin{remark}\label{rem-ab}
    It follows from the proof of Proposition~\ref{prop:privileged} that each coefficient $a_{k\alpha}$ in~(\ref{eq-form-h}) is a universal polynomial in the 
    derivatives $\left.X^{\alpha}(x^{\beta})\right|_{x=0}$ with $\brak \beta \leq w_{k}$ and $|\beta|\geq 1$. Set 
    $X_{j}=\sum_{k=1}^{n}b_{jk}(x)\partial_{x_k}$. An induction shows that
    \begin{equation*}
        X^{\alpha}=\sum_{1\leq |\beta|\leq |\alpha|}b_{\alpha\beta}(x)\partial^{\beta}_x,
    \end{equation*}where $b_{\alpha\beta}(x)$ is a universal polynomial in the partial derivatives 
    $\partial^{\gamma}b_{jk}(x)$ with $|\gamma|\leq |\alpha|-|\beta|$. As 
    $\left.X^{\alpha}(x^{\beta})\right|_{x=0}=\beta! b_{\alpha\beta}(0)$, we then deduce that each coefficient 
    $a_{j\alpha}$ is a universal polynomial in the partial derivatives $\partial^{\gamma}b_{kl}(0)$ with $|\gamma|\leq 
    |\alpha|-1$. 
\end{remark}

\begin{definition}\label{def-psia}
The map $\psi_a:\R^n\rightarrow \R^n$ is composition $\hat{\psi}_a\circ T_a$, where  $T_{a}$ is the affine map from Lemma~\ref{lem-affine} and  
$\hat{\psi}_a$ the polynomial diffeomorphism associated by Proposition~\ref{prop:privileged} with the linearly adapted  coordinates provided by $T_a$.  
\end{definition}

Proposition~\ref{prop:privileged} and the definition of $\psi_a$ immediately imply the following statement. 

\begin{proposition}\label{prop:privileged.psi-privileged}
 The change of coordinates $x\rightarrow \psi_a(x)$ provides us with privileged coordinates at $a$ that are adapted to the $H$-frame $(X_1,\ldots, X_n)$.
\end{proposition}
 
We shall refer to the coordinates provided by Proposition~\ref{prop:privileged.psi-privileged} as the \emph{$\psi$-privileged coordinates}. We conclude this section with the following characterization of these coordinates.  

\begin{proposition}\label{prop-uni}
    The $\psi$-privileged coordinates are the unique privileged coordinates at $a$ adapted to the $H$-frame
    $(X_{1},\ldots,X_{n})$ that are given by a change of coordinates of the form $y=\hat{\psi}(Tx)$, where $T$ is an affine
    map such that $T(a)=0$ and $\hat{\psi}(x)$ is a polynomial diffeomorphism of the form~(\ref{eq-form-h}).
\end{proposition}
\begin{proof}
Let $x\rightarrow \phi(x)$ be a change of coordinates providing us with privileged coordinates at $a$ adapted to the $H$-frame $(X_{1},\ldots,X_{n})$ 
such that $\phi(x)=\hat{\psi}(Tx)$, where $T$ is an affine
    map such that $T(a)=0$ and $\hat{\psi}(x)$ is a polynomial diffeomorphism of the form~(\ref{eq-form-h}). As privileged 
    coordinates are linearly adapted coordinates, we see that  $\phi_{*}X_{j}(0)=\partial_{j}$ for $j=1,\ldots,n$. Note that~(\ref{eq-form-h}) implies that ${\hat{\psi}}'(0)=\op{id}$.
    Thus $\phi_{*}X_{j}(0)=\hat{\psi}'(0)\circ T'(a)\left(X_{j}(a)\right)=T'(a)\left(X_{j}(a)\right)=T_{*}X_{j}(0)$, so
    that we see that $T_{*}X_{j}(0)=\partial_{j}$. This means that the coordinate change $x\rightarrow T(x)$ provides 
    us with coordinates that are linearly adapted at $a$ to the
    $H$-frame $(X_{1},\ldots,X_{n})$. As $T(x)$ is an affine map, it then follows from Lemma~\ref{lem-affine} that
    $T(x)=T_{a}(x)$. Therefore, we see that $\hat{\psi}(x)$ is a polynomial diffeomorphism of the form~(\ref{eq-form-h}) that transforms the coordinates
    $y=\hat{\psi}_{a}(x)$ into privileged coordinates at $a$ adapted to $(X_{1},\ldots,X_{n})$. It then follows from the 
    uniqueness contents of  Proposition~\ref{prop:privileged} that $\hat{\psi}(x)=\hat{\psi}_{a}(x)$, and hence $\phi(x)=\psi_{a}(x)$. This gives the result.
\end{proof}

\begin{remark}
 We refer to~\cite{AS:DANSSSR87, St:1988} for alternative polynomial constructions of privileged coordinates. 
  \end{remark}

\begin{remark}
Examples of non-polynomial  privileged coordinates are provided by the canonical coordinates of the first kind~\cite{Go:LNM76, RS:ActaMath76} and the canonical coordinates of the second kind~\cite{BS:SIAMJCO90, He:SIAMR}.  The former are given by the inverse of the local diffeomorphism, 
\begin{equation}
 (x_1,\ldots, x_n)\longrightarrow \exp\left( x_1X_1+\cdots + x_n X_n\right)\!(a).
 \label{eq:privileged.canonical-coord1} 
\end{equation}
The canonical coordinates of the second kind arise from the inverse of the local diffeomorphism,
\begin{equation}
 (x_1,\ldots, x_n)\longrightarrow \exp\left( x_1X_1\right) \circ \cdots  \circ \exp\left( x_n X_n\right)\!(a). 
  \label{eq:privileged.canonical-coord2}  
\end{equation}
It is shown in~\cite{BS:SIAMJCO90} that the canonical coordinates of the 2nd kind are privileged coordinates in the sense of Definition~\ref{def:Privileged-coordinates} (see also~\cite{Je:Brief14, Mo:AMS02}). For the canonical coordinates of the 1st kind the result is proved in~\cite{Je:Brief14}. We refer to Section~\ref{sec:Canonical-coord} for alternative proofs of these results. 
 \end{remark}

\section{Anisotropic Asymptotic Analysis}\label{sec:anisotropic}
In this section, we gather various results on anisotropic asymptotic expansions of maps and differential operators. Such types of asymptotic expansions have been considered in various levels of generality by a number of authors~\cite{ABB:SRG, BG:CHM, Be:Tangent, Go:LNM76, Gr:CC, He:SIAMR, Je:Brief14, MM:JAM00, Me:CPDE76, Mi:JDG85,  Mo:AMS02, Ro:INA, RS:ActaMath76}. We shall give here a systematic and precise account on this type of 
asymptotic expansions.  In particular, we will show that the various asymptotic expansions at stake are not just pointwise 
asymptotics, but they actually hold with respect to standard $C^{\infty}$-topologies. 

\subsection{Anisotropic approximation of functions} 
In what follows, we let $\delta_t:\R^n\rightarrow \R^n$, $t\in \R$, be the one-parameter group of anisotropic dilations given by 
\begin{equation}
    \delta_{t}(x)=t\cdot x:=(t^{w_{1}}x_{1},\ldots,t^{w_n}x_{n}), \qquad t\in \R, \ x\in \R^n.
    \label{eq:Nilpotent.dilations2}
\end{equation}

\begin{definition}
A function $f(x)$ on $\R^{n}$ is \emph{homogeneous of degree} $w$,
$w\in \N_0$, with respect to the dilations~(\ref{eq:Nilpotent.dilations2}) when
\begin{equation}
    f(t\cdot x)=t^{w}f(x) \qquad \text{for all $x\in \R^n$ and $t\in \R$}.
    \label{eq:Anisotropic.homogeneous}
\end{equation}
\end{definition}

\begin{example}
 For any multi-order $\alpha \in \N_{0}^{n}$, the monomial $x^{\alpha}$ is homogeneous of degree $\brak \alpha$.
\end{example}

\begin{remark}\label{rem-poly}
 Suppose that $f(x)$ is smooth and homogeneous of degree $w$. Differentiating~(\ref{eq:Anisotropic.homogeneous}) with respect to $t$ shows that, for every $\alpha \in \N_0$, we have $\partial^{\alpha}f(t\cdot x)=t^{w-\brak \alpha}\partial^{\alpha}f(x)$. That is, 
 $\partial^{\alpha}f(x)$ is homogeneous of degree $w-\brak \alpha$. Moreover, if we take $\alpha$ so that $\brak \alpha>w$ and let $t\rightarrow 0$, then we get
 \begin{equation*}
     \partial^{\alpha}f(x)=t^{\brak\alpha -w}\partial^{\alpha}f(t\cdot x)\longrightarrow 0\cdot \partial^{\alpha}f(0)=0.
 \end{equation*}Therefore, all the partial derivatives $\partial^{\alpha}f$ with $\brak \alpha>w$ are identically 
 zero. Combining this with the inequality $\brak \alpha<r|\alpha|$ shows that $\partial^{\alpha}f(x)$ is identically 
 zero as soon as $|\alpha|$ is large enough. It then follows that $f(x)$ is a polynomial function.  Note also that this implies that $w$ must be 
 a non-negative integer.
\end{remark}

In what follows we let $U$ be an open neighborhood of the origin $0\in \R^{n}$. 

\begin{definition}
Let $f\in C^{\infty}(U)$ and $w\in \N_0$. We shall say that
\begin{enumerate}
    \item  $f$ has weight $\geq w$ when $\partial^{\alpha}_{x}f(0)=0$ for all multi-orders
    $\alpha \in \N_{0}^{n}$ such that $\brak\alpha<w$.
    \item  $f$ has weight $w$ when $f(x)$ has weight~$\geq w$ and there is a multi-order
    $\alpha\in \N_{0}^{n}$ with $\brak\alpha=w$ such that $\partial^{\alpha}_{x}f(0)\neq 0$.
\end{enumerate}
\end{definition}

\begin{remark}
 Unless $f(x)$ vanishes at infinite order at $x=0$, it always has a (finite) weight. In the case $f(x)$ vanishes at infinite order at $x=0$ we shall say that its weight
  is $w=\infty$. 
\end{remark}

\begin{example}
 Let $f(x)\in C^\infty(\R^n)$ be homogeneous of degree $w$ with respect to the dilations~(\ref{eq:Nilpotent.dilations2}). 
 If $f(x)$ is not zero everywhere, then $f(x)$ has weight $w$. 
\end{example}

We mention the following version of Taylor's formula. 

\begin{lemma}[Anisotropic Taylor Formula]\label{lem:nilpotent-weighted-Taylor} 
    Let $f\in C^{\infty}(U)$.  Then, for any $N\in \N$, we may write
    \begin{equation}
        f(x)=\sum_{\brak\alpha <N} \frac{1}{\alpha!}\partial_{x}^{\alpha}f(0)x^{\alpha} + R_N(x), 
    \label{eq:nilpotent-weighted-Taylor}
    \end{equation}
    where the remainder term $R_N(x)\in C^\infty(U)$ has weight~$\geq N$ and is of the form,     
      \begin{equation}
         R_{N}(x) =  \sum_{|\alpha|\leq N \leq \brak   \alpha}x^{\alpha}R_{N\alpha}(x).
             \label{eq:nilpotent-weighted-Taylor2}
      \end{equation}    
\end{lemma}
\begin{proof}
 By Taylor's formula there are functions $R_{N\alpha}\in C^{\infty}(U)$, $|\alpha|=N$, such that
 \begin{equation*}
        f(x)=\sum_{|\alpha| <N} \frac{1}{\alpha!}\partial_{x}^{\alpha}f(0)x^{\alpha} + 
    \sum_{|\alpha|=N }x^{\alpha}R_{N\alpha}(x).
    \end{equation*} 
     Using the inequality $|\alpha|\leq \brak \alpha$ we may rewrite this as 
    \begin{equation*}
      f(x)=\sum_{\brak\alpha <N} \frac{1}{\alpha!}\partial_{x}^{\alpha}f(0)x^{\alpha} +  \sum_{|\alpha|<N\leq \brak\alpha} \frac{1}{\alpha!}\partial_{x}^{\alpha}f(0)x^{\alpha} +
    \sum_{|\alpha|=N}x^{\alpha}R_{N\alpha}(x). 
    \end{equation*}
     For $|\alpha|<N\leq \brak \alpha$, set $R_{N\alpha}(x)=\frac{1}{\alpha!}f(0)$. In addition, set $R_N(x)= \sum_{|\alpha|\leq N \leq \brak   \alpha}x^{\alpha}R_{N\alpha}(x)$. Then we obtain~(\ref{eq:nilpotent-weighted-Taylor})--(\ref{eq:nilpotent-weighted-Taylor2}). Note also that if $\brak\beta<N$ and $|\alpha|\leq N \leq \brak   \alpha$, then $\partial^\beta_x (x^\alpha R_{N\alpha}(x))$ vanishes at $x=0$, and so $\partial^\beta_x R_N(0)=0$ whenever $\brak\beta<N$. This means that $R_N(x)$ has weight~$\geq N$. The proof is complete. 
\end{proof}

\begin{definition}\label{def:anisotropic-pseudo-norm}
A \emph{pseudo-norm} on $\R^n$ is any continuous function $\|\cdot \|:\R^n \rightarrow [0,\infty)$ which is (strictly) positive on $\R^n\setminus 0$ and satisfies 
\begin{equation}
 \|t\cdot x\|= |t| \|x\| \qquad \text{for all $x\in \R^n $ and $t\in \R$}.
 \label{eq:anisotropic.homogeneity-pseudo-norm}
\end{equation}
\end{definition}
 
\begin{example}
The following functions on $\R^n$ are pseudo-norms, 
\begin{gather}
\|x\|_0= \max\left\{|x_1|^{\frac1{w_1}}, \ldots, |x_n|^{\frac1{w_n}}\right\}, \qquad x\in \R^n,
\label{eq:anisotropic.max-pseudo-norm} \\
\|x\|_1=  |x_1|^{\frac{1}{w_1}} + \cdots + |x_n|^{\frac{1}{w_n}}, \qquad x\in \R^n.
\end{gather}
 \end{example}

In what follows we shall say that two pseudo-norms $\|\cdot\|$ and $\|\cdot\|'$ are equivalent when there are constants $c_1>0$ and $c_2>0$ such that
\begin{equation}
 c_1\|x\| \leq \|x\|' \leq c_2\|x\| \qquad \text{for all $x\in \R^n$}. 
 \label{eq:anisotropic.equivalence-pseudo-norms}
\end{equation}

\begin{lemma} \label{lem:anisotropic.equivalence-pseudo-norms}
 All pseudo-norms are equivalent. 
\end{lemma}
\begin{proof}
 As the inequalities~(\ref{eq:anisotropic.equivalence-pseudo-norms}) define an equivalence relation on the set of pseudo-norms, it is enough to show that every pseudo-norm $\|\cdot\|$ is equivalent to the pseudo-norm $\|\cdot \|_0$ given by~(\ref{eq:anisotropic.max-pseudo-norm}). To see this we observe that  the unit sphere $\Omega_0:=\{x\in \R^n;\ \|x\|_0=0\}$ is just the box $\{x\in \R^n;\ \max(|x_1|,\ldots, |x_n|) =1\}$, and so it is a compact subset of $\R^n\setminus 0$. As $\|\cdot \|$ is a positive continuous function on $\R^n\setminus 0$ we deduce there are constants $c_1>0$ and $c_2>0$ such that $c_1\leq \|x\|\leq c_2$ for all $x\in \Omega_0$. For all $x\in \R^n\setminus 0$, the vector $\|x\|_0^{-1}\cdot x$ is in $\Omega_0$, and so by homogeneity we get
\begin{equation*}
  \|x\|=  \|x\|_0 \left\| \|x\|_0^{-1}\cdot x\right\| \leq c_2\|x\|_0 \qquad \text{and} \qquad  \|x\|=  \|x\|_0 \left\| \|x\|_0^{-1}\cdot x\right\|\geq  c_1\|x\|_0 .  
\end{equation*}
This shows that the pseudo-norms $\|\cdot\|_0$ and $\|\cdot \|$ are equivalent. The proof is complete.  
\end{proof}

We shall equip $C^{\infty}(U)$ with its standard locally convex space topology, i.e., the LCS topology 
defined by the semi-norms, 
\begin{equation*}
    p_{K,\alpha}(f)=\sup_{x\in K}|\partial^{\alpha}f(x)|, \qquad \alpha\in \N_{0}^{n}, \quad \text{$K\subset U$ 
    compact}. 
\end{equation*}
Therefore, a sequence $(f_{\ell}(x))_{\ell}\subset C^{\infty}(U)$ converges to a function $f(x)$ in $C^{\infty}(U)$ if and only if, for  all 
multi-orders $\alpha\in \N_{0}^{n}$, the partial derivatives $\partial^{\alpha}f_{\ell}(x)$ converge to $\partial^{\alpha}f$ uniformly on all compact subsets of $U$. 

\begin{remark}\label{rmk:anisotropic.swallowing-compact}
 Given any compact $K\subset U$, the semi-norm $p_{K,\alpha}(f)$ actually makes sense for any smooth function $f$ on an open neighborhood of $K$. Moreover, as $U$ is an open neighborhood of the origin, it contains some ball $B(0,\rho)$ with $\rho>0$. Note that, for all $t\in (0,1)$, we have $\delta_{t^{-1}}(U)\supset \delta_{t^{-1}}(B(0,\rho)) \supset B(0,t^{-w_n}\rho)$, and so  $\delta_{t^{-1}}(U)\supset K$ as soon as $t$ is small enough. Therefore, for any $f\in C^\infty(U)$, the semi-norm $p_{K,\alpha}(f\circ \delta_{t})$ makes sense as soon as $t$ is sufficiently small. This allows us to speak about asymptotics in $C^\infty(U)$ for $f(t\cdot x)$ as $t\rightarrow 0$.
\end{remark}

In what follows we let $\|\cdot \|$ be a pseudo-norm on $\R^n$. 

\begin{lemma}\label{lem:anisotropic.weight}
 Let $f(x)\in C^\infty(U)$,  and set $ \cU=\{(x,t)\in U \times \R; \ t\cdot x \in U\}$. Given 
 any $w\in \N_0$, the following are equivalent:
 \begin{enumerate}
     \item[(i)]   $f(x)$ has weight~$\geq w$.
 
     \item[(ii)]   $f(x)=\op{O}(\|x\|^w)$ near $x=0$. 
 
     \item[(iii)]  For all $x\in \R^{n}$  and as $t\rightarrow 0$, we have $f(t\cdot x)=\op{O}(t^{w})$. 
 
     \item[(iv)] As $t\rightarrow 0$, we have  $f(t\cdot x)=\op{O}(t^{w})$ in $C^{\infty}(U)$. 
     
     \item[(v)]  There is a function $\Theta(x,t)\in C^\infty(\cU)$ such that $f(t\cdot x)=t^w \Theta(x,t)$ for all $(x,t)\in \cU$. 
 \end{enumerate}
\end{lemma}
\begin{proof}
 Let us first show that (i), (iii), (iv) and (v) are equivalent. It is immediate that (iv) implies (iii). Moreover, by Lemma~\ref{lem:nilpotent-weighted-Taylor} 
 there are functions $R_{N\alpha}(x)\in C^\infty(U)$, $|\alpha|\leq w \leq \brak\alpha$, such that
\begin{equation}
 f(x)=\sum_{\brak\alpha <w} \frac{1}{\alpha!}x^\alpha \partial_x^\alpha f(0) + \sum_{|\alpha|\leq w \leq \brak\alpha} x^\alpha R_{N\alpha}(x). 
 \label{eq:anisotropic.Taylor-expansion2}
\end{equation}
 Let $\Theta:\cU \rightarrow \C$ be the smooth function defined by
 \begin{equation*}
\Theta(x,t)=  \sum_{|\alpha|\leq w \leq \brak\alpha} t^{\brak\alpha-w}x^\alpha R_{N\alpha}(t\cdot x)  \qquad \forall (x,t)\in \cU.
 \end{equation*}
Using~(\ref{eq:anisotropic.Taylor-expansion2}) we see that, for all $(x,t)\in \cU$, we have
\begin{equation}
 f(t\cdot x)- \sum_{\brak\alpha <w} \frac{1}{\alpha!}t^{\brak\alpha} x^\alpha \partial_x^\alpha f(0)=  \sum_{|\alpha|\leq w \leq \brak\alpha} (t\cdot x)^\alpha R_{N\alpha}(t\cdot x)=t^w \Theta(x,t).
 \label{eq:anisotropic.Taylor-expansion-Theta-tx} 
\end{equation}

Let $x\in \R^n$. As soon $t$ is small enough $t\cdot x\in U$, and so $(x,t)\in \cU$. Therefore, we see that $\theta (x,t)=\op{O}(1)$ as $t\rightarrow 0$. Combining this with~(\ref{eq:anisotropic.Taylor-expansion-Theta-tx}) shows that, as $t\rightarrow 0$, we have 
\begin{equation}
 f(t\cdot x)= \sum_{\brak\alpha <w} \frac{1}{\alpha!}t^{\brak\alpha} x^\alpha \partial_x^\alpha f(0)+\op{O}(t^m). 
 \label{eq:anisotropic.ftxxalphaOtm}
\end{equation}
If (iii) holds then, for all $x\in \R^n$, we have $f(t\cdot x)=\op{O}(t^m)$ as $t\rightarrow 0$. Comparing this to~(\ref{eq:anisotropic.ftxxalphaOtm}) we  deduce that $\partial^\alpha f(0)=0$ for $\brak \alpha<w$, i.e., $f(x)$ has weight~$\geq w$. This shows that (iii) implies (i). 

Furthermore, if $f(x)$ has weight~$\geq w$, then $\partial^\alpha f(0)=0$ for $\brak \alpha<w$, and so~(\ref{eq:anisotropic.Taylor-expansion-Theta-tx}) gives $f(t\cdot x)=t^w \Theta(x,t)$ for all $(x,t)\in \cU$. Therefore, we see that (i) implies (v). 

If  (v) holds, then there is a function $\Theta(x,t)\in C^\infty(\cU)$ such that $f(t\cdot x)=t^w \Theta(x,t)$ for all $(x,t)\in \cU$. Note that $\cU$ is an open subset of $\R\times \R^n$. Furthermore, by Remark~\ref{rmk:anisotropic.swallowing-compact}, given any compact $K\subset U$, there is $t_0>0$ small enough so that $K\times [-t_0,t_0] \subset \cU$. Combining this with the smoothness of $\Theta(x,t)$ we deduce that, as $t\rightarrow 0$, we have $\Theta(x,t)=\op{O}(1)$ in $C^\infty(U)$, and so $f(t\cdot x)=t^w \Theta(x,t)=\op{O}(t^w)$ in $C^\infty(U)$. This shows that (v) implies (iv). It then follows that (i), (iii), (iv) and (v) are equivalent.  

Bearing this in mind, suppose that $f(x)=\op{O}(\|x\|^w)$ near $x=0$. Let $x\in \R^n$. As $t\rightarrow 0$ we have
\begin{equation*}
f(t\cdot x)=\op{O}(\|t\cdot x\|^w)=\op{O}(|t|^w\|x\|)=\op{O}(|t|^w). 
\end{equation*}
Therefore, we see that (ii) implies (iii). 

To complete the proof it is enough to show that (v) implies (ii). Assume there is  a function $\Theta(x,t)\in C^\infty(\cU)$ such that $f(t\cdot x)=t^w \Theta(x,t)$ for all $(x,t)\in \cU$. Let $x\in U\setminus 0$. Note that $\|x\|\cdot(\|x\|^{-1}\cdot x)=x\in U$, and so $(\|x\|^{-1}\cdot x, \|x\|)\in \cU$. Thus, 
\begin{equation}
 f(x)=f\left( \|x\|\cdot(\|x\|^{-1}\cdot x)\right) =\|x\|^w \Theta \left( \|x\|^{-1}\cdot x, \|x\|\right). 
 \label{eq:anisotropic.f-theta|x|x}
\end{equation}

\begin{claim}
 The anisotropic sphere $\Omega:=\{x\in \R^n;\ \|x\|=1\}$ is compact.
\end{claim}
\begin{proof}[Proof of the claim] 
The continuity of $\|\cdot\|$ ensures us that $\Omega$ is a closed set. In addition, let $\|\cdot\|_0$ be the pseudo-norm~(\ref{eq:anisotropic.max-pseudo-norm}). Thanks to Lemma~\ref{lem:anisotropic.equivalence-pseudo-norms} we know there is $c>0$ such that $\|x\|_0\leq c \|x\|$ for all $x\in \R^n$. In particular, if $x\in \Omega$, then $|x_k|^{\frac1{w_k}}\leq \|x\|_0 \leq c\|x\|=c$ for $k=1,\ldots, n$. Therefore, we see that $\Omega$ is contained in the closed cube $\prod_{k=1}^n [-c^{w_k}, c^{w_k}]$. As this cube is compact and $\Omega$ is closed, it then follows that $\Omega$ is compact, proving the claim. 
 \end{proof}
 
 As $\Omega$ is compact, there is $t_0>0$ such that $\Omega\times [-t_0,t_0] \subset \cU$. Then $\Omega\times [-t_0,t_0]$ is a compact subset of $\cU$, and so there $C>0$ such that 
 \begin{equation}
|\Theta(x,t)|\leq C \qquad \text{for all $(x,t)\in\Omega\times [-t_0,t_0]$}. 
\label{eq:anisotropic.thetaC}
\end{equation}
 Let $x\in \R^n\setminus 0$ be such that $\|x\|\leq t_0$. Then $(\|x\|^{-1}\cdot x, \|x\|)\in\Omega\times [-t_0,t_0] \subset \cU$, and so using~(\ref{eq:anisotropic.f-theta|x|x}) and~(\ref{eq:anisotropic.thetaC}) we get
\begin{equation*}
 |f(x)|=\|x\|^w \left| \Theta \left(\|x\|^{-1}\cdot x, \|x\|\right)\right| \leq C\|x\|^m.
\end{equation*}
 This shows that $f(x)=\op{O}(\|x\|^w)$ near $x=0$. Therefore, we see that (v) implies (ii). The proof is complete. 
 \end{proof}

\begin{proposition}\label{prop:anisotropic.ftx}
    Let $f(x)\in C^\infty(U)$. Then, as $t\rightarrow 0$, we have 
                 \begin{equation}\label{eq-ftx-1}
            f(t\cdot x)\simeq  \sum_{\ell\geq 0} t^{\ell}f^{[\ell]}(x) \qquad \text{in $C^{\infty}(U)$},
        \end{equation}where $f^{[\ell]}(x)$ is a polynomial which is homogeneous of degree $\ell$ with respect to the dilations~(\ref{eq:Nilpotent.dilations2}) 
        (see Eq.~ (\ref{eq:anisotropic.fell}) below). 
\end{proposition}

\begin{remark}
   The asymptotic expansion~(\ref{eq-ftx-1}) holds in $C^{\infty}(U)$ if and only if, for every $\alpha\in 
  \N_{0}^{n}$ and $N>w$ and for every compact $K\subset U$, we have
  \begin{equation*}
      p_{K,\alpha}\biggl( f\circ \delta_{t}-\sum_{w\leq \ell <N}t^{\ell}f^{[\ell]}\biggr)=\op{O}(t^{N}) \qquad \text{as 
      $t\rightarrow 0$}. 
  \end{equation*}
\end{remark}

\begin{proof}[Proof of Proposition~\ref{prop:anisotropic.ftx}]
 For $\ell =0, 1,\ldots $ set
\begin{equation}
 f^{[\ell]}(x)= \sum_{\brak \alpha =l} \frac{1}{\alpha!}\partial_{x}^{\alpha}f(0)x^{\alpha}. 
 \label{eq:anisotropic.fell}
\end{equation}
 Then $f^{[\ell]}(x)$ is  a polynomial which is homogeneous of degree $\ell$ with respect to the dilations~(\ref{eq:Nilpotent.dilations2}). Moreover, by Lemma~\ref{lem:nilpotent-weighted-Taylor}, for every $N\in \N$, there is a function $R_N(x)\in C^\infty(U)$ of weight~$\geq N$ such that
  \begin{equation*}
    f(x) =\sum_{\ell <  N} f^{[l]}(x) + R_N(x). 
\end{equation*}
As $R_N(x)$ has weight~$\geq N$,  Lemma~\ref{lem:anisotropic.weight} ensures us that, as $t\rightarrow 0$, we have $R_N(t\cdot x)=\op{O}(t^N)$ in $C^\infty(U)$. Therefore, as $t\rightarrow 0$ and in $C^\infty(U)$, we have
\begin{equation*}
 f(t\cdot x)= \sum_{\ell <  N} f^{[l]}(t\cdot x) + R_N(t\cdot x)=  \sum_{\ell <  N} t^{\ell} f^{[l]}(x)+\op{O}(t^N). 
\end{equation*}
This gives the asymptotic expansion~(\ref{eq-ftx-1}). The proof is complete. 
\end{proof}

\begin{corollary}\label{cor:anisotropic.ftx}
 Let $f(x)\in C^\infty(U)$. Then $f(x)$ has weight $w$, $w\in \N_0$, if and only if, as $t\rightarrow 0$, we have  
\begin{equation}
 f(t\cdot x) =t^w f^{[w]}(x) + \op{O}(t) \qquad \qquad \text{in $C^{\infty}(U)$},
 \label{eq:Anisotropic.limit-ftx}
        \end{equation}
        where $f^{[w]}(x)$ is a non-zero polynomial which is homogeneous of degree $w$ with respect to the dilations~(\ref{eq:Nilpotent.dilations2}).
\end{corollary}
\begin{proof}
The asymptotics~(\ref{eq:Anisotropic.limit-ftx}) means that in the asymptotics~(\ref{eq-ftx-1}) for $f(x)$ the leading non-zero terms is $f^{[w]}$. In view of~(\ref{eq:anisotropic.fell}) this means hat 
 $\partial^\alpha f(0)=0$ for $\brak \alpha <w$ and $\partial^\alpha f(0)\neq 0$ for some $\alpha \in \N_0^n$ with $\brak \alpha =w$. That is, $f(x)$ has weight~$w$. This proves the result.
\end{proof}

\subsection{Anisotropic Approximation of Multi-Valued Maps}
We shall now extend the previous results to (smooth) maps with values in some Euclidean space $\R^{n'}$. We assume we are given a weight sequence $(w_{1}',\ldots,w_{n'}')$. By this we mean a non-decreasing sequence of integers where $w_1'=1$.  This enables to endow $\R^{n'}$ with the family of anisotropic dilations~(\ref{eq:Nilpotent.dilations2}) associated with this weight sequence. We shall use the same notation for the dilations on $\R^n$ and $\R^{n'}$. When $n=n'$ we shall tacitly  assume that $w_j'=w_j$ for $j=1,\ldots,n$.

\begin{definition}\label{eq:anisotopic.w-homogeneous}
     A map $\phi:\R^{n}\rightarrow \R^{n'}$ is \emph{$w$-homogeneous} when
     \begin{equation}\label{eq-def-weight}
         \phi(t\cdot x)=t \cdot \phi(x) \qquad \text{for all $x\in \R^{n}$ and $t\in \R$}.
     \end{equation}
 \end{definition}
 
 \begin{remark}\label{rmk:Privileged-w-homogeneous-map}
 Set $\phi(x)=(\phi_{1}(x),\ldots,\phi_{n'}(x))$.  The condition~(\ref{eq-def-weight}) exactly means that, for every $k=1,\ldots,n'$, the component 
 $\phi_{k}(x)$ is homogeneous of degree $w_{k}'$ with respect to the dilations~(\ref{eq:Nilpotent.dilations2}). In particular, in view of Remark~\ref{rem-poly}, we see
 that if $\phi$ is smooth and $w$-homogeneous, then it must be a polynomial map. More precisely,   each 
 component $\phi_{k}(x)$ is a linear combination of monomials $x^{\alpha}$ with $\brak \alpha =w_{k}'$. 
 \end{remark}

\begin{definition}\label{def:anisotropic.Thetaw}
Let $\Theta(x)=(\Theta_{1}(x),\ldots,\Theta_{n'}(x))$ be a smooth map from $U$  to $\R^{n'}$. Given $m\in \Z$, $m\geq -w'_{n'}$, we say that $\Theta(x)$ is $\Ow(\|x\|^{w+m})$, and write $\Theta(x)=\Ow(\|x\|^{w+m})$, when, for $k=1,\ldots,n'$, we have
\begin{equation*}
 \Theta_{k}(x)=\op{O}(\|x\|^{w_{k}'+m}) \qquad \text{near $x=0$}.  
\end{equation*}
\end{definition}

In what follows we equip $C^\infty(U,\R^{n'})$ with its standard Fr\'echet-space topology. As an immediate consequence of Lemma~\ref{lem:anisotropic.weight} we obtain the following characterization of $\Ow(\|x\|^{w+m})$-maps. 

\begin{lemma}\label{lem-eq-we}
Given $m\in \N_0$, let $\Theta(x)=(\Theta_{1}(x),\ldots,\Theta_{n'}(x))$ be a smooth map from $U$ to $\R^{n'}$. In addition, set $ \cU=\{(x,t)\in U \times \R; \ t\cdot x \in U\}$. Then the following are equivalent:
\begin{enumerate}
 \item[(i)]    The map $\Theta(x)$ is $\Ow(\|x\|^{w+m})$ near $x=0$.

 \item[(ii)]   For $k=1,\ldots,n'$, the component $\Theta_{k}(x)$ has weight~$\geq w_{k}'+m$. 

 \item[(iii)] For all $x\in \R^{n}$ and as $t\rightarrow 0$, we have $t^{-1}\cdot \Theta(t\cdot x)=\op{O}(t^{m})$.  

 \item[(iv)]  As $t\rightarrow 0$, we have $t^{-1}\cdot \Theta(t\cdot x)=\op{O}(t^{m})$ in $C^{\infty}(U,\R^{n'})$. 
 
 \item[(v)]  There is $\tilde{\Theta}(x,t)\in C^\infty(\cU, \R^{n'})$ such that 
 $t^{-1}\cdot \Theta(t\cdot x)=t^m \tilde{\Theta}(x,t)$ for all $(x,t)\in \cU$, $t\neq 0$.
\end{enumerate}
\end{lemma}

The following is a multi-valued version of Proposition~\ref{prop:anisotropic.ftx}. 

\begin{proposition}\label{lem:multi.Thetat}
Let $\Theta(x)=(\Theta_1(x), \ldots, \Theta_{n'}(x))$ be a smooth map from $U$ to $\R^{n'}$. Then, as $t\rightarrow 0$, we have
\begin{equation}
t^{-1}\cdot  \Theta(t\cdot x) \simeq \sum_{\ell \geq -w_{n'}'} t^\ell \Theta^{[\ell]}(x) \qquad \text{in $C^\infty(U,\R^{n'})$},
\label{eq:multi.Thetat}
\end{equation}
where $\Theta^{[\ell]}(x)$ is a polynomial map such that $t^{-1}\cdot \Theta^{[\ell]}(x)=t^\ell \Theta^{[\ell]}(x)$ for all $\R^*$ (see Eq.~(\ref{eq:multi.Thetaell}) below). 
\end{proposition}
\begin{proof}
For $k=1,\ldots, n'$, we know by Proposition~\ref{prop:anisotropic.ftx} that, as $t\rightarrow 0$ and in $C^\infty(U)$, we have 
\begin{equation}
\Theta_k(t\cdot x) \simeq \sum_{\ell \geq 0} t^\ell \Theta_k^{[\ell]}(x) \simeq t^{w_k'}\sum_{\ell' \geq -w_k'} t^{\ell'} \Theta_k^{[\ell'+w_k']}(x), 
\label{eq:multi.Thetaktx}
\end{equation}
where $ \Theta_k^{[\ell]}(x)$ is a polynomial which is homogeneous of degree~$\ell$ with respect to the dilations~(\ref{eq:Nilpotent.dilations2}). For $\ell\geq -r$ define 
\begin{equation}
\Theta^{[\ell]}(x) = \left( \Theta_1^{[\ell+w_1']}(x), \ldots, \Theta_{n'}^{[\ell+w_{n'}]}(x)\right), \qquad x \in \R^n,
\label{eq:multi.Thetaell}
\end{equation}
with the convention that $\Theta_k^{[\ell+w_k']}(x)=0$ when $\ell< -w_k'$. Then $ \Theta^{[\ell]}(x)$ is a polynomial map from $\R^n$ to $\R^{n'}$ such that $t^{-1} \cdot \Theta^{[\ell]}(t\cdot x)=t^\ell \Theta^{[\ell]}(x)$ for all $t\in \R^*$. Moreover, the asymptotic expansions~(\ref{eq:multi.Thetaktx}) for $k=1,\ldots,n$ give the asymptotics~(\ref{eq:multi.Thetat}). The proof is complete. 
\end{proof}

\begin{remark}\label{rmk:anisotropic.weighted-asymptotic} 
 Using Lemma~\ref{lem-eq-we} we see that the asymptotic expansion~(\ref{eq:multi.Thetat}) exactly mean that, for every integer $m> -w_{n'}'$, we have 
\begin{equation*}
 \Theta(x) - \sum_{\ell <m}  \Theta^{[\ell]}(x)=\Ow\left( \|x\|^{w+m}\right) \qquad \text{near $x=0$}.
\end{equation*}
In particular, we see that $ \Theta(x)=\Ow( \|x\|^{w+m})$ if and only if $\Theta^{[\ell]}(x)=0$ for $\ell<m$. 
\end{remark}

In the special case of diffeomorphisms we further have the following result. 

\begin{proposition}\label{lem:Carnot-coord.inverse-Ow}
Let $\phi:U_{1}\rightarrow U_{2}$ be a smooth diffeomorphism, where $U_{1}$ and $U_{2}$ are open neighborhoods of 
the origin $0\in \R^{n}$. Assume further that $\phi(0)=0$ and there is $m\in \N$ such that, near $x=0$, we have 
\begin{equation}
    \phi(x)=\hat{\phi}(x)+\Ow(\|x\|^{w+m}),
    \label{eq:Carnot-coord.phi-hatphi}
\end{equation}where $\hat{\phi}(x)$ is a $w$-homogeneous polynomial map. Then $\hat{\phi}(x)$ is a diffeomorphism of $\R^n$ and, near 
$x=0$, we have
\begin{equation*}
    \phi^{-1}(x)=\hat{\phi}^{-1}(x)+\Ow(\|x\|^{w+m}). 
\end{equation*}
\end{proposition}
\begin{proof}
We observe that~(\ref{eq:Carnot-coord.phi-hatphi}) implies that $\phi(x)$ and $\hat{\phi}(x)$ have the same differential at the origin. It then 
follows that $\hat{\phi}(x)$ is a diffeomorphism near the origin. Its $w$-homogeneity then implies it is a 
diffeomorphism from $\R^{n}$ onto itself. 

 Let $\Theta:U_2\rightarrow \R^n$ be the smooth map defined by 
\begin{equation*}
\Theta(x)=\phi^{-1}(x)-\hat{\phi}^{-1}(x) \qquad \text{for all $x\in U_2$}. 
\end{equation*}
Suppose that $\Theta(x)=\Ow(\|x\|^{w+m'})$ with $m'<m$. Lemma~\ref{lem-eq-we} then ensures us that $t^{-1}\cdot  \Theta(t\cdot x)=\op{O}(t^w)$ in  $C^\infty(U,\R^{n'})$
as $t\rightarrow 0$. Therefore, in the asymptotics~(\ref{eq:multi.Thetat}) for $\Theta(x)$ all the maps $\Theta^{[\ell]}(x)$ vanish for $\ell<m$, and so, as $t\rightarrow 0$, we have
\begin{equation}
    t^{-1}\cdot \Theta(t\cdot x)=  t^{m'} \Theta^{[m']}(x)+ \op{O}(t^{m'+1}) \qquad \text{in $C^\infty(U_2,\R^n))$}.  
\label{eq:multi.psi-t}
\end{equation}
For $j=1,2$, set $ \cU_{j}=\{(x,t)\in U \times \R; \ t\cdot x \in U_{j}\}$. Using~(\ref{eq:multi.psi-t}) and  Lemma~\ref{lem-eq-we} shows there is a smooth map $\tilde{\Theta}:\cU_{2}\rightarrow \R^{n}$ such that, for all $(x,t)\in \cU_{2}$ with $t\neq 0$, we have
\begin{equation}
       t^{-1}\cdot \Theta(t\cdot x)= t^{m'} \Theta^{[m']}(x)+t^{m'+1}\tilde{\Theta}(x,t). 
               \label{eq:Carnot-coord.phi-Theta2}
   \end{equation}

Let $x\in \R^n$. For $t\in \R^*$ small enough $(x,t)\in \cU_1$.  In this case set $\phi_t(x)=t^{-1}\cdot \phi(t\cdot x)$. Then~(\ref{eq:Carnot-coord.phi-hatphi}) and Lemma~\ref{lem-eq-we} imply that 
  \begin{equation}
                \phi_t(x)=\hat{\phi}(x) +\op{O}(t^m) \qquad \text{as $t\rightarrow 0$}. 
                \label{eq:multi.phi-t}
    \end{equation}
Thus, 
\begin{equation}
t^{-1} \cdot \Theta \left[ t \cdot \phi_t(x)\right] = t^{-1}\cdot \left[ \phi^{-1}\circ \phi (t\cdot x)\right] - \hat{\phi}^{-1}\left(\hat{\phi}(x) +\op{O}(t^m)\right)=\op{O}(t^m).
\label{eq:Anisotropic.Theta-Otm} 
\end{equation}
Moreover, using~(\ref{eq:Carnot-coord.phi-Theta2}) and~(\ref{eq:multi.phi-t}) we also obtain     
\begin{align*}
 t^{-1} \cdot \Theta\left[ t\cdot \phi_t(x)\right] & =   t^{m'} \Theta^{[m']}\left[ \phi_t(x)\right] +t^{m'+1}\tilde{\Theta}\left(\phi_t(x),t \right) \\
 & =   t^{m'} \Theta^{[m']}\left[ \hat\phi(x) +\op{O}(t)\right] + \op{O}\left(t^{m'+1}\right)\\\
 & =   t^{m'} \Theta^{[m']}\circ  \hat\phi(x) + \op{O}\left(t^{m'+1}\right). 
\end{align*}
 As $m'<m$ comparing this with~(\ref{eq:Anisotropic.Theta-Otm}) shows that $\Theta^{[m']}\circ \hat{\phi}(x)=0$. As $ \hat{\phi}(x)$ is a diffeomorphism this implies that 
 $\Theta^{[m']}(x)=0$ for all $x\in \R^n$, and so~(\ref{eq:Carnot-coord.phi-Theta2}) becomes  $t^{-1}\cdot \Theta(t\cdot x)= t^{m'+1}\tilde{\Theta}(x,t)$. By Lemma~\ref{lem-eq-we} this implies that $\Theta(x)=\Ow(\|x\|^{w+m'+1})$. 

 We just have established that if $\Theta(x)=\Ow(\|x\|^{w+m'})$ with $m'<m$, then $\Theta(x)=\Ow(\|x\|^{w+m'+1})$. An immediate induction then shows that  $\Theta(x)=\Ow(\|x\|^{w+m})$, i.e., $ \phi^{-1}(x)=\hat{\phi}^{-1}(x)+\Ow(\|x\|^{w+m})$. The proof is complete. 
\end{proof}

\subsection{Anisotropic approximation of differential operators} 
The notion of weight of a function extends to differential operators as follows. Given a differential operator $P$ on 
$U$,  for $t\in \R^*$ we denote by $\delta_{t}^{*}P$ the pullback of $P$ by the dilation $\delta_{t}$, i.e., $\delta_{t}^{*}P$ is the differential operator on $ \delta_{t}^{-1}(U)$ given by
\begin{equation*}
  (\delta_{t}^{*}P)u(x)= P(u\circ \delta_{t}^{-1})(t\cdot x) \qquad \text{for all $u\in C^{\infty}\left( 
  \delta_{t}^{-1}(U)\right)$.}  
\end{equation*}
If we write $P=\sum_{|\alpha|\leq m} a_\alpha(x)\partial^\alpha_x$, with $a_\alpha\in C^\infty(U)$, then we have 
\begin{equation}
\delta_{t}^{*}P= \sum_{|\alpha|\leq m} t^{-|\alpha|} a_\alpha(t\cdot x) \partial^\alpha_x. 
\label{eq:Anistropic.deltat*P}
\end{equation}

\begin{definition}
  A differential operator  $P$ on $\R^{n}$ is \emph{homogeneous of degree $w$}, $w\in \Z$, when
\begin{equation*}
    \delta_{t}^{*}P=t^{w} P \qquad \text{for all $t\in \R^*$}.
\end{equation*}
\end{definition}

\begin{example}
 For any $\alpha \in \N_0$, the differential operator $\partial^{\alpha}_x$ is
homogeneous of degree~$-\brak\alpha$.
\end{example}

\begin{remark}
 If we write $P=\sum a_\alpha(x) \partial^\alpha_x$ with $a_\alpha\in C^\infty(\R^n)$, then~(\ref{eq:Anistropic.deltat*P}) shows that $P$ is homogeneous of degree $w$ if and only if each coefficient $a_\alpha(x)$ is homogeneous of degree $w+\brak\alpha$. In particular, all the coefficients $a_\alpha(x)$ must be polynomials. 
\end{remark}

\begin{definition}
  Let  $P=\sum_{|\alpha|\leq m}a_{\alpha}(x)\partial_{x}^{\alpha}$ be a differential operator on $U$. We say that $P$ has \emph{weight} $w$, $w\in \Z$, when
  \begin{enumerate}
      \item[(i)]  Each coefficient $a_{\alpha}(x)$ has weight~$\geq w+\brak\alpha$.

      \item[(ii)]  There is one coefficient $a_{\alpha}(x)$ that has weight~$w+\brak\alpha$.
  \end{enumerate}
\end{definition}

\begin{example}
 Let $P$ be a differential operator on $\R^d$ which is homogeneous of degree~$w$. If $P\neq 0$, then  $P$ has weight $w$. 
\end{example}

\begin{remark}
 A differential operator of order~$\leq m$ always has weight~$\geq -mr$. 
\end{remark}

In what follows, given any $m\in \N_0$, we denote by $\DO^{m}(U)$ the space of $m$-th order differential operators on $U$.  If $P$ 
is an operator in $\DO^{m}(U)$, then they are unique coefficients $a_{\alpha}(P)(x)\in C^{\infty}(U)$, $|\alpha|\leq m$, such that
\begin{equation}
    P=\sum_{|\alpha|\leq m}a_{\alpha}(P)(x)\partial^{\alpha}_x.
    \label{eq:Anisotropic.P-partial}
\end{equation}
Set $N(m)=\# \{\alpha\in \N_{0}^{n};\ |\alpha|\leq m\}$. We then have a linear isomorphism $P\rightarrow P[x]$ from $\DO^m(U)$ onto $C^\infty(U,\R^{N(m)})$ given by
\begin{equation}
P[x]= \left( a_\alpha(P)(x))\right)_{|\alpha|\leq m}, \qquad P\in \DO^m(U).
 \label{eq:anisotropic.ThetaP} 
\end{equation}
The standard Fr\'echet-space topology of $\DO^m(U)$ is so that this map is a topological isomorphism. In addition, on $\R^{N(m)}$ we have a family of dilations 
$\delta_t(x)=t\cdot x$, $t\in \R$, where 
\begin{equation*}
 t\cdot x = \left( t^{\brak \alpha} x_\alpha\right)_{|\alpha|\leq m}, \qquad x=(x_\alpha)_{|\alpha|\leq m}\in \R^{N(m)}. 
\end{equation*}
Using this notation we can rewrite~(\ref{eq:Anistropic.deltat*P}) in the form, 
\begin{equation*}
 \left(\delta_t^* P\right)[x] = t^{-1}\cdot P[t\cdot x] \qquad \text{for all $t\in \R^*$}. 
\end{equation*}
Combining this with Proposition~\ref{lem:multi.Thetat} we obtain the following version for differential operators. 

\begin{proposition}\label{lem:Nilpotent.weight-diff-op}
    Let $P\in \DO^{m}(U)$. Then, as $t\rightarrow 0$, we have
        \begin{equation}\label{eq-delta-p}
            \delta_{t}^{*}P\simeq \sum_{\ell \geq -mr} t^{\ell}P^{[\ell]} \qquad \text{in $\DO^{m}(U)$},
        \end{equation}
        where $P^{[\ell]}$ is a homogeneous differential operator of degree $\ell$.
\end{proposition}

\begin{remark}
    It follows from~(\ref{eq:anisotropic.fell}) and the proof of Proposition~\ref{lem:multi.Thetat} that we have
    \begin{equation}
P^{[\ell]}= \sum_{|\alpha|\leq m} a_\alpha(P)^{[\ell+\brak\alpha]}(x) \partial_x^\alpha = \sum_{\brak\beta + \ell =\brak \alpha} \frac{1}{\beta !}  
\partial^\beta_xa_\alpha(P)(0) x^\beta\partial^\alpha_x. 
\label{eq:anisotropic.Pw}
\end{equation}
\end{remark}

\begin{corollary}\label{cor:anisotropic.dtP-weight}
 Let $P\in \DO^m(U)$. Then $P$ has weight $w$ if and only if, as $t\rightarrow 0$, we have
     \begin{equation}
       \delta_{t}^{*}P= t^{w}P^{[w]} +\op{O}\left(t^{w+1}\right)\qquad \text{in $\DO^{m}(U)$},
       \label{eq:anisotropic.dtP-weight}
    \end{equation}
    where $P^{[w]}$ is a non-zero homogeneous polynomial differential operator of degree $w$.
\end{corollary}
\begin{proof}
 In the same way as in the proof of Corollary~\ref{cor:anisotropic.ftx}, the asymptotics~(\ref{eq:anisotropic.dtP-weight}) means that in the asymptotics~(\ref{eq-delta-p}) for $P$ the leading non-zero term is $P^{[w]}$. In view of~(\ref{eq:anisotropic.Pw}) this means that $\partial^\beta_x a_\alpha(P)(0)=0$ when $\brak\beta <w+\brak\alpha$ and there is at least one pair $(\alpha, \beta)$ with $\brak\beta =w+\brak\alpha$ and $|\alpha|\leq m$ such that  $\partial^\beta_x a_\alpha(P)(0)=0$. Equivalently, each coefficient $a_\alpha(P)(x)$ has weight~$\geq w+\brak\alpha$ and we have equality for at least one of them. That is, $P$ has weight~$w$. This proves the result. 
\end{proof}

Let us now specialize the above results to vector fields. Let $\cX(U)$ be the space of (smooth) vector fields on $U$. Regarding it as a (closed) subspace of $\DO^{1}(U)$ we 
equip it with the induced topology. Equivalently, a vector field $X$ on $U$ has a unique expression as
\begin{equation*}
    X=\sum_{1\leq j \leq n}a_{j}(X)(x)\partial_{j}, \qquad \text{with $a_{j}(X)\in C^{\infty}(U)$}.
\end{equation*}
In the notation of~(\ref{eq:anisotropic.ThetaP}) we have $X[x]=(a_1(X)(x),\ldots, a_n(X)(x))$, which is a the usual identification of a vector field with a vector-valued map. This provides us with a topological isomorphism between $\cX(U)$ and $C^\infty(U,\R^n)$. 

Specializing Lemma~\ref{lem:Nilpotent.weight-diff-op} and Corollary~\ref{cor:anisotropic.dtP-weight} to vector fields leads us to the following statement. 

\begin{proposition}\label{prop:nilp-approx.vector-fields}
    Let $X$ be a smooth vector field on $U$. 
    \begin{enumerate}
\item As $t\rightarrow 0$, we have 
                \begin{equation}\label{eq-delta-X}
            \delta_{t}^{*}X\simeq \sum_{\ell\geq -r} t^{\ell}X^{[\ell]} \qquad \text{in $\cX(U)$},
        \end{equation}
        where $X^{[\ell]}$ is a homogeneous polynomial vector field of degree $\ell$.

\item $X$ has weight $w$ if and only if, as $t\rightarrow 0$, we have 
       \begin{equation*}
       \delta_{t}^{*}X= t^{w}X^{[w]} +\op{O}\left(t^{w+1}\right)\qquad \text{in $\cX(U)$}.
    \end{equation*}
\end{enumerate}
\end{proposition}

\section{Characterization of Privileged Coordinates}\label{sec:charact-privileged}
In this section, we produce a characterization of privileged coordinates in terms of the anisotropic approximation 
of vector fields the previous section. This will show that our definition of privileged coordinates
is equivalent to that of Goodman~\cite{Go:LNM76}. This will also show how to get all privileged coordinates at a given point and lead us to new proofs that 
that canonical coordinates of the first and second kind are privileged coordinates.

\subsection{Characterization of privileged  coordinates}
Let $a\in M$ and $(x_{1},\ldots,x_{n})$ local coordinates centered at $a$. We denote by $U$ the range of these local 
coordinates and by $V$ their domain. Note that $U$ is an open neighborhood of the origin in $\R^{n}$. We shall further 
assume there is an $H$-frame $(X_1,\ldots, X_n)$ over $V$. 

Using the local coordinates $(x_{1},\ldots,x_{n})$ we can regard any function (resp., vector field, differential 
operator) on $V$ as a function (resp., vector field, differential operator) on $U$. This enables us to define the 
\emph{weight} of such objects. In fact, as the weight depends only on the germ near $x=0$, we actually can define the weight for 
any such object that is defined near $x=a$. However, it should be pointed out that this notion of weight is \emph{extrinsic}, since it depends on the 
choice of the local coordinates. For instance, it is not preserved by permutation of the coordinates 
$(x_{1},\ldots,x_{n})$ (unless $r=1$). For this reason we will refer to this weight as the weight in the local coordinates 
$(x_{1},\ldots,x_{n})$. 

\begin{definition}\label{def:Privileged-coordinates.model-vector-field}
    Let $X$ be a vector field on an open neighborhood of $a$. Let $w$ be its weight in the local coordinates 
    $(x_{1},\ldots,x_{n})$. Then the vector field $X^{[w]}$ in~(\ref{eq-delta-X}) is denoted by $X^{(a)}$ and is called the \emph{model 
    vector field} of $X$ in the local coordinates $(x_{1},\ldots,x_{n})$. 
\end{definition}

\begin{remark}
The asymptotic expansion~(\ref{eq-delta-X}) implies that, as $t\rightarrow 0$, we have
  \begin{equation}\label{eq-t0X}
      t^{-w}\delta_{t}^{*}X=X^{(a)} +\op{O}(t) \qquad \text{in $\cX(U)$}.
 \end{equation}
 Furthermore, if we write $X=\sum_{k=1}^n b_k(x)\partial_x^\alpha$, with $b_k(x)\in C^\infty(U)$,  then~(\ref{eq:anisotropic.Pw}) gives
 \begin{equation}
 X^{(a)} = \sum_{\brak\alpha+w=w_k} \frac{1}{\alpha!} \partial_x^\alpha b_k(0)x^{\alpha} \partial_{x_k}. 
 \label{eq:anisotropic.X(a)} 
\end{equation}
\end{remark}

\begin{remark}
    We can similarly  define the model operator of any differential operator on an open neighborhood of $a$. 
\end{remark}

As mentioned above, the notion of weight is an \emph{extrinsic} notion. However, as the following shows, when using
privileged coordinates this extrinsic notion of weight actually agrees with the \emph{intrinsic} notion of order defined 
in the previous section.

\begin{lemma}[see also~\cite{Je:Brief14}]\label{lem-order-weight}
 Let $f$ be a smooth function near $a$ of order $N$, and $(x_1,\ldots, x_n)$ privileged coordinates at $a$ adapted to the $H$-frame
    $(X_{1},\ldots, X_{n})$. Then $f(x)$ has weight~$N$ in these coordinates. 
\end{lemma}
\begin{proof}
As the order of a function is an intrinsic notion, we may work in the local coordinates $(x_{1},\ldots,x_{n})$. Suppose that $f(x)$ has weight~$\geq N+1$ in these coordinates. 
By Lemma~\ref{lem:nilpotent-weighted-Taylor} near $x=0$ there are functions $R_{\alpha}(x)$, 
$|\alpha|\leq N+1\leq \brak \alpha$, such that
\begin{equation}
  f(x)=  \sum_{|\alpha|\leq N+1 \leq \brak \alpha}x^{\alpha}R_{\alpha}(x).
  \label{eq:privileged.Taylo-fx}  
\end{equation}
As $(x_1,\ldots, x_n)$ are privileged coordinates, for $k=1,\ldots, n$, the coordinate $x_k$ has weight $w_k$. Lemma~\ref{lem:privileged.properties-order-product} then implies that each monomial $x^\alpha$ has order $\geq \alpha_1 w_1+ \cdots + \alpha_n w_n=\brak\alpha$, and so each reminder term $x^\alpha R_\alpha(x)$ in~(\ref{eq:privileged.Taylo-fx}) has order~$\geq \brak\alpha\geq N+1$. It then follows that $f(x)$ has order~$\geq N+1$, which is not possible since by assumption $f(x)$ has order~$N$. Thus, $f(x)$ must have weight~$\leq N$. 

Let $w$ be the weight of $f(x)$. Then $\partial^\alpha f(0) =0$ for $\brak\alpha <w$, and so by Lemma~\ref{lem:nilpotent-weighted-Taylor} there is a smooth function $R(x)$ of weight~$\geq w+1$ such that
\begin{equation*}
 f(x) =\sum_{\brak\beta =w} \frac{1}{\beta!}\partial^{\beta}f(0)x^{\beta} +R(x). 
\end{equation*}
As $f(x)$ has weight $w$ there is a multi-order $\alpha$ such that $\partial^\alpha f(0)\neq 0$. We may choose $\alpha$ such that $\partial^\beta f(0)= 0$ if $\brak \beta=w$ and $|\beta|>|\alpha|$. Then we have 
\begin{equation}
 X^\alpha f(0) =\sum_{\substack{\brak\beta =w\\ |\beta|\leq |\alpha|}} \frac{1}{\beta!}\partial^{\beta}f(0)X^\alpha(x^{\beta})(0) + X^\alpha R(0). 
 \label{eq:privileged.Xalphaf0}
\end{equation}
As $R(x)$ has weight~$\geq w+1$, it has order~$\geq w+1$, and so $X^\alpha R(0)=0$. Moreover, it follows from Lemma~\ref{lem:privileged.polynomial} that if $|\beta|\leq |\alpha|$, then $X^\alpha(x^\beta)(0)=\alpha ! \delta_{\alpha \beta}$. Combining this with~(\ref{eq:privileged.Xalphaf0}) we see that $X^\alpha f(0) = \partial^\alpha f(0)\neq 0$. This implies that  $N$, the order  of $f(x)$,  must be~$\geq w$. As it also is $~\geq w$, we deduce that $w=N$. This proves the result. 
\end{proof}

\begin{remark}
 In the ECC case, the equality between weight and order is mentioned without proof in~\cite[p.~43]{Be:Tangent}, but a proof
is given in~\cite{Je:Brief14}. 
\end{remark}

We are now in a position to establish the following characterization of privileged coordinates. 

\begin{theorem}\label{thm-pri-equiv}
    Let $(x_{1},\ldots,x_{n})$ be local coordinates centered at $a$ that are linearly adapted to the $H$-frame
    $(X_{1},\ldots, X_{n})$.  In addition, let $U\subset \R^n$ be the range of these coordinates.  Then the following are equivalent:
    \begin{enumerate}
        \item[(i)]  The local coordinates $(x_{1},\ldots,x_{n})$ are privileged coordinates at $a$.
        
        \item[(ii)] For $j=1,\ldots, n$ and as $t\rightarrow 0$, we have
                       \begin{equation}
                              t^{w_j} \delta_t^*X_j = X_j^{(a)} + \op{O}(t) \qquad \text{in $\cX(U)$}, 
                              \label{eq:char-priv.model-vector-field2}
                       \end{equation}
                     where $X_j^{(a)}$ is  homogeneous of degree~$-w_j$. 
        
        \item[(iii)]  For $j=1,\ldots,n$, the vector field $X_{j}$ has weight $-w_{j}$ in the local coordinates $(x_{1},\ldots,x_{n})$.

        \item[(iv)] For every multi-order $\alpha\in \N_{0}^{n}$, the differential operator $X^{\alpha}$ has weight $-\brak \alpha$ in the local coordinates 
                        $(x_{1},\ldots,x_{n})$.
    \end{enumerate}
\end{theorem}
\begin{proof}
It follows from Proposition~\ref{prop:nilp-approx.vector-fields} that (iii) implies (ii). Conversely, suppose that (ii) holds.  As the coordinates $(x_{1},\ldots,x_{n})$ are linearly adapted, we know that $X_j(0)=\partial_j$ for $j=1,\ldots,n$. Combining this with~(\ref{eq:Anistropic.deltat*P}) and~(\ref{eq:char-priv.model-vector-field2}) we see that, as $t\rightarrow 0$, we have 
\begin{equation}
X_j^{(a)}(0) + \op{O}(t)=  t^{w_j} \delta_t^*X_j(0) =  t^{w_j}   t^{-w_j} \partial_j= \partial_j. 
\label{eq:Char-priv.Xj(a)(0)}
\end{equation}
Thus, $X_j^{(a)}$ agrees with $\partial_j$ at $x=0$, and hence this is a non-zero vector field. Combining this with~(\ref{eq:char-priv.model-vector-field2}) and Proposition~\ref{prop:nilp-approx.vector-fields} then shows that $X_j$ has weight~$-w_j$ for $j=1,\ldots, n$. Therefore, we see that (ii) and (iii) are equivalent.  

Let us now establish the equivalence between (i), (iii) and (iv). If $X=\sum_{j=1}^n a_{j}(x)\partial_{x_j}$ is a vector field on $U$, then, for $k=1,\ldots,n$, we have 
$X(x_{k})=\sum_{j=1}^n a_{j}(x)\partial_{x_j}(x_{k})=a_{k}(x)$. Thus,
\begin{equation}
 X_j = \sum_{1\leq k \leq n} X_j(x_k)\partial_{x_k}, \qquad j=1,\ldots, n. 
 \label{eq:anistropic.Xj-Xjxk}
\end{equation}
More generally, if $P=\sum_{1\leq |\alpha| \leq m} a_{\alpha}(x)\partial^{\alpha}_x$ is a differential operator near the origin such that $P(0)=0$, then $a_{\alpha}(x)=P(x^{\alpha})$ for $|\alpha|=1$. In particular, given any non-zero multi-order $\alpha \in \N_0^n$, the differential operator $X^\alpha$ takes the form,
\begin{equation*}
X^\alpha = \sum_{1\leq k \leq n} X^\alpha(x_k) \partial_{x_k} + \sum_{2 \leq |\beta| \leq |\alpha|} a_{\alpha\beta}(x) \partial^\beta, \qquad a_{\alpha\beta}(x)\in C^\infty(U).
\end{equation*}
If we further assume that each such differential operator $X^\alpha$ has weight~$-\brak\alpha$, then, for every $k=1,\ldots, n$, the function $X^\alpha(x_k)$ must have weight~$\geq w_k-\brak\alpha$, and so $X^\alpha(x_k)(0)=0$ whenever $\brak\alpha <w_k$. Using Remark~\ref{rmk:Privileged.order-xj} we then deduce that $(x_{1},\ldots,x_{n})$ are privileged coordinates. This shows that (iv) implies (i). 

Suppose now that $(x_{1},\ldots,x_{n})$ are privileged coordinates at $a$. Thus, each coordinate $x_k$ then has order $w_k$. It then follows from the definition of the order of a function that the function $X_j(x_k)$ has order~$\geq w_k-w_j$ for $j=1,\ldots, n$. As $(x_{1},\ldots,x_{n})$ are privileged coordinates, Lemma~\ref{lem-order-weight} ensures us that $X_j(x_k)$ has weight~$\geq w_k-w_j$ when $w_k\geq w_j$. Together with~(\ref{eq:anistropic.Xj-Xjxk}) this implies that $X_j$ has weight~$\geq -w_j$. As $X_j(0)=\partial_j$, it then follows that $X_j$ has weight $-w_j$ for $j=1,\ldots, n$. Therefore, we see that (i) implies (iii).

To complete the proof it remains to show that (iii) implies (iv).  Assume that $X_{j}$ has weight $-w_{j}$ for $j=1,\ldots,n$. Set $X_{j}=\partial_{x_j}+\sum_{k=1}^n b_{jk}(x)\partial_{x_k}$, 
with $b_{jk}(x)\in C^{\infty}(U)$. As the coordinates $(x_{1},\ldots,x_{n})$ are linearly adapted at $a$ to the 
$H$-frame $(X_{1},\ldots,X_{n})$. We have 
$b_{jk}(0)=0$ for $j,k=1,\ldots,n$. Combining this with~(\ref{eq:anisotropic.X(a)}) then shows that the model vector fields $X_{j}^{(a)}$ are given by
\begin{equation}\label{eq-homo-app}
    X_{j}^{(a)}= \partial_{x_j}+ \sum_{\substack{w_{j}+\brak\alpha=w_{k}\\ w_k>w_j}} \frac{1}{\alpha!}\partial^{\alpha}_x
    b_{jk}(0)x^{\alpha} \partial_{x_k}.
\end{equation}
More generally, given $\alpha \in \N_{0}^{n}$, the differential operator $(X^{(a)})^\alpha$ is of the form, 
\begin{equation}
 \left(X^{(a)}\right)^\alpha= \partial_x^\alpha + \sum_{\substack{\brak\alpha+\brak\beta=\brak\gamma\\ \brak\gamma>\brak \alpha}} b_{\alpha\beta\gamma}x^\beta \partial_x^\gamma, \qquad b_{\alpha\beta\gamma}\in \R.
 \label{eq:nilp-approx.model-Xalpha} 
\end{equation}
In particular, we see that $ \left(X^{(a)}\right)^\alpha\neq 0$. Set $m=|\alpha|$. Using~(\ref{eq-t0X}) we deduce that, as $t\rightarrow 0$ and in $\DO^{m}(U)$, we have
\begin{equation}
    t^{\brak\alpha}\delta_t^* X^{\alpha}= \left(t^{w_{1}}\delta_{t}^{*}X_{1}\right)^{\alpha_{1}}\cdots
    \left(t^{w_{n}}\delta_{t}^{*}X_{n}\right)^{\alpha_{n}}\longrightarrow \left(X_{1}^{(a)}\right)^{\alpha_{1}}\cdots
    \left(X_{n}^{(a)}\right)^{\alpha_{n}}. 
    \label{eq:nilp-approx.limit-Xalpha}
\end{equation}
As~$(X_{1}^{(a)})^{\alpha_{1}}\cdots(X_{n}^{(a)})^{\alpha_{n}}= (X^{(a)})^\alpha\neq 0$, this shows that $X^{\alpha}$ has weight $-\brak\alpha$. This proves that (iii) implies (iv). The proof is complete. 
\end{proof}

\begin{remark}
 The fact that we have the asymptotic expansions of the form~(\ref{eq:char-priv.model-vector-field2}) in privileged coordinates is well known (see~\cite{BS:SIAMJCO90}; see also~\cite{Be:Tangent, Je:Brief14}). 
\end{remark}

Goodman~\cite{Go:LNM76} defined privileged coordinates as linearly adapted local coordinates satisfying~(\ref{eq:char-priv.model-vector-field2}). Therefore, Theorem~\ref{thm-pri-equiv} implies that Goodman's definition is equivalent to ours. There is a number of constructions of privileged coordinates in the sense of~\cite{Go:LNM76} (see, e.g.,~\cite{ABB:SRG, BS:SIAMJCO90, Go:LNM76, Gr:CC, He:SIAMR, MM:JAM00, Me:CPDE76, Mo:AMS02, RS:ActaMath76}). We thus arrive at the following statement.   

\begin{corollary}\label{cor:nilp-approx.priv-weight}
 All privileged coordinates in the sense of~\cite{Go:LNM76} are privileged coordinates in the sense of Definition~\ref{def:Privileged-coordinates}. This includes all the aforementioned examples.   
\end{corollary}

\subsection{Getting all privileged coordinates at a point} We shall now explain how Theorem~\ref{thm-pri-equiv}
 enables us to get \emph{all} systems of privileged coordinates at a given point. 
 
 Bearing in mind Definition~\ref{eq:anisotopic.w-homogeneous}, we have the following result. 

\begin{proposition}\label{prop:char-priv.phi-hatphi-Ow}
 Let $(x_1,\ldots, x_n)$ be privileged coordinates at $a$ adapted to the $H$-frame $(X_1,\ldots, X_n)$. Then a change of coordinates $x\rightarrow \phi(x)$ produces privileged coordinates at $a$ adapted to $(X_1,\ldots, X_n)$ if and only if we have 
\begin{equation}
 \phi(x)=\hat{\phi}(x) +\Ow\left(\|x\|^{w+1}\right)  \qquad \text{near $x=0$},
 \label{eq:char-priv-coord.phi-x-Ow}
\end{equation}
where $\hat{\phi}(x)$ is a $w$-homogeneous polynomial diffeomorphism such that $\hat{\phi}'(0)=\op{id}$. 
\end{proposition}
\begin{proof}
 Let $U$ be the range of the coordinates $(x_1,\ldots, x_n)$. We may assume that $U$ agrees with the domain of $\phi$. We also set $V=\phi(U)$ and $y=(y_1, \ldots,y_n)=\phi(x)$. Suppose that $(y_1, \ldots,y_n)$ are privileged coordinates at $a$ adapted to $(X_1,\ldots, X_n)$, so that each coordinate $y_k=\phi_k(x)$ has order $w_k$. As $(x_1, \ldots, x_n)$ are privileged  coordinates as well, Lemma~\ref{lem-order-weight} then ensures us that $\phi_k(x)$ has weight $w_k$ in the coordinates $(x_1, \ldots, x_n)$ for $k=1,\ldots, n$. By Lemma~\ref{lem-eq-we} this implies that $t^{-1}\cdot \phi(t\cdot x)=\op{O}(1)$ in $C^\infty(U,\R^n)$ as $t\rightarrow 0$. Combining this with Proposition~\ref{lem:multi.Thetat} shows that, as $t \rightarrow 0$, we have  
\begin{equation*}
t^{-1}\cdot \phi(x) =\hat{\phi}(x) + \op{O}(t) \qquad \text{in $C^\infty(U,\R^n)$},
\end{equation*}
where $\hat\phi(x)$ is a $w$-homogeneous polynomial map. By Lemma~\ref{lem-eq-we} this implies that
\begin{equation}
 \phi(x) =\hat{\phi}(x) +  \Ow\left(\|x\|^{w+1}\right)  \qquad \text{near $x=0$}. 
 \label{eq:char-priv-coord.phi-hatphi-Ow}
\end{equation}
This also implies that $\phi(x) = \hat{\phi}(x)+\op{O}(|x|^2)$ near $x=0$, and so $\phi(x)$ and $\hat{\phi}(x)$ have the same differential at $x=0$. Moreover, as both $(x_1,\ldots, x_n)$ and $(y_1,\ldots, y_n)$ are linearly adapted coordinates, for $j=1,\ldots, n$, we have 
\begin{equation}
 \partial_j= \phi_*X_j(y)_{|y=0} = \phi'(0)[X_j(x)_{|x=0}] = \phi'(0)\partial_j. 
 \label{eq:char-priv-coord.phidj}
\end{equation}
Thus, we see that $\hat\phi'(0)=\phi'(0)=\op{id}$. 

Conversely, suppose that $\phi(x)$ has a behavior of the form~(\ref{eq:char-priv-coord.phi-hatphi-Ow}) near $x=0$, where  $\hat{\phi}(x)$ is a $w$-homogeneous polynomial diffeomorphism such that $\hat{\phi}'(0)=\op{id}$. As above, the asymptotics~(\ref{eq:char-priv-coord.phi-hatphi-Ow}) implies that $\phi'(0)=\hat{\phi}(0)=\op{id}$, and so in the same way as in~(\ref{eq:char-priv-coord.phidj}), for $j=1,\ldots,n$, we have $ \phi_*X_j(y)_{|y=0} = \phi'(0)[X_j(x)_{|x=0}]= X_j(x)_{|x=0}=\partial_j$. Therefore, the coordinates $(y_1,\ldots, y_n)$ are linearly adapted to  $(X_1,\ldots, X_n)$. Moreover, by Proposition~\ref{lem:Carnot-coord.inverse-Ow} the asymptotics~(\ref{eq:char-priv-coord.phi-hatphi-Ow}) also implies that $\hat{\phi}(x)$ is a diffeomorphism and $ \phi^{-1}(x) =\hat{\phi}^{-1}(x) +  \Ow\left(\|x\|^{w+1}\right)$ near $x=0$. Combining this with Lemma~\ref{lem-eq-we} we see that, as $t\rightarrow 0$, we have 
\begin{equation}
t^{-1}\cdot \phi^{-1}(x) =\hat{\phi}^{-1}(x) + \op{O}(t) \qquad \text{in $C^\infty(V,\R^n)$}.
\label{eq:char-priv-coord.phi-1tx}
\end{equation}
In addition, as $(x_1,\ldots, x_n)$ are privileged coordinates, we know by Theorem~\ref{thm-pri-equiv} that, for $j=1,\ldots, n$ and as 
$t\rightarrow 0$, we have
\begin{equation}
 t^{w_j} \delta_t^* X_j(x)= X_j^{(a)}(x)+\op{O}(t) \qquad \text{in $\cX(U)$}, 
 \label{eq:char-priv-coord.phidtXj} 
\end{equation}
 where the vector field $X^{(a)}$ is homogeneous of degree~$-w_j$. 

For $t\in \R^*$ set $\phi_t(x)=t^{-1}\cdot \phi(t\cdot x)$. Combining~(\ref{eq:char-priv-coord.phi-1tx}) and~(\ref{eq:char-priv-coord.phidtXj}) shows that, as $t\rightarrow 0$ and in $\cX(V)$, we have 
\begin{align}
 t^{w_j} \delta_t^* (\phi_*X_j)(y)& = t^{w_j} \left( \phi_t\right)_* \delta_t^*X_j(y)\nonumber \\
 & = \left( \phi_t^{-1}\right)'(y) \left[ (t^{w_j} \delta_t^*X_j)\left(\phi_t^{-1}(y)\right)\right] \nonumber\\
 & =  \left( \hat{\phi}^{-1}\right)'(y) \left[ X_j^a\left(\hat{\phi}^{-1}(y)\right)\right] +\op{O}(t)  \label{eq:char-priv.dtphiXj} \\
 & =\hat{\phi}_*X_j^{(a)}(y) +\op{O}(t).\nonumber
\end{align}
It is immediate that $ \hat{\phi}_*X_j^{(a)}(y)$ is homogeneous of degree $-w_j$. It then follows from Theorem~\ref{thm-pri-equiv} that the change of coordinates $x\rightarrow \phi(x)$ provides us with privileged coordinates. The proof is complete. 
\end{proof}

Combining Proposition~\ref{prop:char-priv.phi-hatphi-Ow} with Proposition~\ref{prop:privileged.psi-privileged} we then arrive at the following statement. 

\begin{corollary}\label{cor:privileged.all-privileged-coord}
 A system of local coordinates is a system of privileged coordinates at $a$ adapted to $(X_1,\ldots, X_n)$ if and only if it arises from a local chart of the form $\phi \circ \psi_{\kappa(a)} \circ \kappa$, where $\kappa$ is  local chart near $a$ such that $\kappa(a)=0$, the map $\psi_{\kappa(a)}$ is as in Definition~\ref{def-psia}, and $\phi(x)$ is a diffeomorphism near the origin $0\in \R^n$ satisfying~(\ref{eq:char-priv-coord.phi-x-Ow}).  
\end{corollary}

\section{Nilpotent Approximations of a Carnot Manifold}\label{sec:Nilpotent-approx}
In this section, after recalling the nilpotent approximation of a Carnot manifold at a given point in privileged coordinates, we shall determine \emph{all} the nilpotent approximations that occur at a given point. Incidentally, this will clarify the dependance of the nilpotent approximation on the choice of the privileged coordinates. 

Throughout this section we let $(X_{1},\ldots,X_{n})$ be an $H$-frame near a given point $a\in M$. 

\subsection{Nilpotent approximation}\label{sec:extrinsic}
We shall now recall how the anisotropic approximation of vector fields in privileged coordinates described in 
the previous section leads us to the so-called nilpotent approximation. In particular, we will recover 
the tangent groups of Bella\"iche~\cite{Be:Tangent} and Gromov~\cite{Gr:CC} (see 
also~\cite{ABB:SRG, Go:LNM76, Je:Brief14, MM:JAM00, Me:CPDE76, Ro:INA, Mo:AMS02, RS:ActaMath76}). As opposed to the tangent group in the sense of 
Definition~\ref{def:Tangent.Tangent-group} this construction is \emph{extrinsic} since it depends on the choice of privileged coordinates. 

We shall work in privileged coordinates centered at $a$ and adapted to the $H$-frame $(X_{1},\ldots.,X_{n})$. 
For $j=1,\ldots,n$, we denote by $X_{j}^{(a)}$ the model vector 
field~(\ref{def:Privileged-coordinates.model-vector-field}) in these coordinates. 

\begin{definition} $\tilde{\fg}^{(a)}$ is the subspace of $T\R^{n}$ spanned by the model vector fields
    $X^{(a)}_{j}$, $j=1,\ldots,n$.
\end{definition}

For $w=1,\ldots,r$, let $\tilde{\fg}^{(a)}_{w}$ be the subspace of $\tilde{\fg}^{(a)}$ spanned by the vector fields $X_{j}^{(a)}$,
with $w_{j}=w$. This provides us with the grading,
\begin{equation}\label{eq-g-grading}
    \tilde{\fg}^{(a)}=\fg_{1}^{(a)}\oplus \cdots \oplus \fg_{r}^{(a)}.
\end{equation}
Moreover, as $(X_{1},\ldots,X_{n})$ is an $H$-frame, it follows from Remark~\ref{rmk:Carnot-mfld.brackets-H-frame} 
that there are smooth functions $L_{ij}^{k}(x)$, $w_{k}\leq w_{i}+w_{j}$, satisfying~(\ref{eq:Carnot-mfld.brackets-H-frame}). 

\begin{lemma}\label{lem-XX}
    For $i,j=1,\ldots,n$, we have
 \begin{equation}
     [X_{i}^{(a)}, X_{j}^{(a)}]=\left\{
     \begin{array}{cl}
   {\displaystyle \sum_{w_{k}=w_{i}+w_{j}}L_{ij}^{k}(a)X_{k}^{(a)}} & \text{if $w_{i}+w_{j}\leq  r$},   \\
         0 & \text{otherwise}.  \\
     \end{array}\right.
     \label{eq:Nilpotent.structure-constants-Xj(a)}
 \end{equation}
\end{lemma}
\begin{proof}In $\cX(U)$ we have
    \begin{equation*}
     [X_{i}^{(a)}, X_{j}^{(a)}]=\lim_{t\rightarrow 0} [
     t^{w_{i}}\delta_{t}^{*}X_{i},t^{w_{j}}\delta_{t}^{*}X_{j}]=\lim_{t\rightarrow 0}
     t^{w_{i}+w_{j}}\delta_{t}^{*}[X_{i},X_{j}].
    \end{equation*}Combining this with~(\ref{eq:Carnot-mfld.brackets-H-frame}) we get
    \begin{equation*}
      [X_{i}^{(a)}, X_{j}^{(a)}]=\sum_{w_{k}\leq w_{i}+w_{j}}\lim_{t\rightarrow 0}
      t^{w_{i}+w_{j}}\delta_{t}^{*}(L_{ij}^{k}X_{k})=  \sum_{w_{k}\leq w_{i}+w_{j}}  L_{ij}^{k}(a)
     \lim_{t\rightarrow 0} t^{w_{i}+w_{j}}\delta_{t}^{*}X_{k}.
    \end{equation*}
    Note that $\lim_{t\rightarrow 0} t^{w_{i}+w_{j}}\delta_{t}^{*}X_{k}=X_{k}^{a}$ if $w_{k}=w_{i}+w_{j}$ and
    $\lim_{t\rightarrow 0} t^{w_{i}+w_{j}}\delta_{t}^{*}X_{k}=0$ if $w_{k}<w_{i}+w_{j}$. Therefore, $[X_{i}^{(a)},
    X_{j}^{(a)}]$ is equal to $\sum_{w_{k}=w_{i}+w_{j}}L_{ij}^{k}(a)X_{k}^{(a)}$ if $w_{i}+w_{j}\leq
    r$ and is zero otherwise. The proof is complete.
\end{proof}

As an immediate consequence of Lemma~\ref{lem-XX} we obtain the following result.
\begin{proposition}
With respect to the Lie bracket of vector fields and the grading~(\ref{eq-g-grading}) the vector space $\tilde{\fg}^{(a)}$ is a 
graded nilpotent Lie algebra of step $r$
\end{proposition}

 In fact, it follows from~(\ref{eq-jk}) and Lemma~\ref{lem-XX} that the Lie algebras ${\fg}M(a)$ and $\tilde{\fg}^{(a)}$ have the same structure constants
 with respect to their respective bases $\{\xi_{j}(a)\}$ and $\{X_{j}^{(a)}\}$, where $\xi_j(a)$ is the class of $X_j(a)$ in $\fg_{w_j} M(a)$. Therefore, we obtain the following
 result.

\begin{proposition}\label{prop-hat-L}
Let $\lie_a:{\fg}M(a)\rightarrow   \tilde{\fg}^{(a)}$ be the linear map defined by
\begin{equation}\label{eq-hat-L}
\lie_{a}\left(x_{1}\xi_{1}(a)+\cdots + x_{n}\xi_{n}(a)\right) =x_{1}X_{j}^{(a)}+\cdots
  +x_{n}X_{n}^{(a)},\qquad x_{j}\in \R.
\end{equation}
Then $\lie_{a}$ is a graded Lie algebra isomorphism from ${\fg}M(a)$ onto $\tilde{\fg}^{(a)}$.
\end{proposition}

We also observe that~(\ref{eq:Nilpotent.structure-constants-Xj(a)}) implies that, for $i,j=1,\ldots,n$, the vector field $[X_{i}^{(a)},X_{j}^{(a)}]$ is 
homogeneous of degree $-(w_{i}+w_{j})$. Therefore, for all $t>0$, we have
\begin{equation*}
   \delta_{t}^{*}[X_{i}^{(a)},X_{j}^{(a)}]=t^{-(w_{i}+w_{j})}[X_{i}^{(a)},X_{j}^{(a)}] = 
   [\delta_{t}^{*}X_{i}^{(a)},\delta_{t}^{*}X_{j}^{(a)}]. 
\end{equation*}
It then follows that the dilations $\delta_{t}^{*}$, $t>0$, induce Lie algebra automorphisms of $\tilde{\fg}^{(a)}$. 

We realize $\tilde{\fg}^{(a)}$ as the Lie algebra of left-invariant vector fields on a graded nilpotent Lie group as follows.  Let $U$ be the range of the privileged coordinates $(x_1,\ldots, x_n)$. As there are linearly adapted at $a$ to $(X_1,\ldots, X_n)$, in these coordinates we can write
\begin{equation*}
 X_j= \partial_j + \sum_{1\leq k \leq n}b_{jk}(x)\partial_{x_k}, \qquad b_{jk}(x)\in C^\infty(U),\ b_{jk}(0)=0. 
\end{equation*}
The formula~(\ref{eq-homo-app}) then expresses each model vector field $X_j^{(a)}$, $j=1,\ldots, n$, in terms of the partial derivatives $\partial^\alpha b_{jk}(0)$ with $\brak{\alpha}+w_j=w_k$ and $w_k>w_j$. 

\begin{lemma}\label{lem:nilpotent.flowX}
 Let $X=\sum \xi_j X_j^{(a)}$, $\xi_j\in \R$, be a vector field in $\tilde{\fg}^{(a)}$. For every $y\in \R^n$, the flow $x(t):=\exp(tX)(y)$ is defined for all $t\in \R$. Moreover, it takes the form, 
\begin{equation}
 x_k(t)= y_k+ t\xi_k  + \sum_{\substack{\brak\alpha+\brak\beta=w_{k}\\ |\alpha|+|\beta|\geq 2}} \hat{c}_{k\alpha\beta} y^\alpha(t\xi)^\beta , \qquad k=1,\ldots, n.  
\label{eq:Carnot-coord.xk(t)}
\end{equation}
where $\hat{c}_{k\alpha\beta}$ is a universal polynomial $\hat{\Gamma}_{k\alpha\beta}(\partial^\gamma b_{jl}(0))$ in the partial derivatives $\partial^\gamma b_{jl}(0)$ with 
$w_j+\brak\gamma=w_l\leq w_k$ and $w_j\leq w_l$. Each polynomial $\hat{\Gamma}_{k\alpha\beta}$ is 
determined recursively by the polynomials $\hat{\Gamma}_{j\gamma \delta}$ with $w_j<w_k$  (see~(\ref{eq:nilpotent.Gammakalphbeta-recursive1})--(\ref{eq:nilpotent.Gammakalphbeta-recursive2}) \textit{infra}). 
\end{lemma}
\begin{proof}
 Let $X=\sum \xi_j X_j^{(a)}$, $\xi_j\in \R$, be a vector field in $\tilde{\fg}^{(a)}$. The flow $x(t)=\exp(tX)(y)$ is the solution of the ODE system, 
     \begin{equation}
        \dot{x}(t)=X\left(x(t)\right), \qquad x(0)=y.
        \label{eq:nilpotent.flowX}
    \end{equation}
 Set $b_{jk\alpha}=(\alpha!)^{-1}\partial^\alpha b_{jk}(0)$. Then using~(\ref{eq-homo-app}) we get 
   \begin{equation*}
     X= \sum_{1\leq j \leq n}\xi_{j}X_j^{(a)}= \sum_{k=1}^{n}\biggl( \xi_{k}+
     \sum_{\substack{w_{j}+\brak\alpha=w_{k}\\ w_k>w_j}} \xi_{j} b_{jk\alpha} x^\alpha \biggr) \partial_{x_k} .
   \end{equation*}
 We also observe that if $w_{j}+\brak\alpha=w_{k}$, then $\brak\alpha<w_{k}$, and so $x^{\alpha}$ must be a monomial  
   in the components $x_{l}$ with $w_{l}<w_{k}$.  
  Therefore, setting $x(t)=(x_{1}(t),\ldots,x_{n}(t))$ the equation $ \dot{x}(t)=X(t)$ can be rewritten in the form, 
   \begin{equation}
       \dot{x}_{k}(t)  =\xi_{k}+ \sum_{\substack{w_{j}+\brak\alpha=w_{k}\\
     w_k>w_j}} \xi_{k} b_{jk\alpha} \prod_{w_l<w_k} x_l(t)^{\alpha_l},       \qquad k=1,\ldots,n. 
     \label{eq:Carnot-coord.ODE-system}
   \end{equation}
We thus get a triangular ODE system that can be solved recursively. Combining this with the initial condition $x(0)=y$, we then see that the solution of~(\ref{eq:nilpotent.flowX}) is given by the recursive relations, 
\begin{gather}
 x_k(t)=y_k +t\xi_k \qquad \text{if $w_k=1$},
 \label{eq:nilpotent.Gammakalphbeta-recursive1}\\
 x_k(t)= y_k + t\xi_k +  \sum_{\substack{w_{j}+\brak\alpha=w_{k}\\
    w_k>w_j}} \xi_{j} b_{jk\alpha} \int_0^t\prod_{w_l<w_k}x_l(s)^{\alpha_l}ds \qquad \text{if $w_k\geq 2$}. 
    \label{eq:nilpotent.Gammakalphbeta-recursive2}
\end{gather}
The solution exists for all $t\in \R$. Moreover, an induction on $w_k$ shows that every component $x_k(t)$ is a polynomial in $t\xi$ and $y$ of the form~(\ref{eq:Carnot-coord.xk(t)}), where each coefficient $\hat{c}_{k\alpha\beta}$ is a universal polynomial in the coefficients $b_{jl\gamma}$ with 
$w_j+\brak\gamma=w_l\leq w_k$ and $w_j\leq w_l$. This proves the result.  
\end{proof}

Lemma~\ref{lem:nilpotent.flowX} shows that we have a globally defined smooth exponential map $\exp: \tilde{\fg}^{(a)} \rightarrow \R^n$ given by 
\begin{equation}\label{eq-CH}
    \exp(X):=\left.\exp(tX)(0)\right|_{t=1} \qquad \text{for all $X\in \tilde{\fg}^{(a)}$}. 
\end{equation}
Although, $\exp(X)$ \emph{a priori} arises from the solution of an ODE system, it follows from~(\ref{eq:Carnot-coord.xk(t)}) that $\exp(X)$ is determined effectively in terms of  the coordinates of $X$ in the basis $(X_1^{(a)}, \ldots, X_n^{(a)})$ and the coefficients of the model vector fields $X_j^{(a)}$. In fact, if  $X= \sum \xi_j X_j^{(a)}$, $\xi_j\in \R$, then setting $t=1$ and $y=0$ shows that $x=\exp(X)$ is given by
\begin{equation}
 x_k= \xi_k  +   \sum_{\substack{\brak\alpha=w_{k}\\ |\alpha|\geq 2}} \hat{c}_{k\alpha} \xi^\alpha ,\qquad k=1,\ldots, n, 
\label{eq:Carnot-coord.exp(X)k}
\end{equation}
where we have set $\hat{c}_{k\alpha}=\hat{c}_{k0\alpha}$. It also follows from this formula that $\exp: \tilde{\fg}^{(a)}\rightarrow \R^n$ is a diffeomorphism, since it expresses $x=\exp(X)$ as a triangular polynomial map in the coordinates $(\xi_1,\ldots, \xi_n)$ the diagonal of which is the identity map. In addition, we observe that $x_k$, as a polynomial in $\xi$, is homogeneous of degree $w_k$ with respect to the dilations~(\ref{eq:Nilpotent.dilations2}). Thus, for all $t\in \R^*$, we have 
\begin{equation}
 t\cdot \exp(X)= \exp\left( \sum (t\cdot\xi)_j X_j^{(a)}\right) = \exp \left( \sum \xi_j t^{w_j} X_j^{(a)}\right)= \exp\left( \delta_{t^{-1}}^*X\right).
\label{eq:fga-homogeneity-expX}
\end{equation}

We define the Lie group $G^{(a)}$ as $\R^n$ equipped with the group law given by
\begin{equation}
 x\cdot y = \exp(X\cdot Y), \qquad X,Y\in \tilde{\fg}^{(a)},
 \label{eq:fga.group-law-G(a)}
\end{equation}
where $X$ and $Y$ are the unique elements of $\tilde{\fg}^{(a)}$ such that $\exp(X)=x$ and $\exp(Y)=y$, and $X\cdot Y$ is the Dynkin 
product~(\ref{eq:Carnot.Dynkin-product}). For $j=1,\ldots, n$, the vector field $X_j^{(a)}$ generates a one-parameter subgroup $\exp(tX_j^{(a)})$, $t\in \R$, in $G^{(a)}$, and so this is a left-invariant vector field on $G^{(a)}$. As $(X_1^{(a)}, \ldots, X_n^{(a)})$ is a basis of $\tilde{\fg}^{(a)}$, we then arrive at the following statement. 

\begin{proposition}\label{prop:Nilpotent.G(a)-tildefg(a)}
 The Lie algebra of left-invariant vector fields on $G^{(a)}$ is precisely $\tilde{\fg}^{(a)}$. 
\end{proposition}

\begin{definition}\label{def-Ga}
The graded nilpotent Lie group $G^{(a)}$ equipped with its left-invariant Carnot manifold structure is called the \emph{nilpotent approximation} of $(M,H)$ at $a$ with respect to the privileged coordinates $(x_1,\ldots, x_n)$. 
\end{definition}

\begin{remark}
 In what follows we will often abuse language and call $G^{(a)}$ the nilpotent approximation of $(M,H)$ at $a$.   
\end{remark}

\begin{definition}
 $\fg(a)$ is the nilpotent Lie algebra obtained by equipping $T\R^{n}(0)$ with the Lie bracket given by
  \begin{equation}
     [\partial_{i}, \partial_{j}]=\left\{
     \begin{array}{cl}
   {\displaystyle \sum_{w_{k}=w_{i}+w_{j}}L_{ij}^{k}(a)\partial_{k}} & \text{if $w_{i}+w_{j}\leq  r$},   \\
         0 & \text{otherwise}  \\
     \end{array}\right.
     \label{eq:Nilpotent.structure-constants-fg(a)}
 \end{equation}
\end{definition}

\begin{remark}
 The Lie algebra $\fg(a)$ depends only on the structure constants $L_{ij}^k(a)$, $w_i+w_j=w_k$, and so it does not depend on the choice of the privileged coordinates $(x_1,\ldots, x_n)$. 
\end{remark}

\begin{remark}
$\fg(a)$ is a graded nilpotent Lie algebra with respect to the grading,
\begin{equation}
 \fg(a)=\fg_1(a)\oplus \cdots \oplus \fg_r(a), \qquad \text{where}\ \fg_w(a):=\op{Span}\{\partial_j;\ w_j=w\}. 
 \label{eq:Nilpotent.grading-fg(a)}
\end{equation}
\end{remark}

\begin{proposition}\label{prop:Nilpotent.G(a)-fg(a)-graded}
 The Lie algebra of $G^{(a)}$ is precisely $\fg(a)$. Moreover, the dilations~(\ref{eq:Nilpotent.dilations2}) are group automorphisms of $G^{(a)}$.  
\end{proposition}
\begin{proof}
 We know that $\tilde{\fg}^{(a)}$ is the Lie algebra of left-invariant vector fields on $G^{(a)}$. By definition this is the Lie algebra generated by the vector fields $X_1^{(a)}, \ldots, X_n^{(a)}$. These vector fields satisfy the commutator relations~(\ref{eq:Nilpotent.structure-constants-Xj(a)}). As $X_j^{(a)}(0)=\partial_j$  we see that the Lie bracket of the Lie algebra $TG^{(a)}(0)$ satisfies~(\ref{eq:Nilpotent.structure-constants-fg(a)}). Thus, as a Lie algebra $TG^{(a)}(0)$ agrees with $\fg(a)$. 

In addition, using~(\ref{eq:Carnot.BCH-Formula}) and~(\ref{eq:fga-homogeneity-expX}) shows that, for all $X,Y\in \tilde{\fg}^{(a)}$ and $t\in \R$, we have
\begin{gather*}
 \delta_t\left( \exp(X)\cdot \exp(Y)\right) = \delta_t\left( \exp(X\cdot Y)\right) = \exp\left( (\delta_t)_*(X\cdot Y)\right) = \exp\left[ \left((\delta_t)_*X\right)\cdot  \left((\delta_t)_*Y\right)\right],\\
 \left[ \delta_t\left(\exp(X)\right)\right]\cdot \left[ \delta_t\left(\exp(X)\right)\right]= \exp\left((\delta_t)_*X\right) \cdot  \exp\left((\delta_t)_*Y\right)= \exp\left[ \left((\delta_t)_*X\right)\cdot  \left((\delta_t)_*Y\right)\right]. 
\end{gather*}
It then follows that the dilations $\delta_t$, $t\in \R$, are group automorphisms of $G^{(a)}$. The proof is complete. 
\end{proof}

\subsection{The Class $\sN_X(a)$}
In view of Proposition~\ref{prop:Nilpotent.G(a)-fg(a)-graded} it is natural to introduce the following class of nilpotent groups. 

\begin{definition}\label{def:Nilpotent.class-cN(a)}
 $\sN_{X}(a)$ consists of nilpotent groups $G$ that are obtained by equipping $\R^n$ with a group law such that
 \begin{enumerate}
 \item[(i)] The dilations~(\ref{eq:Nilpotent.dilations2}) are group automorphisms of $G$. 
 
 \item[(ii)] The Lie algebra $TG(0)$ of $G$ is precisely $\fg(a)$. 
\end{enumerate}
\end{definition}

\begin{remark}
 The condition (i) automatically implies that the origin is the unit of $G$.
\end{remark}

\begin{remark}\label{rmk:Nilpotent.compatibility-dilations-cN(a)}
 The condition (i) also implies that the dilations $\delta_t$, $t\in \R$, induce a family of dilations $\delta_t'(0)$, $t\in \R$, on $TG(0)=\fg^{(a)}$. In the basis 
 $(\partial_1, \ldots, \partial_n)$ there are given by~(\ref{eq:Nilpotent.dilations2}), and so they agree with the dilations~(\ref{eq:Carnot.dilations}) defined by the grading~(\ref{eq:Nilpotent.grading-fg(a)}) of $\fg(a)$. 
\end{remark}

\begin{definition}
 Suppose that $G$ is a nilpotent Lie group which is built out of $\R^n$ and has the origin has unit. Let $\tilde{\fg}$ be its Lie algebra of left-invariant vector fields. The \emph{canonical basis} of $\tilde{\fg}$ is the basis $(Y_1,\ldots, Y_n)$, where $Y_j$ is the unique left-invariant vector field on such that $Y_j(0)=\partial_j$. 
\end{definition}

It will be convenient to describe the nilpotent Lie groups in the class $\sN_X(a)$ in terms of their Lie algebras of left-invariant vector fields. 

\begin{proposition}\label{prop:Nilpotent.distinguished-basis}
 Suppose that $G$ is a nilpotent Lie group in the class $\sN_X(a)$. Then the canonical basis $(Y_1,\ldots, Y_n)$ of its Lie algebra of left-invariant vector fields has the following properties:
\begin{enumerate}
\item[(i)] For $j=1,\ldots, n$, the vector field $Y_j$ is homogeneous of degree~$-w_j$ with respect to the dilations~(\ref{eq:Nilpotent.dilations2}) and agrees with $\partial_j$ at $x=0$. 

\item[(ii)] The vector fields $Y_1, \ldots, Y_n$ satisfy the commutator relations~(\ref{eq:Nilpotent.structure-constants-Xj(a)})
\end{enumerate}
 
 Conversely, let $(Y_1,\ldots, Y_n)$ be a family of vector fields on $\R^n$  satisfying~(i)--(ii). Then this is the canonical basis of left-invariant vector fields on a unique nilpotent Lie group in the class $\sN_X(a)$. 
\end{proposition}
\begin{proof}
 Suppose that $G$ is a nilpotent Lie group in the class $\sN_X(a)$. Let $(Y_1,\ldots, Y_n)$ be the canonical basis of its Lie algebra of left-invariant vector fields. By definition $Y_j(0)=\partial_j$ for $j=1,\ldots,n$. In addition, by assumption $\fg(a)$ is the Lie algebra $TG(0)$ of $G$. Therefore, it follows from the definition~(\ref{eq:Nilpotent.structure-constants-fg(a)}) of the Lie bracket of $\fg(a)$ that  $Y_1, \ldots, Y_n$ satisfy the commutator relations~(\ref{eq:Nilpotent.structure-constants-Xj(a)}). 
 
As mentioned in Remark~\ref{rmk:Nilpotent.compatibility-dilations-cN(a)}, the dilations $\delta_t$, $t\in \R$, in~(\ref{eq:Nilpotent.dilations2}) induce on $TG(0)=T\R^n(0)$ dilations that agree with the dilations defined by the grading~(\ref{eq:Nilpotent.grading-fg(a)}). Thus, for all $t\in \R^*$ and $j=1,\ldots, n$, we have $\delta_t'(0)\partial_j=t^{-w_j}\partial_j$. Bearing this in mind, let $t\in \R^*$ and $j\in \{1, \ldots, n\}$. As $\delta_t$ is a group automorphism of $G$, the vector field $t^{w_j}\delta_t^* Y_j$ is a left-invariant vector field on $G$. Moreover, at $x=0$ we have 
\begin{equation*}
 \left(t^{w_j}\delta_t^* Y_j\right)(0)= t^w_j \delta_{t}'(0)^{-1}\left[ Y_j(0)\right] =  t^w_j \delta_{t}'(0)^{-1}\partial_j(0)=\partial_j=Y_j(0).
\end{equation*}
Therefore, by left-invariance $t^{w_j}\delta_t^* Y_j=Y_j$. This shows that $Y_j$ is homogeneous of degree~$-w_j$ for all $j=1,\ldots, n$. 

Conversely, let  $(Y_1,\ldots, Y_n)$ be a family of vector fields satisfying (i)--(ii). The property (ii) implies that $Y_1,\ldots, Y_n$ generate a graded nilpotent Lie algebra $\tilde{\fg}$ of vector fields on $\R^n$. Moreover, thanks to the property~(i) we can argue along the same lines as that of the proof of Lemma~\ref{lem:nilpotent.flowX} to show that, for every vector field $Y=\sum \eta_j Y_j$, $\eta_j\in \R$, in $\tilde{\fg}$ and for every $y\in \R^n$, the flow $\exp(tY)(y)$ exists for all times $t\in \R$ and is of the form~(\ref{eq:Carnot-coord.xk(t)}). Therefore, in the same way as in~(\ref{eq-CH}), we have a globally defined exponential map $\exp:\tilde{\fg}\rightarrow \R^n$, which is a smooth diffeomorphism. This allows us to define a group law on $\R^n$ as in~(\ref{eq:fga.group-law-G(a)}). Letting $G$ be $\R^n$ equipped with that group law, we obtain a nilpotent Lie group whose Lie algebra of left-invariant vector fields is $\tilde{\fg}$. Note that the product law of $G$ is uniquely determined by $\tilde{\fg}$, and so it is uniquely determined by $(Y_1,\ldots, Y_n)$.

Obviously $(Y_1,\ldots, Y_n)$ is the canonical basis of left-invariant vector fields on $G$. Moreover, the properties (i)--(ii) allow us to argue along the same lines as that of the proof of Proposition~\ref{prop:Nilpotent.G(a)-fg(a)-graded} to see that the Lie algebra  $TG(0)$ of $G$ agrees with the Lie algebra $\fg(a)$ and the dilations~(\ref{eq:Nilpotent.dilations2}) are group automorphisms of $G$. That is, the group $G$ is in the class $\sN_X(a)$. The proof is complete.  
\end{proof}

\subsection{Getting all nilpotent approximations} It follows from Proposition~\ref{prop:Nilpotent.G(a)-fg(a)-graded} that, given any system of privileged coordinates at $a$ adapted to $(X_1,\ldots, X_n)$, the nilpotent approximation $G^{(a)}$ is in the class $\sN_{X}(a)$. We shall now establish the converse of this result. More precisely, given any graded nilpotent Lie group $G$ in the class $\sN_X(a)$, we will see how to get \emph{all} the systems of privileged coordinates in which the nilpotent approximation is given by $G$. 

In what follows, we let $(x_1,\ldots, x_n)$ be privileged coordinates at $a$ adapted to $(X_1,\ldots, X_n)$. For $j=1,\ldots, n$, we let $X_j^{(a)}$ be the model vector field at $a$ of $X_j$, and we denote by $\tilde{\fg}^{(a)}$ the Lie algebra generated by $X_1^{(a)}, \ldots, X^{(a)}_n$. By Proposition~\ref{prop:Nilpotent.G(a)-tildefg(a)} this is the Lie algebra of left-invariant vector fields on the nilpotent approximation $G^{(a)}$ in the privileged coordinates  $(x_1,\ldots, x_n)$. 

\begin{lemma}\label{lem:many.unicity-phi}
 Let $\phi:\R^{n}\rightarrow \R^{n}$ be a $w$-homogeneous smooth diffeomorphism such that
 \begin{equation}
     \phi_{*}X_{j}^{(a)}=X_{j}^{(a)} \qquad \text{for $j=1,\ldots,n$}.
     \label{eq:many.phi*XX(a)}
 \end{equation}Then $\phi$ is the identity map.
\end{lemma}
\begin{proof}
We note that the $w$-homogeneity and smoothness of $\phi$ imply that $\phi(0)=0$. Therefore, in order to prove that 
$\phi=\op{id}$ we only have to show that
\begin{equation}
    \partial_{j}\phi_{k}(x)=\delta_{jk} \qquad \text{for $j,k=1,\ldots,n$}.
    \label{eq:many.dphi-delta}
\end{equation}
We also observe that~(\ref{eq:many.phi*XX(a)}) means that
\begin{equation}
    \phi'(x)[X_{j}^{(a)}(x)]=X_{j}^{(a)}(\phi(x)) \qquad \text{for $j=1,\ldots,n$}.
    \label{eq:many.phi*XX(a)-bis}
\end{equation}
In particular, setting $x=0$ gives
\begin{equation*}
    \phi'(0)[X_{j}^{(a)}(0)]=\phi'(0)[\partial_{j}]=X_{j}^{(a)}(0)=\partial_{j}. 
\end{equation*}
This shows that $\phi'(0)=\op{id}$. Combining this with the $w$-homogeneity of $\phi$ we deduce that its components $\phi_k(x)$, $k=1,\ldots,n$, are of 
the form,
\begin{equation*}
    \phi_{k}(x)=x_{k}+\sum_{\substack{\brak \alpha=w_{k}\\|\alpha|\geq 2}} c_{k\alpha}x^{\alpha}, \qquad c_{k\alpha}\in \R. 
\end{equation*}
In particular, we see that $\phi_{k}(x)$ does not depend on the variables $x_{j}$ with $w_{j}\geq w_{k}$ 
and $j\neq k$ and its linear component is just $x_{k}$. Thus, 
\begin{equation*}
    \partial_{j}\phi_{k}(x)=\delta_{jk} \qquad \text{for $w_{j}\geq w_{k}$}. 
\end{equation*}In particular, this gives~(\ref{eq:many.dphi-delta}) when $w_{k}-w_{j}\leq 0$. 

We also know by~(\ref{eq-homo-app}) that the vector fields $X_{j}^{(a)}$, $j=1,\ldots,n$, are of the form,
\begin{equation}
    X_{j}^{(a)}=\partial_{j}+\sum_{w_{l}>w_{j}}b_{jl}(x)\partial_{k},
    \label{eq-homo-app2}
\end{equation}where $b_{jl}(x)$ is a linear combination of monomials $x^{\alpha}$ with $\brak \alpha=w_{l}-w_{j}$. In particular, the coefficient
$b_{jl}(x)$ does not depend on the variables $x_{p}$ with $w_{p}>w_{l}-w_{j}$. In addition, using~(\ref{eq-homo-app2}) we get
\begin{align}
 \phi'(x)[X_{j}^{(a)}(x)]& = \phi'(x)\partial_{j}+\sum_{w_{l}>w_{j}}b_{jl}(x) \phi'(x)\partial_{k} \nonumber \\
 & = \sum_{1\leq k \leq n}\biggl( \partial_j\phi_k(x) + \sum_{w_{j}<w_{l}}b_{jl}(x)\partial_l\phi_k(x)\biggr) \partial_k.
 \label{eq:many.phi'X}
\end{align}
Note also that, as $\phi_k(x)$ is polynomial and homogeneous of degre $w_k$, we have $\partial_l\phi_k(x)=0$ when $w_l>w_k$. Therefore, by combining~(\ref{eq:many.phi*XX(a)-bis}) and~(\ref{eq:many.phi'X}) we get
\begin{equation}
    \partial_{j}\phi_{k}(x)+\sum_{w_{j}<w_{l}\leq w_k}b_{jl}(x)\partial_{l}\phi_{k}(x)=b_{jk}(\phi(x)) \qquad \text{when 
    $w_{k}>w_{j}$}. 
    \label{eq:many.phi*XX(a)-coeff}
\end{equation}

We shall now proceed to prove~(\ref{eq:many.dphi-delta}) by induction on $w_{k}-w_{j}$. We already know that~(\ref{eq:many.dphi-delta}) holds when 
$w_{k}-w_{j}\leq 0$. Assume that~(\ref{eq:many.dphi-delta}) holds for $w_{k}-w_{j}<m$ for some $m\in \N$. We remark that if $w_{k}\leq m$, 
then, for all $j=1,\ldots,n$, we have $w_{k}-w_{j}\leq m-1<m$, and so $\partial_{j}\phi_{k}(x)=\delta_{jk}$. As 
$\phi_{k}(0)=0$, we then deduce that
\begin{equation}
    \phi_{k}(x)=x_{k} \qquad \text{when $w_{k}\leq m$}.
    \label{eq:many-phik-xk}
\end{equation}
Let $j$ and $k$ be positive integers~$\leq n$ such that $w_{j}<w_{k}\leq w_{j}+m$. Then~(\ref{eq:many.phi*XX(a)-coeff}) gives
\begin{equation*}
    \partial_{k}\phi_{j}(x)=b_{jk}\left(\phi(x)\right)-\sum_{w_{j}<w_{l}\leq w_k}b_{jl}(x)\partial_{l}\phi_{k}(x).
\end{equation*}
If $w_{j}<w_{l}\leq w_{k}$, then $w_{k}-w_{l}<w_{k}-w_{j}\leq m$, and so $\partial_{l}\phi_{k}(x)=\delta_{lk}$. 
Moreover, as mentioned above, the coefficient $b_{jk}(x)$ depends only on the variables $x_{p}$ with $w_{p}\leq w_{k}-w_{j}\leq m$. 
As~(\ref{eq:many-phik-xk}) ensures us that $\phi_{p}(x)=x_{p}$ for $w_{p}\leq m$, we deduce that $b_{jk}(\phi(x))=b_{jk}(x)$. It then follows that
\begin{equation*}
    \partial_{k}\phi_{j}(x)=b_{jk}\left(\phi(x)\right)-\sum_{w_{j}<w_{l}\leq w_k}b_{jl}(x)\delta_{lk}=b_{jk}(x)-b_{jk}(x)=0.
\end{equation*}
This shows that~(\ref{eq:many.dphi-delta}) is true when $w_{k}-w_{j}\leq m$. It then follows that $\partial_j\phi_k(x)=\delta_{jk}$ for all $j.k=1,\ldots, n$, and hence $\phi$ is the identity map.  The proof is complete. 
\end{proof}

Suppose that $G$ is a nilpotent group in the class $\sN_{X}(a)$. That is, $G$ is the manifold $\R^n$ equipped with a group law satisfying the conditions~(i)--(ii) of Definition~\ref{def:Nilpotent.class-cN(a)}. We denote by $\tilde{g}$ be the Lie algebra of left-invariant vector fields on $G$, and let $(Y_1,\ldots, Y_n)$ be its canonical basis. 
As mentioned in the proof of Proposition~\ref{prop:Nilpotent.distinguished-basis}, in the same way as in~(\ref{eq-CH}) we have a globally defined exponential map $\exp:\tilde{\fg}\rightarrow \R^n$, which is a smooth diffeomorphism. Therefore, we define a smooth diffeomorphism $\exp_Y:\R^n \rightarrow \R^n$ by
\begin{equation*}
 \exp_Y(x) = \exp(x_1Y_1+\ldots+x_nY_n), \qquad x\in \R^n
\end{equation*}
In fact, as also mentioned in the proof of Proposition~\ref{prop:Nilpotent.distinguished-basis}, $\exp_Y(x)$ is of the form~(\ref{eq:Carnot-coord.xk(t)}) with $y=0$ and $t=1$, i.e., it is of the form~(\ref{eq:Carnot-coord.exp(X)k}). Therefore, we obtain a $w$-homogeneous polynomial diffeomorphism such that $\exp_Y'(0)=\op{id}$. 

Likewise, we have a $w$-homogeneous polynomial diffeomorphism $\exp_{X^{(a)}}:\R^n \rightarrow \R^n$ given by
\begin{equation}
 \exp_{X^{(a)}}(x)= \exp\left(x_1X^{(a)}_1+\ldots+x_nX^{(a)}_n\right) \qquad \text{for all $x\in \R^n$}. 
 \label{eq:Nilpotent.expX(a)}
\end{equation}
We then define the map $ \phi_Y:\R^n\rightarrow \R^n$ by 
\begin{equation}
 \phi_Y(x)=\exp_Y\circ \exp_{X^{(a)}}^{-1}(x) \qquad \text{for all $x\in \R^n$}. 
 \label{eq:Nilpotent.phiY}
\end{equation}
Note that $\phi_Y$  is a $w$-homogeneous diffeomorphism of $\R^n$ whose differential at $x=0$ is the identity map, since $\exp_Y $ and $\exp_{X^{(a)}}$ are both such maps.

\begin{proposition}
 The diffeomorphism $\phi_Y$ is a group isomorphism from $G^{(a)}$ onto $G$. Moreover, this is the unique $w$-homogeneous diffeomorphism of $\R^n$ such that
     \begin{equation}
        \left(\phi_Y\right)_{*}X_{j}^{(a)}=Y_{j} \qquad \text{for $j=1,\ldots,n$}. 
        \label{eq:many.phi*XY}
    \end{equation}
\end{proposition}
\begin{proof}
As mentioned above, $\phi_Y$ is a $w$-homogeneous diffeomorphism of $\R^n$ whose differential at $x=0$ is the identity map. We also know that $(X_1^{(a)}, \ldots, X_n^{(a)})$ and $(Y_1,\ldots, Y_n)$ satisfy the same commutator relations~(\ref{eq:Nilpotent.structure-constants-Xj(a)}). Therefore, we have a Lie algebra isomorphism $\lie:\tilde{\fg}^{(a)}\rightarrow \tilde{\fg}$ given by
    \begin{equation*}
        \lie\left(x_{1}X_{1}^{(a)}+\cdots +x_{n}X_{n}^{(a)}\right)=x_{1}Y_{1}+\cdots + x_{n}Y_{n}, \qquad x_{j}\in 
        \R.
    \end{equation*} 
In particular, this is a Lie group map with respect to the Dynkin products~(\ref{eq:Carnot.Dynkin-product}) on $\tilde{\fg}^{(a)}$ and $\tilde{\fg}$. 

We also observe that $\exp_Y$ is the exponential map $\exp_{\tilde{\fg}}:\tilde{\fg}\rightarrow G$ in the coordinates defined by the basis $(Y_1, \ldots, Y_n)$. Likewise, the map $\exp_{X^{(a)}}$ is the exponential map $\exp_{\tilde{\fg}^{(a)}}:\tilde{\fg}^{(a)}\rightarrow G^{(a)}$ in the coordinates defined by the basis $(X_1^{(a)}, \ldots, X_n^{(a)})$. Thus, 
 \begin{equation*}
\phi_Y= \exp_Y\circ \exp_{X^{(a)}}^{-1}= \exp_{\tilde{\fg}} \circ \chi \circ \exp_{\tilde{\fg}^{(a)}}^{-1}.
\end{equation*}

By the Baker-Campbell-Hausdorff formula~(\ref{eq:Carnot.BCH-Formula}) the exponential map $\exp_{\tilde{\fg}}$ (resp., $\exp_{\tilde{\fg}^{(a)}}$) intertwines the Dynkin product~(\ref{eq:Carnot.Dynkin-product}) on ${\tilde{\fg}}$ (resp, $\tilde{\fg}^{(a)}$) with the product of $G$ (resp., $G^{(a)})$. It then follows that $\phi_Y$ is given by the composition of maps of Lie groups, and so this is a Lie group isomorphism from $G^{(a)}$ onto $G$. 

Recall that, for $j=1, \ldots, n$, the vector field $Y_j$ (resp., $X_j^{(a)}$) is the unique left-invariant vector field on $G$ (resp., $G^{(a)}$) that agrees with $\partial_j$ at $x=0$. As $\phi_Y$ is a Lie group isomorphism and $\phi_Y'(0)=\op{id}$, the vector field $(\phi_Y)_*X_j^{(a)}$ is a left-invariant vector field on $G$ which at $x=0$ is equal to 
$(\phi_Y)_*X_j^{(a)}(0)=\phi_Y'(0)(X_j^{(a)}(0))=X_j^{(a)}(0)=\partial_j$. It then follows that $(\phi_Y)_{*}X_{j}^{(a)}=Y_{j}$ for $j=1,\ldots,n$. 

It remains to show that $\phi_Y$ is the unique $w$-homogeneous diffeomorphism of $\R^n$ satisfying~(\ref{eq:many.phi*XY}).
Let $\psi: \R^n\rightarrow \R^n$ be another such diffeomorphism. Then $\phi^{-1}_Y\circ \psi$ is a $w$-homogeneous smooth diffeomorphism. Moreover,  for $j=1,\ldots, n$, we have
\begin{equation*}
(\phi^{-1}_Y\circ \psi)_{*}X_{j}^{(a)}=\phi^{*}_Y(\psi_{*}X_{j}^{(a)})=\phi^{*}_YY_{j}=X_{j}^{(a)}.
\end{equation*}
It then follows from Lemma~\ref{lem:many.unicity-phi} that $\phi^{-1}_Y\circ \psi=\op{id}$, i.e., $\psi=\phi_Y$. Thus, $\phi_Y$ is the unique $w$-homogeneous diffeomorphism of $\R^n$ satisfying~(\ref{eq:many.phi*XY}). The proof is complete. 
\end{proof}

We are now in a position to prove the following result. 

\begin{theorem}\label{thm:Nilpotent.coord-G(a)G}
 Suppose that $G$ is a nilpotent Lie group in the class $\sN_{X}(a)$.  Let $(x_1,\ldots, x_n)$ be privileged coordinates at $a$ adapted to $(X_1,\ldots, X_n)$. Then a change of coordinates $x\rightarrow \phi(x)$ produces privileged coordinates at $a$ adapted to $(X_1,\ldots, X_n)$ in which the nilpotent approximation is $G$  if and only if, near $x=0$, we have
\begin{equation}
 \phi(x) =\phi_Y(x) + \Ow\left(\|x\|^{w+1}\right),
 \label{eq:nilp-approx.phi-phiY-Ow} 
\end{equation}
where $\phi_Y$ is defined in~(\ref{eq:Nilpotent.phiY}). In particular, $x\rightarrow \phi_Y(x)$ is the unique such change of coordinates which is $w$-homogeneous.  
\end{theorem}
\begin{proof}
Let $x\rightarrow \phi(x)$ be a change of coordinates.  By Proposition~\ref{prop:char-priv.phi-hatphi-Ow} the new coordinates $(y_1,\ldots, y_n)=\phi(x)$ are privileged coordinates at $a$ adapted to 
$(X_1,\ldots, X_n)$ if and only if
\begin{equation}
 \phi(x) =\hat{\phi}(x) +\Ow\left( \|x\|^{w+1}\right) \qquad \text{near $x=0$},  
  \label{eq:nilp-approx.phi-hatphi-Ow2} 
\end{equation}
where $\hat{\phi}(x)$ is a polynomial $w$-homogeneous map such that $\phi'(0)=\op{id}$. We observe that~(\ref{eq:nilp-approx.phi-phiY-Ow}) is a special case of such an asymptotics, since $\phi_Y(x)$ is a $w$-homogeneous polynomial map such that $\phi_Y'(0)=\op{id}$. Therefore, we may assume that $\phi(x)$ has a behavior of the form~(\ref{eq:nilp-approx.phi-hatphi-Ow2}). 

Let  $U$ be the range of the privileged coordinates $(x_1,\ldots, x_n)$, and set $V=\phi(U)$. If $\phi(x)$ has a behavior of the form~(\ref{eq:nilp-approx.phi-hatphi-Ow2}) near $x=0$, then using~(\ref{eq:char-priv.dtphiXj}) we see that, for $j=1,\ldots, n$ and as $t\rightarrow 0$, we have 
\begin{equation*}
t^{w_j} \delta_t^* (\tilde{\phi}_*X_j)(y)=\left[ \hat{\phi}_*X_j^{(a)}\right](y) +\op{O}(t) \qquad \text{ in $\cX(V)$}. 
\end{equation*}
This shows that $\hat{\phi}_*X_j^{(a)}$ is the model vector field of $X_j$ in the privileged coordinates $(y_1,\ldots, y_n)$. We know by Proposition~\ref{prop:Nilpotent.distinguished-basis} that the nilpotent Lie groups in the class $\sN_X(a)$ are uniquely determined by the canonical bases of their Lie algebras of left-invariant vector fields. Therefore, we see that 
the  nilpotent approximation in the privileged coordinates $(y_1,\ldots, y_n)$ is given by $G$ if and only if 
$\hat{\phi}_*X_j^{(a)}=Y_j$ for $j=1, \ldots, n$.  
By Proposition~\ref{prop:char-priv.phi-hatphi-Ow} this happens if and only if $\hat\phi=\phi_Y$, i.e., the map $\phi(x)$ has a behavior of the form~(\ref{eq:nilp-approx.phi-phiY-Ow}) near $x=0$. 

Finally, note that in~(\ref{eq:nilp-approx.phi-hatphi-Ow2}) the $w$-homogeneous map $\hat{\phi}(x)$ is uniquely determined by $\phi$ (\emph{cf}.~Proposition~\ref{lem:multi.Thetat} and Remark~\ref{rmk:anisotropic.weighted-asymptotic}). 
Therefore, if $\phi(x)$  is $w$-homogeneous and has a behavior of the form~(\ref{eq:nilp-approx.phi-phiY-Ow}), then it must agree with $\phi_Y(x)$ everywhere. Thus, $x\rightarrow \phi_Y(x)$ is the unique $w$-homogeneous change of coordinates that provide us with privileged coordinates in which the nilpotent approximation is given by $G$. The proof is complete. 
\end{proof}

Combining Theorem~\ref{thm:Nilpotent.coord-G(a)G} with Proposition~\ref{prop:Nilpotent.G(a)-fg(a)-graded} we arrive at the following statement. 

\begin{corollary}\label{cor:nilpotent-approx.characterization}
 Let $G$ be a graded nilpotent Lie group of step~$r$ built out of $\R^n$. Then the following are equivalent:
 \begin{enumerate}
\item[(i)] $G$ provides us with the nilpotent approximation of $(M,H)$ at $a$ in some privileged coordinates at $a$ adapted to 
$(X_1,\ldots, X_n)$.

\item[(ii)] $G$ belongs to the class $\sN_{X}(a)$. 
\end{enumerate}
\end{corollary}

 \subsection{Example: Nilpotent approximations of a Heisenberg manifold}   Let us illustrate  Theorem~\ref{thm:Nilpotent.coord-G(a)G} and 
 Corollary~\ref{cor:nilpotent-approx.characterization} in the case of a Heisenberg manifold $(M^n,H)$, where $H\subset TM$ is a hyperplane bundle that gives rise to the step~2 Carnot filtration 
 $(H,TM)$. Let $(X_1,\ldots, X_n)$ be an $H$-frame near a given point $a\in M$. Then the class $\sN_{X}(a)$ is uniquely determined by the coefficients $L_{ij}(a):=L^n_{ij}(a)$, $i,j=1,\ldots, n-1$, such that
\begin{equation*}
 [X_i,X_j](a)= L_{ij}(a) X_n \qquad \bmod H(a), \qquad i,j=1, \ldots, n-1. 
\end{equation*}
 In addition, in this setup the dilations~(\ref{eq:Nilpotent.dilations2}) are given by 
\begin{equation}
 t\cdot x=(tx_1, \ldots, tx_{n-1}, t^2x_n), \qquad x\in \R^n, \ t\in \R. 
 \label{eq:Nilpotent.dilations-Heisenberg}
\end{equation}

 An example of group in the class $\sN_{X}(a)$ is the group $G^0$ which is obtained by equipping $\R^n$ with the group law, 
\begin{equation*}
    x\cdot y = \biggl( x_{1}+y_{1},\ldots,x_{n-1}+y_{n-1}, x_n+y_n+\frac{1}{2}L(a)(x,y) \biggr), \qquad x,y\in \R^n, 
 \end{equation*}
 where we have set $L(a)(x,y)= \sum_{i,j=1}^{n-1} L_{ji}(a)x_iy_j$.  The canonical basis $(Y_1^0,\ldots, Y_n^0)$ of left-invariant vector fields on $G^0$ is given by
\begin{equation*}
 Y_n^0=\partial_{x_n}, \qquad Y_j^0= \partial_{x_j} + \frac{1}{2}\sum_{1\leq k \leq n-1} L_{kj}(a) x_k\partial_{x_n}, \quad j=1,\ldots, n-1. 
\end{equation*}
 Note that $Y_1^0,\ldots, Y^0_{n-1}$ are homogeneous of degree~$-1$ with respect to the dilations~(\ref{eq:Nilpotent.dilations-Heisenberg}), while $Y_n^0$ is homogeneous of degree~$-2$. Moreover, we have the commutator relations, 
\begin{equation}
 \left[Y_i^0, Y_j^0\right]=L_{ij}(a)Y_n^0, \qquad  \left[Y_i^0, Y_n^0\right]=0, \qquad i,j=1,\ldots, n-1. 
 \label{eq:Nilpotent.Heisenberg-relations}
\end{equation}
 
 Let $G$ be another nilpotent Lie group in the class $\sN_{X}(a)$ and denote by  $\tilde{\fg}$ its Lie algebra of left-invariant vector fields. We know by Proposition~\ref{prop:Nilpotent.distinguished-basis} that $G$ is uniquely determined by the canonical basis $(Y_1, \ldots, Y_n)$ of $\tilde{\fg}$.  By Proposition~\ref{prop:Nilpotent.distinguished-basis} this basis has the following properties: 
 \begin{enumerate}
\item[(i)] $Y_j(0)=\partial_j$ for $j=1,\ldots,n$. 

\item[(ii)] $Y_1, \ldots, Y_{n-1}$ are homogeneous of degree~$-1$ with respect to the dilations~(\ref{eq:Nilpotent.dilations-Heisenberg}), while $Y_n^0$ is homogeneous of degree~$-2$.

\item[(iii)] $Y_1, \ldots, Y_{n}$ satisfy the commutator relations~(\ref{eq:Nilpotent.Heisenberg-relations}). 
\end{enumerate}
Note that the properties (i)--(ii) imply that $Y_n=\partial_{x_n}=Y_n^0$. Moreover, for $j=1, \ldots, n-1$, we can write $Y_j=Y_j^0+Z_j$, where $Z_j$ is homogeneous of degree~$-1$ and vanishes at the origin. Thus, it takes the form, 
\begin{equation}
 Z_j= \sum_{1\leq k \leq n-1}b_{jk}x_k \partial_{x_n}, \qquad b_{jk}\in \R. 
 \label{eq:Nilpotent.Zj}
\end{equation}
  We observe that, for $i,j=1, \ldots, n-1$, we have $[Z_i, Y_n]=[Z_i,Z_j]=0$, and 
\begin{equation*}
[Y_i^0,Z_j]=\sum_{1\leq k \leq n-1} b_{jk}[\partial_{x_j},x_k]\partial_{x_n} = b_{ji} \partial_{x_n}. 
\end{equation*}
Therefore, we have 
\begin{equation*}
 [Y_j,Y_j] = \left[Y_i^0,Y_j^0\right] + \left[Y_i^0,Z_j\right]+ \left[Z_i,Y_j^0\right] 
  = L_{ij}(a) \partial_{x_n} + (b_{ji}-b_{ij})\partial_{x_n}. 
\end{equation*}
 Thus, the vector fields $Y_1, \ldots, Y_n$ satisfy the commutator relations~(\ref{eq:Nilpotent.Heisenberg-relations}) if and only if $b_{ij}=b_{ji}$ for $i,j=1,\ldots,n-1$, i.e., the matrix $b=(b_{ij})$ is symmetric. 
 
 Conversely, let $b=(b_{ij})$ be a real symmetric $(n-1)\times (n-1)$-matrix. For $j=1,\ldots, n-1$, set $Y_j=Y_j^0+Z_j$, where $Z_j$ is given by~(\ref{eq:Nilpotent.Zj}). In addition, set $Y_n=\partial_{x_n}$. Then $(Y_1, \ldots, Y_n)$ satisfy the conditions~(i)--(iii) above. Therefore, by Proposition~\ref{prop:Nilpotent.distinguished-basis} this 
  is the canonical basis of the Lie algebra of left-invariant vectors on a unique nilpotent group in the class $\sN_X(a)$. In fact, this group is obtained by equipping $\R^n$ with the group law, 
 \begin{equation*}
    x\cdot y = \biggl( x_{1}+y_{1},\ldots,x_{n-1}+y_{n-1}, x_n+y_n+B(x,y) \biggr), \qquad x,y\in \R^n, 
 \end{equation*}
 where we have set
\begin{equation*}
 B(x,y)= \frac{1}{2}\sum_{1\leq i,j\leq n-1} L_{ji}(a)x_iy_j+  \sum_{1\leq i,j\leq n-1} b_{ij}x_iy_j. 
\end{equation*}
 This shows that the class $\sN_{X}(a)$ is parametrized by the space of real symmetric $(n-1)\times (n-1)$-matrices. Combining this with Corollary~\ref{cor:nilpotent-approx.characterization} we then arrive at the following statement. 

\begin{proposition}\label{prop:nilpotent-approx.Heisenberg-approx}
 Let $(M^n,H)$ be a Heisenberg manifold and $(X_1, \ldots, X_n)$ an $H$-frame near a point $a\in M$. Then there is a one-to-one correspondance between real symmetric $(n-1)\times (n-1)$-matrices and the nilpotent approximations of $(M,H)$ at $a$ associated with privileged coordinates at $a$ adapted to $(X_1, \ldots, X_n)$. 
\end{proposition}

More generally, let $(M,H)$ be a step $r$ Carnot manifold of step~$r$ with weight sequence $(w_1, \ldots, w_n)$  and type $(m_1, \ldots, m_r)$ (where $m_j=\op{rk}H_j$). By elaborating on the previous considerations, it can be shown that every class $\sN_{X}(a)$ contains a subclass which is parametrized by a subspace of real $n\times n$-matrices, the dimension of which is equal to
\begin{equation*}
\frac{m_r}2 \cdot \sum_{\substack{w_i+w_j=r\\ w_i<w_j}} m_i m_j + \frac{m_r}2\cdot \sum_{w_i=\frac12 r} m_i(m_i+1). 
\end{equation*}
More precisely, suppose that $G$ is a nilpotent Lie group in the class $\sN_{X}(a)$. Let $(Y_1,\ldots, Y_n)$ be the canonical basis of its Lie algebra of left-invariant vector fields. Then we get a family of other nilpotent Lie groups in $\sN_{X}(a)$ that are associated with the Lie algebras of vector fields generated by families $(Y_1+Z_1, \ldots, Y_n+Z_n)$, where the vector fields $Z_j$, $j=1,\ldots, n$, are of the form, 
\begin{equation*}
 Z_j= \sum_{w_i+w_j=w_k=r} b_{ij}^kx_i\partial_{x_k}, \qquad b_{ij}^k=b_{ji}^k\in \R. 
\end{equation*}
Like in the Heisenberg manifold case, when $r=2$ this produces all the nilpotent approximations associated with privileged coordinates at $a$ adapted to $(X_1, \ldots, X_n)$. 
In any case, we obtain a very large class of nilpotent approximations at any given point of $M$.

\section{Canonical coordinates}\label{sec:Canonical-coord}
In this section, we explain how to use Proposition~\ref{prop:char-priv.phi-hatphi-Ow} to recover the facts that the canonical privileged coordinates of the first kind of~\cite{Go:LNM76, RS:ActaMath76} and the canonical coordinates of the second kind of~\cite{BS:SIAMJCO90, He:SIAMR} are privileged coordinates in the sense considered in this paper. 
The approach is based on the observation that in privileged coordinates the passages to the canonical coordinates of the 1st and 2nd kind on a Carnot manifold are suitably approximated by the passages to the canonical coordinates on the nilpotent approximation (see Proposition~\ref{prop:can.1stkind-approx} and Proposition~\ref{prop:can.2ndkind-approx} below for the precise statements). 
This approach has the advantage of avoiding technical manipulations with flows of vector fields, but it presupposes the existence of privileged coordinates. 

Throughout this section we let $(X_1,\ldots, X_n)$ be an $H$-frame near a given point $a\in M$. 

\subsection{Canonical coordinates of the first kind}
Let $U_0$ be an open subset of $M$ over which the $H$-frame $(X_1,\ldots, X_n)$ is defined. Given any vector field $X$ on $U_0$ and any point $y_0\in U$, the flow $\exp(tX)(y_0)$ is the solution of the initial-value problem, 
\begin{equation*}
 \dot{y}(s)=X\left( y(s)\right), \qquad y(0)=y_0. 
\end{equation*}
When the solution of this initial-value problem is defined on a given interval $I$ containing $0$ we shall say that the flow 
$\exp(tX)(y_0)$ \emph{exists for all $s\in I$}. The only result on vector field flows that we need is the following lemma, which follows from standard ODE theory (see, e.g., \cite[Theorem~2.3.2]{Pe:Springer01}). 

\begin{lemma}\label{lem:Can.parameter-flows}
 Let $(X(u))_{u\in \cU}$ be a $C^\infty$-family of (smooth) vector fields on $U_0$ parametrized by an open set of some Euclidean space $\R^N$. For every $u_0\in \cU$ and every $y_0\in U_0$, there are $c>0$, an open neighborhood $\cV$ of $u_0$ in $\cU$, and an open neighborhood $V$ of $y_0$ in $U_0$ such that: 
 \begin{enumerate}
\item[(i)] The flow $\exp(sX(u))(y)$ exists for all $(s,u,y)$ in $(-c,c)\times \cV\times V$. 

\item[(ii)] The map $(s,u,y)\ni (-c,c)\times \cV\times V\rightarrow \exp(sX(u))(y)\in M$ is smooth. 
\end{enumerate}
\end{lemma}

Suppose now that $(x_1,\ldots, x_n)$ are local coordinates centered at $a$ adapted to the $H$-frame $(X_1,\ldots, X_n)$. We denote by $U$ the range of these coordinates. This is an open neighborhood of $0\in \R^n$. For sake of simplicity we also assume that $U$ is \emph{$w$-balanced} in the sense that $\delta_t(U)\subset U$ for all $t\in [0,1]$. For instance, the cubes $\prod_{j=1}^n (-c^{w_j},c^{w_j})$, $c>0$, are $w$-balanced. As these cubes form a basis of neighborhoods of the origin, there is no loss of generality in assuming that $U$ is $w$-balanced. 

As $(x_1,\ldots, x_n)$ are privileged coordinates at $a$ adapted to $(X_1,\ldots, X_n)$, we know by Theorem~\ref{thm-pri-equiv} that, for $j=1, \ldots, n$, the vector field $X_j$ has weight $-w_j$. Let $X_j^{(a)}$ be its model vector field. For $t\in [-1,1]$, we let $\hat{X}_j(t)$ be the vector field 
defined by 
\begin{equation}
\hat{X}_j(t)=\left\{
\begin{array}{cl}
t^{w_j}\delta_t^*X_j  & \text{if $0<|t\leq 1$},   \\
 X_j^{(a)}  & \text{if $t=0$}.  
 \end{array}
\right.
\label{eq:Can.hXjt}
\end{equation}

\begin{lemma}\label{lem:Can.hXjt}
 For $j=1,\ldots, n$, the family $(\hat{X}_j(t))_{|t|\leq 1}$ is a $C^\infty$-family of $C^\infty$-vector fields on $U$.  
\end{lemma}
\begin{proof}
 Set $X_j=\sum a_{jk}(x) \partial_{x_k}$, $a_{jk}(x)\in C^\infty(U)$, and $\cU=\{(x,t)\in U \times \R; \ t\cdot x \in U\}$. Note that $\cU$ contains $[-1,1]\times U$. 
 Moreover, as $X_j$ has weight $w_j$, for each $k=1,\ldots, n$, the coefficient $a_{jk}(x)$ has weight $\geq w_k-w_j$. Thus, by Lemma~\ref{lem:anisotropic.weight} there is a function $\Theta_{jk}(x,t)\in C^\infty(\cU)$ such that $a_{jk}(t\cdot x)=t^{w_k-w_j} \Theta_{jk}(x,t)$ for all $(x,t)\in \cU$, $t\neq 0$. 
 Combining this with~(\ref{eq:Anistropic.deltat*P}) we see that, for all $(x,t)\in \cU$ with $t\neq 0$, we have 
\begin{equation}
 t^{w_j}\delta_t^*X_j=\sum_{1\leq j \leq n} t^{w_j-w_k} a_{jk}(t\cdot x) \partial_{x_k}= \sum_{1\leq j \leq n}  \Theta_{jk}(x,t) \partial_{x_k}.
 \label{eq:Can.twjdtXj} 
\end{equation}
The $C^\infty$-regularity near $t=0$ of the functions $ \Theta_{jk}(x,t)$ then implies that, as $t\rightarrow 0$, we have
\begin{equation*}
 t^{w_j}\delta_t^*X= \sum_{1\leq k \leq n}  \Theta_{jk}(x,0) \partial_{x_k} +\op{O}(t) \qquad \text{in $\cX(U)$}. 
\end{equation*}
In view of~(\ref{eq-t0X}) this implies that $X^{(a)}_j=\sum  \Theta_{jk}(x,0) \partial_{x_k}$. Combining this with~(\ref{eq:Can.hXjt}) and~(\ref{eq:Can.twjdtXj}) we then 
deduce that on $U$ we have
\begin{equation*}
\hat{X}_j(t)= \sum_{1\leq jk\leq n}  \Theta_{jk}(x,t) \partial_{x_k} \qquad \text{for all $t\in [-1,1]$}. 
\end{equation*}
As the functions $ \Theta_{jk}(t,x)$ are smooth on $\cU\supset [-1,1]\times U$, this shows that $(\hat{X}_j(t))_{|t|\leq 1}$ is a $C^\infty$-family of smooth 
vector fields on $U$. The proof is complete. 
\end{proof}

For $x\in \R^n$ and $t\in [-1,1]$, we set 
\begin{equation*}
 \hat{X}(t,x)=x_1\hat{X}_1(t) +\cdots +x_n \hat{X}_n(t). 
\end{equation*}
It follows from Lemma~\ref{lem:Can.hXjt} that this defines a $C^\infty$-family of vector fields on $U$ parametrized by $[-1,1]\times \R^n$. In particular, for every $x\in \R^n$, this provides us with a smooth homotopy between $\hat{X}(1,x)=x_1X_1+\cdots + x_nX_n$ and $\hat{X}(0,x)=x_1X_1^{(a)}+\cdots +x_nX_n^{(a)}$.

\begin{lemma}\label{lem:Can.expshXtx}
 There are open neighborhoods $V$ and $W$ of the origin $0\in \R^n$ with $W\subset U$,  such that the flow $\exp(s \hat{X}(t,x))(y)$ exists for all $s,t\in [-1,1]$ and $(x,y)\in V\times W$ and depends smoothly on these parameters. 
\end{lemma}
\begin{proof}
 As $(\hat{X}(t,x))$ is a smooth family of vector fields, it follows from Lemma~\ref{lem:Can.parameter-flows} that, there are $c >0$ and neighborhoods $V$ and $W$ of the origin $0\in \R^n$ with $W\subset U$ such that, the flow $\exp(s \hat{X}(t,x))(y)$ exists for all $s,t\in [-c,c]$ and $(x,y)\in V\times W$ and depends smooth on these parameters. In addition, let $\lambda\in (0,1)$. For $j=1,\ldots, n$, the homogeneity of $X_j^{(a)}$ implies that $\delta_\lambda^*X_j^{(a)}=\lambda^{-w_j}X_j^{(a)}$. Moreover, for all $t\in \R^*$, we have $\delta_\lambda^*(t^{w_j}\delta_t^*X_j)=\lambda^{-w_j} (\lambda t)^{w_j}(\delta_{\lambda t})^*X_j$. Thus, for all $t\in [-1,1]$, we have $\delta_\lambda^* [\hat{X}_j(t)]= \lambda^{-w_j}\hat{X}_j(\lambda t)$, i.e., $\hat{X}_j(t)= \lambda^{-w_j} (\delta_\lambda)_*[ \hat{X}_j(\lambda t)]$. This implies that, for all $s\in \R$ and $x\in \R^n$, we have
 \begin{equation*}
 s\hat{X}(t,x)= \sum_{1\leq j \leq n} s x_j \lambda^{-w_j} (\delta_\lambda)_*\left[ \hat{X}_j(\lambda t)\right] = 
  \sum_{1\leq j \leq n} \lambda s\left[ \lambda^{-1}\cdot (\lambda^{-1} x)\right]_j (\delta_\lambda)_*\left[ \hat{X}_j(\lambda t)\right].   
\end{equation*}
That is, we have
 \begin{equation}
 s\hat{X}(t,x)=   \lambda s (\delta_\lambda)_*\left[\hat{X}\left(\lambda t, \lambda^{-1}\cdot (\lambda^{-1}x)\right)\right] \qquad 
 \text{for all $s\in \R$ and $x\in \R^n$}. 
 \label{eq:Can.shXtx}
\end{equation}

Suppose that $c<1$ and set $\hat{x}= c^{-1}\cdot (c^{-1}x)$. Using~(\ref{eq:Can.shXtx}) we deduce that,  for all $s,t\in [-c,c]$ and $(x,y)\in V\times W$ such that $\hat{x}\in V$ and $c^{-1}\cdot y\in W$, we have 
\begin{equation}
 \exp\left( s\hat{X}(t,x)\right)\!\!(y) =  \exp \left(  c s(\delta_c)_*\hat{X}(c t, \hat{x})\right)\!\!(y)
  = c \cdot 
  \exp\left( c s\hat{X}(c t, \hat{x})\right)\!(c^{-1}\cdot y). 
  \label{eq:Can.expshXtx}
\end{equation}
Note that the flow $ \exp( c s\hat{X}(c t, \hat{x}))\!(c^{-1}\cdot y)$ actually exists for all $s,t\in [-1,1]$. Therefore, if we set $V'=V\cap (c\delta_c(V))$ and $W'=\delta_c(W)$, then~(\ref{eq:Can.expshXtx}) shows that the flow $ \exp( s\hat{X}(t,x))(y)$ exists for all $s,t\in [-1,1]$ and $(x,y)\in V'\times W'$ and depends smoothly on these parameters. This proves the result.
 \end{proof}

Let $V$ and $W$ be as in Lemma~\ref{lem:Can.expshXtx}.  We thus define a smooth map $\exp_X:V\rightarrow \R^n$ by letting
\begin{align*}
 \exp_X(x) =\exp\left(x_1X_1+\cdots + x_nX_n\right)\!(0) = \left. \exp\left(s \hat{X}(1,x)\right)\!(0)\right|_{s=1}, \quad x\in V. 
\end{align*}

 \begin{proposition}\label{prop:can.1stkind-approx}
 Suppose that $(x_1,\ldots, x_n)$ are privileged coordinates at $a$ adapted to $(X_1,\ldots, X_n)$. Then, near $x=0$, we have
\begin{equation*}
 \exp_X(x)= \exp\left(x_1X_1^{(a)}+\cdots + x_nX_n^{(a)}\right)\!(0)+\Ow\left(\|x\|^{w+1}\right).
\end{equation*}
\end{proposition}
\begin{proof}
 Let $x\in V$ and $t\in (-1,1)$, $t\neq 0$. In the same way as in~(\ref{eq:Can.shXtx})--(\ref{eq:Can.expshXtx}) we have
\begin{equation}
 t^{-1}\cdot \exp_X(t\cdot x)= t^{-1}\cdot \exp\left( \hat{X}(1,t\cdot x)\right)\!(t\cdot 0)= \exp\left( \hat{X}(t,x)\right)\!(0). 
 \label{eq:Can.rescaling-expX}
\end{equation}
 Note that $ \exp( \hat{X}(0,x))(0)=\exp(x_1X_1^{(a)}+\cdots +x_n X_n^{(a)})(0)= \exp_{X^{(a)}}(x)$. Thus, the $C^\infty$-regularity of $t \rightarrow  \exp( \hat{X}(t,x))(0)$ on $[-1,1]\times V$ implies that, as $t\rightarrow 0$, we have 
 \begin{equation*}
 t^{-1}\cdot \exp_X(t\cdot x)=\exp_{X^{(a)}}(x) +\op{O}(t) \qquad \text{in $C^\infty(V)$}. 
\end{equation*}
Using Lemma~\ref{lem-eq-we} we then deduce that $\exp_X(x)=\exp_{X^{(a)}}(x)+\Ow(\|x\|^{w+1})$ near $x=0$. 
The proof is complete.
\end{proof}

Using Lemma~\ref{lem:Can.expshXtx} and pulling back the flow $\exp(s\hat{X}(1,x))(y)$ to $M$ shows there is a 
neighborhood $W_0$ of $a$ in $M$ such that the flow $\exp(s(x_1X_1+\cdots +x_nX_n))(y)$ is defined for all $s\in [-1,1]$ and $(x,y)\in V\times W_0$ and depends smoothly on these parameters. In particular, we have a smooth map $V\ni x\rightarrow \exp_X(x;a)\in M$ given by
\begin{equation*}
\exp_{X}(x;a)=\exp(x_1X_1+\cdots +x_nX_n)(a), \qquad x\in V. 
\end{equation*}
For $j=1,\ldots, n$, we have $\partial_{x_j} \exp_{X}(0;a)= X_j(a)$. As $(X_1(a),\ldots, X_n(a))$ is a basis of $TM(a)$, we deduce that $\exp_X'(0;a)$ is non-singular. Therefore, possibly by shrinking $V$ we may assume that $x\rightarrow \exp_X(x;a)$ is a diffeomorphism from $V$ onto an open neighborhood of $a$ in $M$. Its inverse map then is a local chart around $a$. The local coordinates defined by this chart are the so-called canonical coordinates of the first kind (\emph{cf}.~\cite{Go:LNM76, RS:ActaMath76}).  

\begin{proposition}[\cite{Go:LNM76, RS:ActaMath76}; see also~\cite{Je:Brief14}] \label{prop:can.1stkind-privileged}
 The canonical coordinates of the 1st kind above are privileged coordinates at $a$ adapted to $(X_1,\ldots, X_n)$. 
\end{proposition}
\begin{proof}
 These local coordinates are defined by the local chart  $\kappa_a$ that inverts the map $V\ni x\rightarrow \exp_X(x;a)$. Let $\kappa$ be a local chart that gives rise to privileged coordinates at $a$ adapted to $(X_1, \ldots, X_n)$. Without any loss of generality we may assume that $\kappa_a$ and $\kappa$ have the same domain. Set $\phi=\kappa_a\circ \kappa^{-1}$. If we denote by $(x_1, \ldots, x_n)$ the local coordinates defined by $\kappa$, then we pass from these coordinates to the canonical coordinates of the 1st kind by means of the change of coordinates $x\rightarrow \phi(x)$. Therefore, by Proposition~\ref{prop:char-priv.phi-hatphi-Ow} in order to show that the canonical coordinates of the 1st kind are privileged coordinates we only have to check that $\phi(x)$ has a behavior of the form~(\ref{eq:char-priv-coord.phi-x-Ow}) near $x=0$. 
 
 We observe that, for all $x\in V$, we have 
\begin{align*}
 \phi^{-1}(x) & = \kappa\left[ \exp(x_1X_1+\cdots +x_nX_n)(a)\right] \\
 & = \kappa \circ  \exp\left( x_1\kappa_*X_1+\cdots + x_n\kappa_*X_n\right)\circ \kappa^{-1}(0)\\
 & = \exp\left( x_1\kappa_*X_1+\cdots + x_n\kappa_*X_n\right)(0). 
\end{align*}
 In other word, $\phi^{-1}(x)$ is the map $\exp_X$ in the privileged coordinates defined by $\kappa$. Therefore, by Proposition~\ref{prop:can.1stkind-approx}, near $x=0$, we have 
\begin{equation}
 \phi^{-1}(x)=\exp_{X^{(a)}}(x) +\Ow\left(\|x\|^{w+1}\right). 
  \label{eq:Can.phi-1-expX}
\end{equation}
As $\exp_{X^{(a)}}(x)$ is a $w$-homogeneous map, Proposition~\ref{lem:Carnot-coord.inverse-Ow} implies that $\phi(x)$ has a behavior of the form~(\ref{eq:char-priv-coord.phi-x-Ow}) near $x=0$. As mentioned above this shows that the canonical coordinates of the 1st kind are privileged coordinates.  
The proof is complete.
\end{proof}

\subsection{Canonical coordinates of the 2nd kind} Let us now turn to canonical coordinates of the 2nd kind. Suppose that $(x_1,\ldots, x_n)$ are privileged coordinates at $a$ adapted to $(X_1,\ldots, X_n)$. 

\begin{lemma}\label{lem:Can.sxjhXj}
There exist neighborhoods $V$ and $W_k$, $k=1,\dots, n$, of the origin $0\in \R^n$ with $W_1\subset \cdots \subset W_n\subset U$ such that
 \begin{enumerate}
\item[(i)] For every $j=1,\ldots, n$, the flow $\exp(sx_j \hat{X}_j(t))(y)$ exists for all $s,t\in [-1,1]$ and $(x,y)$ in $V\times W_n$ and depends smoothly on these parameters. 

\item[(ii)]  For every $j=1,\ldots, n$ and $k=1,\ldots, n-1$, we have
\begin{equation*}
\exp\left( sx_j\hat{X}_j(t)\right)\!\!(W_k) \subset W_{k+1} \qquad \text{for all $s,t\in [-1,1]$ and $x\in V$}.  
\end{equation*}
\end{enumerate}
\end{lemma}
\begin{proof}
Let $V$ and $W$ be as in Lemma~\ref{lem:Can.expshXtx}. Given any $\epsilon>0$, we denote by $B_\epsilon(0)$ the ball of radius $\epsilon$ about the origin $0\in \R^n$. Let $\rho>0$ and $\delta>0$ be such that $\overline{B_\rho(0))}\subset V$ and  $\overline{B_\delta(0))}\subset W$. Note that $ \exp( s\hat{X}(t,0))(y)=\exp(0)(y)=y$. Therefore, the $C^1$-regularity near $x=0$ of the flow $\exp( s\hat{X}(t,0))(y)$ implies there is constant $C>0$ such that, for all $s,t\in [-1,1]$ and $(x,y)$ in $B_\rho(0)) \times B_\delta(0)$, we have
\begin{equation}
  \left| \exp\left( s\hat{X}(t,x)\right)\!\!(y)-y\right| \leq C|x|. 
  \label{eq:Can.expxhXtxy-y}
\end{equation}

Set $\rho'=\min\{\rho, (nC)^{-1}\delta\}$, and let $s,t\in [-1,1]$ and $x\in B_{\rho'}(0)$.  Given any $\delta'\in(0,\delta)$ and $y\in B_{\delta'}(0)$ the estimate~(\ref{eq:Can.expxhXtxy-y}) implies that $| \exp( s\hat{X}(t,x))(y)|\leq |y|+C|x|\leq \delta' +n^{-1} \delta$. Thus, if for $k=1,\ldots, n$ we set $W_k=B_{\frac{k}{n}\delta}(0)$, then, for $k=1,\ldots, n-1$, we have 
\begin{equation}
 \exp\left( s\hat{X}(t,x)\right)\!\!(W_k) \subset W_{k+1} \qquad \text{for all $s,t\in [-1,1]$ and $x\in V$}.
 \label{eq:Can.expxhXtxWk}
\end{equation}

Let $(\epsilon_1,\ldots, \epsilon_n)$ be the canonical basis of $\R^n$, and set $V'=(-\rho',\rho')^n$. Given any $x\in V'$, for $j=1,\ldots, n$, 
the point $x_j\epsilon_j$ is in $V'$ and $\hat{X}(t,x_j\epsilon_j)=x_j\hat{X}_j(t)$. Therefore, the flow  $\exp(sx_j \hat{X}_j(t))(y)$ exists for all $s,t\in [-1,1]$ and $(x,y)\in V'\times W_n$ and depends smoothly on these parameters. Moreover, for $k=1,\ldots, n-1$, it follows from~(\ref{eq:Can.expxhXtxWk}) that 
$\exp( s\hat{X}(t,x))(W_k) \subset W_{k+1}$ for all $s,t\in [-1,1]$ and $x\in V$. This proves the lemma.
\end{proof}

Let $V$ and $W_k$, $k=1,\ldots, n$, be neighborhoods of the origin $0\in \R^n$ as in Lemma~\ref{lem:Can.sxjhXj}. Then, for $j=1,\ldots, n$, the flow $\exp(sx_j \hat{X}_j(t))(y)$ exists for 
$s,t\in [-1,1]$ and $(x,y)\in V\times W_{n-j+1}$ and gives rise to a smooth map from $[-1,1]^2\times V\times W_{n-j+1}$ to $W_{n-j+2}$ (with the convention that $W_{n+1}=U$). Therefore, the composition of flows $\exp(sx_1\hat{X}(t))\circ \cdots \circ \exp(sx_n\hat{X}(t))(y)$ is well defined for all $s,t\in [-1,1]$ and $(x,y)\in V\times W_1$ and depends smoothly on these parameters. In particular, this allows us to define a smooth map $\gamma_X:V\rightarrow \R^n$ by letting
\begin{align*}
  \gamma_X(x) & = \exp\left( x_1X_1\right) \circ \cdots \circ  \exp\left(x_nX_n\right)\!(0)\\
  &=  \left.\exp\left( x_1\hat{X}_1(1)\right) \circ \cdots \circ  \exp\left(x_n\hat{X}_n(1)\right)\!(0)\right|_{s=1},   \qquad x\in V. 
\end{align*}
 
\begin{proposition}\label{prop:can.2ndkind-approx}
 Suppose that $(x_1,\ldots,x_n)$ are privileged coordinates at $a$ adapted to $(X_1,\ldots, X_n)$. Then, near $x=0$, we have 
\begin{equation*}
 \gamma_X(x)= \exp\left( x_1X_1^{(a)}\right) \circ \cdots \circ  \exp\left(x_nX_n^{(a)}\right)\!(0) +\Ow\left(\|x\|^{w+1}\right). 
\end{equation*}
\end{proposition}
\begin{proof}
 Given any $x\in V$ and $t\in [-1,1]$, $ t \neq 0$, in the same way as in~(\ref{eq:Can.rescaling-expX}) we have
 \begin{align}
 t^{-1}\cdot \gamma_X (t\cdot x) & = \left( \delta_t^{-1} \circ \exp\left(t^{w_1}x_1X_1\right) \circ \delta_t\right)\circ \cdots \circ  
 \left( \delta_t^{-1} \circ \exp\left(t^{w_n}x_n X_n\right)\circ \delta_t\right)\!(0) \nonumber \\
 & =  \exp\left(x_1 t^{w_1}\delta_t^*X_1\right) \circ \cdots \circ    \exp\left(x_n t^{w_n}\delta_t^*X_n\right)\!(0) 
 \label{eq:Can.rescaling-gammaX}\\ 
 & =  \exp\left(x_1 \hat{X}_1(t)\right) \circ \cdots \circ \exp\left(x_n \hat{X}_n(t)\right)\!(0). \nonumber
\end{align}

As it follows from the discussion right after Lemma~\ref{lem:Can.sxjhXj}, for $j=1,\ldots, n$ we have a smooth map $[-1,1]\times V\times W_{n-j+1}\ni (t,x,y)\rightarrow \exp(x_j\hat{X}_j(t))(y)\in W_{n-j+2}$. In addition, we have $\hat{X}_j(0)=X_j^{(a)}$. Therefore, the $C^\infty$-regularity near $t=0$ of the above map implies that, as $t\rightarrow 0$, we have 
\begin{equation*}
 \exp\left(x_j \hat{X}_j(t)\right)(y)= \exp\left(x_j X_j^{(a)}\right)(y) +\op{O}(t) \qquad \text{in $C^\infty(V\times W_{n-j+1}, W_{n-j+2})$}. 
\end{equation*}
Combining this with~(\ref{eq:Can.rescaling-gammaX}) we deduce that, as $t\rightarrow 0$ and in $C^\infty(V)$, we have 
 \begin{equation*}
 t^{-1}\cdot \gamma_X (t\cdot x) = \exp\left( x_1X_1^{(a)}\right) \circ \cdots \circ  \exp\left(x_nX_n^{(a)}\right)\!(0) + \op{O}(t).  
\end{equation*}
Using Lemma~\ref{lem-eq-we} then shows that $ \gamma_X(x)= \exp( x_1X_1^{(a)}) \circ \cdots \circ  \exp(x_nX_n^{(a)})(0) +\Ow(\|x\|^{w+1})$ near $x=0$, completing the proof. 
\end{proof}

Pulling back the flows $\exp(sx_j X_j)$ to the manifold $M$ enables us to define a smooth map $V\ni x \rightarrow \gamma_X(x;a)$ by
\begin{equation*}
 \gamma_X(x;a)=  \exp\left( x_1X_1\right) \circ \cdots \circ  \exp\left(x_nX_n\right)\!(a), \qquad x \in V. 
\end{equation*}
Furthermore, for $j=1,\ldots, n$ we have $\partial_{x_j} \gamma_X(0)= \partial_{x_j} \exp(x_jX_j)(a)= X_j(a)$. 
Therefore, in the same way as with the map $\exp_X$, possibly by shrinking $V$ we may assume that $\gamma_X$ is a diffeomorphism from $V$ onto an open neighborhood of $a$. 
The inverse of this map is a local chart around $a$. 
The local coordinates defined by this chart are called \emph{canonical coordinates of the second kind}~(\emph{cf}.~\cite{BS:SIAMJCO90, He:SIAMR}). 

By using Proposition~\ref{prop:can.2ndkind-approx} and arguing as in the proof of Proposition~\ref{prop:can.1stkind-privileged} we recover the following result. 

\begin{proposition}[\cite{BS:SIAMJCO90, He:SIAMR}; see also~\cite{Je:Brief14, Mo:AMS02}] \label{prop:can.2ndkind-privileged}
 The canonical coordinates of the 2nd kind above are privileged coordinates at $a$ adapted to $(X_1,\ldots, X_n)$. 
\end{proposition}


\begin{thebibliography}{99}
\bibitem{ABB:SRG} Agrachev, A., Barilari, D., Boscain, U.: \emph{Introduction to Riemannian and sub-Riemannian
geometry}. To appear, \texttt{http://webusers.imj-prg.fr/$\sim$davide.barilari/Notes.php}. 

\bibitem{AM:ERA03} Agrachev, A.; Marigo, A.: \emph{Nonholonomic tangent spaces: intrinsic construction and rigid 
dimensions}. Electron.\ Res.\ Announc.\ Amer.\ Math.\ Soc.\ \textbf{9} (2003), 111--120. 

\bibitem{AM:JDCS05} Agrachev, A.; Marigo, A.:  \emph{Rigid Carnot algebras: a classification}. J.\ Dyn.\ Control Syst.\ 
\textbf{11} (2005), 449--494. 

\bibitem{AS:DANSSSR87} Agrachev, A.A., Sarychev, A.V.: \emph{Filtrations of a Lie algebra of vector fields and nilpotent
approximations of control systems.} Dokl.\ Akad.\ Nauk SSSR \textbf{285} (1987), 777--781 . (English
transl.: Soviet Math.\ Dokl.\ \textbf{36} (1988), 104--108.)    

\bibitem{ACGL:Contact17} Apostolov, V.; Calderbank, D.M.J.; Gauduchon, P.; Legendre, E.:
 \emph{Toric contact geometry in arbitrary codimension}. Preprint,  \texttt{arXiv:1708.04942}, 22 pages. To appear in Int.\ Math.\ Res.\ Notices. 
   
\bibitem{ACGL:CR17} Apostolov, V.; Calderbank, D.M.J.; Gauduchon, P.; Legendre, E.:
 \emph{Levi-K\"ahler reduction of CR structures, products of spheres, and toric geometry}. Preprint, \texttt{arXiv:1708.05253}, 39 pages.  
          
\bibitem{Ba:JGeom00} Banyaga, A.: \emph{On essential conformal groups and a conformal invariant}, J.\ Geom.\ 
\textbf{68} (2000), 10--15.
    
\bibitem{BG:CHM} Beals, R.; Greiner,  P.: \emph{Calculus on Heisenberg manifolds}. Annals of Mathematics Studies, 119,  Princeton University Press, Princeton, NJ, 1988.

\bibitem{Be:Tangent} Bella\"\i che, A.: \emph{The tangent space in sub-Riemannian geometry}. Sub-Riemannian 
geometry, pp.~1--78, Progr.\ Math.\ 144, Birkh\"auser, Basel, 1996. 

 \bibitem{BS:SIAMJCO90} Bianchini, R.M.; Stefani, C.: \emph{Graded approximations and controllability along a trajectory},
SIAM J.\ Control Optim.\ \textbf{28} (1990), 903--924.

\bibitem{Bi:Roma99} Biquard, O.: \emph{Quaternionic contact structures}. Quaternionic structures in mathematics and
physics, Univ.\ Studi Roma La Sapienza, Rome, 1999.

\bibitem{Bi:Asterisque} Biquard, O.: \emph{M\'etriques d'Einstein asymptotiquement sym\'etriques}.  Ast\'erisque \textbf{265} (2000), 115 pages.

\bibitem{Bl:Control} Bloch, A.M.: \emph{Nonholonomic mechanics and control}. Interdisciplinary Applied 
Mathematics, Vol.\ 24. Springer, New York, 2003.

\bibitem{BLU:Springer07} Bonfiglioli, A.; Lanconelli, E.; Uguzzoni, F.: \emph{Stratified Lie groups and potential 
theory for their sub-Laplacians}. Springer Monographs in Mathematics. Springer, Berlin (2007). 

\bibitem{Br:Brief14} Bramanti, M.: \emph{An Invitation to hypoelliptic operators and H\"ormander's vector fields}. 
SpringerBriefs in Mathematics. Springer International Publishing, New York, 2014.

\bibitem{Br:RIMS06} Bryant, R.: \emph{Conformal geometry and 3-plane fields on 6-manifolds}. 
\emph{Developments of Cartan Geometry and Related Mathematical Problems}, RIMS Symposium
Proceedings, Vol.~1502, June 2006, pp.~1--15.

\bibitem{CC:SRG} Calin, O., Chang, D.-C.: \emph{Sub-Riemannian geometry. General theory and examples}. Encyclopedia
of Mathematics and its Applications, vol.\ 126. Cambridge University Press, Cambridge, 2009. 

\bibitem{Ca:AENS10} Cartan, E.: \emph{Les syst\`emes de Pfaff \`a cinq variables et les \'equations aux deriv\'ees partielles
du second ordre}. Ann.\ Sci.\ \'Ecole Norm.\ Sup.\ \textbf{27} (1910), 263--355.

\bibitem{CS:AMS09} \v{C}ap, A.; Slov\'ak, J.: \emph{Parabolic Geometries I: Background and General Theory.} Mathematical
Surveys and Monographs vol. 154, American Mathematical Society, (2009).

\bibitem{CP:AMP17} Choi, W.; Ponge, R.: \emph{Privileged coordinates and nilpotent approximation for Carnot manifolds, II. Carnot coordinates}. E-print, 
\texttt{arXiv:1703.05494v2} (v2: September 2017), 36 pages.  

\bibitem{CP:Groupoid} Choi, W.; Ponge, R.: \emph{Tangent maps and tangent groupoid for Carnot manifolds}. E-print, 
\texttt{arXiv:1510.05851v2} (v2: September 2017), 40 pages. 

\bibitem{CP:Carnot-calculus} Choi, W.; Ponge, R.: \emph{A Pseudodifferential calculus on Carnot Manifolds.} In preparation. 

\bibitem{Co:NCG} Connes, A.:  \emph{Noncommutative geometry}. Academic Press, Inc., San Diego, CA, 1994.

 \bibitem{CM:GAFA95}  Connes, A., Moscovici, H.: \emph{The local index formula in
 noncommutative geometry}. Geom.\ Funct.\ Anal.\ \textbf{5} (1995), 174--243. 

 \bibitem{CG:Cambridge90} Corwin, L.; Greenleaf, F.: \emph{Representations of nilpotent Lie groups and their applications. Part I. Basic theory and examples}. 
 Cambridge Studies in Advanced Mathematics, 18. Cambridge University Press, Cambridge, 1990. 

 \bibitem{Cu:CPDE89} Cummins, T.: \emph{A pseudodifferential calculus associated with 3-step nilpotent groups}. 
 Comm.\ Partial Differential Equations \textbf{14} (1989), no.~1, 129--171.

\bibitem{ET:C} Eliashberg, Y.; Thurston, W.: \emph{Confoliations}. University Lecture Series, 13, AMS, Providence, RI, 1998.

\bibitem{FJ:JAM03} Falbel, E.; Jean, F.: \emph{Measures of transverse paths in sub-Riemannian geometry}. 
J.\ Anal.\ Math.\ \textbf{91} (2003), 231--246.

\bibitem{FR:Birkhauser16} Fischer, V.; Ruzhansky, M.: \emph{Quantization on nilpotent Lie groups}. Progress in Mathematics, 314. Birkh\"auser/Springer,  2016. 

\bibitem{Fo:SM79} Folland, G. B.:  \emph{Lipschitz classes and Poisson integrals on stratified groups}. Studia Math.\ \textbf{66} (1979), 37--55.

\bibitem{FS:CMPAM74} Folland, G.; Stein, E.: \emph{Estimates for the $\overline{\partial}_b$-complex and analysis on the Heisenberg group}. 
Comm.\ Pure Appl.\ Math.\ \textbf{27} (1974) 429--522.

\bibitem{FS:Hardy82} Folland, G.; Stein, E.: \emph{Hardy spaces on homogeneous groups}. 
Mathematical Notes, 28, Princeton University Press, Princeton, NJ, 1982.
 
\bibitem{Fo:IUMJ05} Fox, D.J.F.:   \emph{Contact projective structures}. Indiana Univ.\ Math.\ J.\ \textbf{54} (2005) 1547--1598.

\bibitem{Fo:arXiv05} Fox, D.J.F.: \emph{Contact path geometries}. E-print, \texttt{arXiv:math.DG/0508343}, 36 
pages. 

\bibitem{GV:JGP88} Gershkovich, V.; Vershik, A.: \emph{Nonholonomic manifolds and nilpotent analysis}. J.\ Geom.\ 
Phys.\ \textbf{5} (1988), 407--452. 

\bibitem{Go:LNM76} Goodman, N.: \emph{Nilpotent Lie groups: structure and applications to analysis.} 
Lecture Notes in Mathematics, vol.\ {562}. Springer-Verlag, Berlin-New York, 1976.

\bibitem{Gr:CC} Gromov, M.: \emph{Carnot-Carath\'eodory spaces seen from within}. Sub-Riemannian 
geometry, pp.~85--323, Progr.\ Math.\ 144, Birkh\"auser, Basel, 1996. 

\bibitem{He:SIAMR} Hermes, H.: \emph{Nilpotent and high-order approximations of vector field systems}. SIAM Rev.\ \textbf{33} (1991), 238--264.
 
\bibitem{Ho:ActaMath67} H\"ormander, L.:  \emph{Hypoelliptic second order differential equations}. 
 Acta Math.\ \textbf{119} (1967), 147--171. 

\bibitem{Je:ESAIM96} Jean, F.: \emph{The car with $N$ trailers: characterization of the singular configurations}.  
ESAIM: Cont.\ Opt.\ Calc.\ Var.\ \textbf{1} (1996) 241--266.    

\bibitem{Je:Brief14} Jean, F.:  \emph{Control of nonholonomic systems: from sub-Riemannian geometry to motion planning}. Springer Briefs in 
Mathematics. Springer International Publishing, New York, 2014.

 \bibitem{Ka:TAMS80} Kaplan, A.: \emph{Fundamental solutions for a class of hypoelliptic PDE generated by composition of quadratic forms}.  
Trans.\ Amer.\ Math.\ Soc.\ \textbf{258} (1980), 147--153. 

\bibitem{MM:JAM00} Margulis, G.A; Mostow,G.D.: \emph{Some  remarks on the definition of  tangent  cones in  a 
Carnot-Carath\'eodory space}. J.\ Anal.\ Math.\ \textbf{80} (2000), 299--317. 

\bibitem{Me:Preprint82} Melin, A.: \emph{Lie filtrations and pseudo-differential operators}. Preprint, 1982. 

\bibitem{Me:CPDE76} M\'etivier, G.: \emph{function spectrale et valeurs propres d'une classe d'op\'erateurs non elliptiques}.
Comm.\ Partial Differential Equations \textbf{47} (1976), 467--510.

\bibitem{Me:Duke80} M\'etivier, G.: \emph{Hypoellipticit\'e analytique sur des groupes nilpotents de rang 2}.
Duke Math.\ J.\ \textbf{47} (1980), 195--213.

\bibitem{Mi:JDG85} Mitchell, J.: \emph{On Carnot-Carath\'eodory metrics}. J.\ Differential Geom.\ \textbf{21} 
(1985), 35--45. 

 \bibitem{MM:Cambridge03} Moerdijk, I.; Mr\v{c}un, J.:
  \emph{Introduction to foliations and Lie groupoids}. Cambridge Studies in Advanced Mathematics, 91. 
  Cambridge University Press, Cambridge, 2003.

\bibitem{Mo:AMS02} Montgomery, R.: \emph{A tour of subriemannian geometries, their geodesics and applications.} 
Mathematical Surveys and Monographs, vol.\ 91. American Mathematical Society, Providence, RI (2002). 

\bibitem{NSW:ActaM85} Nagel, A.; Stein, E.M.; Wainger, S.: \emph{Metrics defined by vector fields I: Basic Properties}. 
Acta Math.\ \textbf{155} (1985), 103--147.

\bibitem{Pa:AM89} Pansu, P.: \emph{M\'etriques de Carnot-Carath\'eodory et quasiisom\'etries des espaces 
sym\'etriques de rang un}. Ann.\ Math.\ (2), \textbf{129} (1989), 1--60.

\bibitem{Pe:Springer01} Perko, L.: \emph{Differential Equations and Dynamical Systems}. Springer, New York (2001). 

\bibitem{Po:PJM06} Ponge, R.:  \emph{The tangent groupoid of a Heisenberg manifold}. Pacific J.\ Math.\ \textbf{227} 
(2006), no. 1, 151--175.

\bibitem{Ri:Brief14} Rifford, L.: \emph{Sub-Riemannian geometry and optimal transport}. Springer Briefs in 
Mathematics. Springer International Publishing, New York, 2014. 

\bibitem{Ro:INA} Rockland, C.: \emph{Intrinsic nilpotent approximation.} Acta Appl. Math. \textbf{8} (1987), no. 3, 213--270.

\bibitem{RS:ActaMath76} Rothschild, L.; Stein, E.;  \emph{Hypoelliptic differential operators and nilpotent groups.} Acta Math.\ 
\textbf{137} (1976), no. 3-4, 247--320.

\bibitem{St:1988} Stefani, G.: \emph{On local controllability of the scalar input control systems}. \emph{Analysis and Control
of Nonlinear Systems},  North-Holland, Amsterdam, 1988, pp.~213--220.
 
\bibitem{Ta:JMKU70} Tanaka, N.:  \emph{On differential systems, graded Lie algebras and pseudogroups}. J.\ Math.\ 
Kyoto Univ.\ \textbf{10} (1970), 1--82.
 
 \bibitem{vE:Polycontact} van Erp, E.: \emph{Contact structures of arbitrary codimension and idempotents in the Heisenberg 
algebra}. E-print, \texttt{arXiv:1001.5426}, 13 pages. 

\bibitem{Ve:BSMF70} Vergne, M.: \emph{Cohomologie des alg\`ebres de Lie nilpotentes. Application \`a l'\'etude de la 
vari'\'et\' des alg\`ebres de Lie nilpotentes}. Bull.\ Soc.\ Math.\ France \textbf{98} (1970), 81--116.

\bibitem{VG:JSM92} Vershik, A.; Gershkovich, V.:  \emph{A bundle of nilpotent Lie algebras over a nonholonomic manifold}. 
J.\ Soviet Math.\ \textbf{59} (1992), 1040--1053. 

\bibitem{We:Adv80} Weinstein, A.: \emph{Fat bundles and symplectic manifolds}. Adv.\ Math.\ \textbf{37}  (1980), 239--250. 

\end{thebibliography}
\end{document}